\documentclass[reqno]{amsart}

\usepackage{amsmath}
\usepackage{amsthm}
\usepackage{amsfonts}
\usepackage{amssymb}
\usepackage{url}
\usepackage{enumerate}
\usepackage[pdftex,bookmarks=true]{hyperref}
\usepackage{comment}
\usepackage{xcolor}
\usepackage{bm}
\usepackage{hyperref}
\usepackage{url}
\newcommand\rank{\operatorname{rank}}

\newcommand\R{{\mathbb{R}}}
\newcommand\C{{\mathbb{C}}}
\newcommand\Z{{\mathbb{Z}}}

\renewcommand\P{{\mathbb{P}}}
\newcommand\E{{\mathbb{E}}}
\newcommand\F{{\mathbb{F}}}

\newcommand\Var{\operatorname{Var}}
\newcommand\eps{{\varepsilon}}

\newcommand\dist{\operatorname{dist}}

\newcommand\Bx{{\mathbf x}}

\newcommand\BX{{\mathbf X}}

\newcommand\BMa{{\bm a}}
\newcommand\BMb{{\bm b}}
\newcommand\BMe{{\bm e}}
\newcommand\BMh{{\bm h}}
\newcommand\BML{{\bm L}}
\newcommand\BMr{{\bm r}}
\newcommand\BMS{{\bm S}}
\newcommand\BMt{{\bm t}}
\newcommand\BMu{{\bm u}}
\newcommand\BMv{{\bm v}}
\newcommand\BMw{{\bm w}}
\newcommand\BMx{{\bm x}}
\newcommand\BMX{{\bm X}}

\newcommand\CE{{\mathcal E}}

\newcommand\CG{{\mathcal G}}

\newcommand\CN{{\mathcal N}}

\newcommand\CR{{\mathcal R}}

\newcommand\BBN{{\mathbb N}}

\newcommand\LCD{\mathbf{LCD}}

\newcommand\al{\alpha}

\def\perm{\operatorname{perm}}

\def\Vol{\operatorname{Vol}}

\theoremstyle{plain}
 \newtheorem{theorem}{Theorem}[section]

 \newtheorem{problem}[theorem]{Problem}

 \newtheorem{proposition}[theorem]{Proposition}
 \newtheorem{fact}[theorem]{Fact}
 \newtheorem{lemma}[theorem]{Lemma}
 \newtheorem{corollary}[theorem]{Corollary}
 \newtheorem{claim}[theorem]{Claim}
 \newtheorem{condition}[theorem]{Condition}

\newtheorem{remark}[theorem]{Remark}

\newtheorem{example}[theorem]{Example}

\theoremstyle{definition}
\newtheorem{definition}[theorem]{Definition}

\begin{document}
\include{psfig}

\title{The anti-concentration phenomenon with respect to random permutations}

\author{Viet H. Do}
\address{Department of Mathematics, Yale University, New Haven, CT 06511}
\email{viet.do@yale.edu}

\author{Hoi H. Nguyen}
\address{Department of Mathematics, The Ohio State University, Columbus OH 43210}
\email{nguyen.1261@math.osu.edu}
\thanks{H. Nguyen is supported by a Simons Travel Grant TSM-00013318. T. Tran is supported by the Excellent Young Talents Program
(Overseas) of the National Natural Science Foundation of China GG0010007003. A substantial part of this work was carried out during the ADM program 2024 supported by VinIF foundation; the authors are grateful to the VinUni BigData Research Insitute  for its hospitality.}

\author{Kiet  H. Phan}
\address{School of Mathematics and Statistics, UNSW Sydney, Sydney NSW 2052, Australia}
\email{kevin.phan4@student.unsw.edu.au}

\author{Tuan Tran}
\address{School of Mathematical Sciences, University of Science and Technology of China, Hefei, Anhui 230026, China}
\email{trantuan@ustc.edu.cn}
\thanks{}

\author{Van H. Vu}
\address{Department of Mathematics, Hong Kong University, Hong Kong}
\email{vanvu@hku.hk}

 \begin{abstract}   
 
 The anti-concentration phenomenon in probability theory has been intensively
studied in recent years, with applications across many areas of mathematics. In
most existing works, the ambient probability space is a product space generated
by independent random variables.

In this paper, we initiate a systematic study of anti-concentration when the
ambient space is the symmetric group, equipped with the uniform measure.
Concretely, we focus on the random sum
\(
S_{\pi} = \sum_{i=1}^{n} w_i\, v_{\pi(i)},
\)
where $\BMw=(w_1,\dots,w_n)$ and $\BMv=(v_1,\dots,v_n)$ are fixed
vectors and $\pi$ is a uniformly random permutation.

\vskip2mm

The paper contains several new results, addressing both discrete and continuous
anti-concentration phenomena. On the discrete side, we establish a near-optimal
structural characterization of the vectors $\BMw$ and $\BMv$ under
the assumption that the concentration probability
\(
\sup_x \P(S_{\pi}=x)
\)
is polynomially large.  This is an ``inverse theorem" and can be seen as an analog of earlier results 
by Tao--Vu and Nguyen--Vu in the product space setting. 
In fact, our technique allows to obtain an intervese theorem for random subsums of entries of a matrix, 
which contains the result concerning random permutation as the special case where the matrix has rank one. 

\vskip2mm

On the continuous side, we study the small-ball event $|S_{\pi}-L|\le \delta$. Our results exhibit
sub-gaussian decay in $L$,  answering an open question of S{\"o}ze~\cite{Soze2}. With additional effort, we are also able to treat the joint distribution of these events. We also extend the inverse theorem obtained in the discrete case to the continuous case.

\vskip2mm

 Our results have applications in various areas. 
First, we use our inverse theorems to derive and strengthen a number of
previous anti-concentration bounds. In particular, we show that if both $\BMw$ and
$\BMv$ have distinct entries, then
\(
\sup_x \P(S_{\pi}=x) \le n^{-5/2+o(1)}.
\)
This bound serves as an analog of the classical
Erd\H{o}s--Moser bound in the product-space setting and answers a question posed
by Alon--Pohoata--Zhu~\cite{APZ}. Next, we apply our new results to study random polynomials, and  prove that the number of extremal points of random permutation polynomials is bounded by $O(\log n)$, extending results of S{\"o}ze~\cite{Soze1, Soze2} on the number of real roots. In the final application, we prove that random matrices whose rows are independent random permutations of a fixed non-degenerate vector are nonsingular with high probability, strengthening  earlier  results of  Nguyen and Nguyen--Vu  \cite{Nguyen,NgV-Comb}. 

  \end{abstract}

\maketitle


\section{Introduction}

\subsection{Anti-concentration in product spaces} \label{sect:intro}
Let $\BMw=(w_1,\dots,w_n)$ be a real vector. Consider the random sum
\begin{equation}\label{eqn:S}
S = \sum_{i=1}^{n} w_i \xi_i,
\end{equation}
where the $\xi_i$ are i.i.d.\ copies of a real-valued random variable $\xi$ with
mean zero and variance one. This sum can be viewed as the inner product of
$\BMw$ and the random vector $(\xi_1,\dots,\xi_n)$.

A typical anti-concentration bound asserts that, under suitable assumptions, the
probability that $S$ lies in a small interval is small. We consider two
settings: \emph{discrete} and \emph{continuous}. In the discrete setting, we
study the probability of the event $S=x$ for a fixed value $x$. In the
continuous setting, we consider small-ball events of the form
$|S-L|\le \delta$ for some $L\in\R$ and $\delta>0$.

For ease of exposition, we begin with the discrete setting. In the 1940s,
Littlewood--Offord~\cite{LO} and Erd\H{o}s~\cite{Erdos} proved the following
fundamental result.

\begin{theorem}\label{thm:ELO}
Assume that $w_i$, $1\le i\le n$, are nonzero and that the $\xi_i$ are i.i.d.\
Rademacher random variables (that is, $\xi_i$ takes values $\pm1$ with
probability $1/2$, independently). Then
\[
\sup_x \P(S=x) = O\Big(\frac{1}{\sqrt{n}}\Big).
\]
\end{theorem}

\noindent {\it Notation.}\label{notation-sec}
Here and throughout, asymptotic notation is taken in the limit $n\to\infty$.
For real quantities $X,Y$, we write $X = O(Y)$, or equivalently $Y = \Omega(X)$, if $|X| \le C |Y|$ for some constant $C>0$ independent of $n$.
This constant $C$ may depend on other fixed quantities. If $X=O(Y)$ and $Y=O(X)$, we write $X=\Theta(Y)$. We also write $X=o(Y)$ if
$|X|\le c(n)|Y|$ for some function $c(n)\to 0$ as $n\to\infty$. Again, the function $c(\cdot)$ may depend on fixed quantities. We denote by $\|x\|_{\R/\Z}$ the distance from a real number $x$ to the nearest integer. 
For $x \in \R$, we set $e(x) := e^{2\pi i x}$. 
For a positive integer $n$, we write $[n] = \{1,2,\ldots,n\}$. For a vector $\BMv\in\R^n$ and a permutation $\pi$ of $[n]$, we write $\pi(\BMv)$ for the vector $(v_{\pi(1)},\dots,v_{\pi(n)})$.

\vskip2mm 

When the coefficients $w_i$ are distinct, this bound can be significantly
improved. This was shown independently by Erd\H{o}s--Moser~\cite{EM},
S\'ark\H{o}zy--Szemer\'edi~\cite{SS}, and Stanley~\cite{Stan}.

\begin{theorem}\label{thm:SSz}
Assume that the $w_i$ are distinct real numbers and that the $\xi_i$ are i.i.d.\
Rademacher random variables. Then
\[
\sup_x \P(S=x) = O(n^{-3/2}).
\]
\end{theorem}

\begin{remark}\label{remark:scaling}
In this paper, we focus on orders of magnitude, and all bounds are stated in
big-Oh form $O(\cdot)$. Bounds of this type remain unchanged (in both the
discrete and continuous settings) if we replace $w_i$ by $\alpha w_i$, where
$\alpha = O(1)$. This observation allows us to normalize many of our
assumptions.
\end{remark}

These results initiated a substantial body of work known as
\emph{Littlewood--Offord theory}, which has developed over several decades; see
\cite{NVsur} for a comprehensive survey. The guiding principle of this theory
is that stronger structural assumptions on the coefficients $w_i$ lead to
stronger anti-concentration bounds. Results of this type are commonly referred
to as \emph{forward} theorems.

In the early 2000s, Tao and the last author initiated a new line of research,
known as the \emph{inverse Littlewood--Offord theory}. The goal is to
characterize the additive structure of the coefficients $w_i$ under the
assumption that the concentration probability
\[
\sup_{x\in\R} \P(S=x)
\]
is \emph{relatively large}. Results of this kind are called \emph{inverse}
theorems. In this paper, we focus on the polynomial regime, where “relatively
large” means at least $n^{-C}$ for some constant $C>0$.

Let $W$ denotes the multi-set $\{w_1, \dots, w_n \}$, and define 
  \[ \rho(W) := \sup_{x \in \R} \P(S = x).\]
  Assume that  \( \rho(W)\ge n^{-C} \)  for some constant $C>0$. Then at least \( 2^n n^{-C} \) of the \( 2^n \) subset sums of $W$ coincide, suggesting
  that \( W \) must possess substantial additive structure.  Tao and the last author formalized this intuition quantitatively via
the notion of \emph{generalized arithmetic progressions} (GAPs).

\begin{definition}\label{def:GAP}
A subset \( P \subset \R \) is a \emph{generalized arithmetic progression (GAP) of rank \( r \)} if it can be expressed as
\[
P = \Big\{ g_0 + m_1 g_1 + \dots + m_r g_r \ \Big| \ m_i \in \Z,\ N_i \le m_i \le N_i' \Big\},
\]
where \( g_1, \dots, g_r \in \R \) are called the \emph{generators} of \( P \), and the integers \( N_i, N_i' \) are its \emph{dimensions}. The \emph{volume} of \( P \) is defined as \( \Vol(P) := \prod_{i=1}^r (N_i' - N_i + 1) \). We say that \( P \) is \emph{proper} if every element of \( P \) has a unique representation in the above form; this is equivalent to \( |P| = \Vol(P) \). If \( N_i = -N_i' \) for all \( i \) and \( g_0 = 0 \), we say that \( P \) is \emph{symmetric}.
\end{definition}

For two sets \( A, B \subset \R \), their (Minkowski) sum is defined by
\[
A + B := \{ a + b \mid a \in A,\ b \in B \}.
\]
For \( n \in \Z^+ \), we define
\[
nA := \{ a_1 + \dots + a_n \mid a_i \in A \}.
\]
For example, if \( P \) is as in Definition~\ref{def:GAP}, then
\[
nP = \Big\{ ng_0 + m_1 g_1 + \dots + m_r g_r \mid nN_i \le m_i \le nN_i' \Big\}.
\]

\begin{example}
Assume \( P \) is a proper symmetric GAP of rank \( r = O(1) \) and cardinality \( n^{O(1)} \), and all elements of \( W = \{w_1, \dots, w_n\} \) lie in \( P \). 
Then since \( |nP| \le n^r |P| \), by the pigeonhole principle, we obtain \( \rho(W) = \Omega(n^{-O(1)}) \).
\end{example}
This example shows that if \( W \) lies inside a proper symmetric GAP with small rank and cardinality, then \( \rho(W) \) is necessarily large. In a series of works, 
Tao--Vu \cite{TVstrong,TVinverse}, 
Nguyen--Vu \cite{NgV-Adv}, and more recently Tao \cite{Tao}, demonstrated that these are essentially the only configurations for which \( \rho(W) \) is polynomially large. One may also consider the sub-exponential (\( \rho \ge \exp(-n^c) \)) or exponential (\( \rho \ge \exp(-cn) \)) regimes, but we do not address them in this paper.

\begin{theorem}[Inverse Littlewood--Offord result for \( \rho \)]\cite[Theorem 2.1]{NgV-Adv}\label{thm:ILO}
Let \( C>0\) and \( \eps \in (0,1) \) be constants. Suppose \( W = \{w_1, \dots, w_n\} \) is a multiset of real numbers such that
\[
\rho(W) \ge n^{-C},
\]
where $\xi_{i}$ are iid copies of a random variable $\xi$ of mean zero, variance one, and bounded $(2+\eps)$-moment. Then for any \( n^\eps \le n' \le n \), there exists a proper symmetric GAP \( P \subset \R \) of rank \( r = O_{C, \eps}(1) \) such that \( P \) contains all but \( n' \) elements of \( W \) (counting multiplicity), and
\[
|P| = \max\Big\{ 1,\ O_{C, \eps}\Big(\rho(W)^{-1} (n)^{-r/2} \Big) \Big\}.
\]
\end{theorem}

There are many extensions to other settings, such as the extension from linear forms \(S\) of the \(\xi_i\) to quadratic and higher-order multilinear forms by Costello and Nguyen \cite{Cost,Ng-Duke}, and Meka et al. \cite{Meka}, as well as extensions to non-abelian groups by Tiep–Vu \cite{TiepV} and Juskevicius--Semetulskis \cite{JS} (see also \cite{KNgP,Ng-unpublished,Tao}). 
Other notable contributions include the works of Tao–Vu \cite{TVcir}, Rudelson--Vershynin \cite{RV}, and, more recently, Fox et al.\ \cite{Fox} and Kwan et al.\ \cite{KwanSSS,KwanS}. 
A detailed discussion of these papers is beyond the scope of this work. 

The continuous case proceeds in parallel with the discrete case, and we refer
the reader to the survey~\cite{NVsur} for further discussion. In particular, the
continuous analog of Theorem~\ref{thm:ILO} asserts that the elements of $W$
lie close to a small GAP.

\subsection{New setting: the space of random permutations (discrete setting)}

In this paper, we investigate the anti-concentration behavior of random sums in
a fundamentally different setting, where the ambient probability space is the
symmetric group equipped with the uniform distribution. Let
$\BMw=(w_1,\dots,w_n)$ and $\BMv=(v_1,\dots,v_n)$ be two fixed
vectors. We consider the random sum
\begin{equation}\label{eqn:S_p}
S_{\pi} (\BMw, \BMv)  := \sum_{i=1}^n w_{i}\, v_{\pi(i)},
\end{equation}  
where $\pi$ is a uniformly random permutation of $\{1,\dots,n\}$. In most of
the paper, when there is no danger of confusion, we use the shorthand
$S_{\pi}$ for  $S_{\pi} (\BMw, \BMv)$. 

\subsubsection {Recent forward theorems}\label{subsection:discrete} 
In this subsection, we survey recent results on bounding  the probability that $S_{\pi}=x$, for any fixed
value $x$.  In~\cite{Soze2}, S{\"o}ze considered the special case
$\BMv=(1,2,\dots,n)$ and proved the following result.

\begin{theorem}\label{thm:soze_main}
Let $\BMv=(1,2,\dots,n)$ and let $\BMw\in\R^n$ be a nonzero vector such that
$\BMw\cdot\mathbf{1}=0$, where $\mathbf{1}$ denotes the all-ones vector. Then
\[
\sup_x \P(S_{\pi}=x)=O\Big(\frac{1}{n}\Big).
\]
\end{theorem}

We remark that the assumption $\BMw\cdot\mathbf{1}=0$ may be imposed without
loss of generality. Indeed, if all coordinates of $\BMw$ are equal, then
$S_{\pi}$ is constant and no nontrivial anti-concentration statement can hold.
Thus, from the perspective of anti-concentration, it is natural to decompose
\[
\BMw=\alpha\mathbf{1}+\BMw',
\qquad \BMw'\cdot\mathbf{1}=0,
\]
and to focus on $S_{\pi}(\BMw',\BMv)$.

The bound $O(1/n)$ is sharp, as shown by taking $w_1=1$, $w_2=-1$, and
$w_i=0$ for $i\ge3$. In~\cite{Vuetal}, Berger et al. extended this
result to an arbitrary vector $\BMv$ with distinct coordinates.

\begin{theorem}\label{thm:Ber}
If the $v_i$ are all distinct and $\BMw$ is a nonzero vector such that
$\BMw\cdot\mathbf{1}=0$, then
\[
\sup_x \P(S_{\pi}=x)=O\Big(\frac{1}{n}\Big).
\]
\end{theorem}
Both Theorems~\ref{thm:soze_main} and~\ref{thm:Ber} are special cases of
continuous theorems that we will discuss later. Theorem~\ref{thm:Ber} may be
viewed as a permutation analog of the Littlewood--Offord--Erd\H{o}s result
from the product-space setting. Note that there is a $\sqrt{n}$ improvement
from the bound $O(n^{-1/2})$ to $O(n^{-1})$.

One may also interpret the above problems in terms of orbits of the symmetric
group. Given a vector $\BMv=(v_1,\dots,v_n)$ with distinct coordinates,
consider the orbit $\pi(\BMv)$ of $\BMv$ under $S_n$. How many of these orbit
points can lie in a fixed hyperplane $H\subset\R^n$? A more precise result was conjectured in~\cite{HMS} (see also~\cite{HGy}) and
later verified in~\cite{Paw} (and independently in~\cite{APZ}).

\begin{theorem}\label{thm:discrete:1}
Under the assumptions of Theorem~\ref{thm:Ber},
\begin{equation}\label{eqn:1}
\sup_x \P(S_{\pi}=x)
\le \frac{2\lfloor n/2\rfloor}{n(n-1)}.
\end{equation}
\end{theorem}
This bound is optimal. Indeed, take
\(
\BMv=(1,2,\dots,n)
\quad\text{and}\quad
\BMw=\bigl(-\!\sum_{i=2}^{n-1} i,\,
-\!\sum_{i=2}^{n-1} i,\,
n+1,\dots,n+1\bigr).
\)
In this case, $S_{\pi}=0$ if and only if
$v_{\pi(1)}+v_{\pi(2)}=n+1$.

The proof of Theorem~\ref{thm:discrete:1} is quite involved and relies on tools from algebraic combinatorics. In spirit, it is similar to Stanley’s proof \cite{Stan} of an optimal variant of Theorem~\ref{thm:SSz}. Roughly speaking, if $\BMv$ has distinct
coordinates and $S_{\pi}=x$, then the set of such permutations $\pi$ forms an
antichain (in a suitable weakening of the Bruhat order on $S_n$). This
antichain can be identified with a disjoint union of Bruhat orders on
$S_n/S_{\alpha}$ for certain parabolic subgroups $S_{\alpha}$. These posets
have the Sperner property, which implies that their size is bounded by the
largest rank of $S_n/S_{\alpha}$. These ranks correspond to coefficients of the
$q$-multinomial coefficient $\binom{n}{\alpha}_q$, and
Theorem~\ref{thm:discrete:1} follows from appropriate bounds on those
coefficients.

While Theorem~\ref{thm:discrete:1} is elegant—being optimal and valid for all
$n$—the above methods do not appear to extend easily to more general settings
where additional structure is imposed on $\BMv$ and $\BMw$. As observed by
Alon et al.~\cite{APZ}, if both the $v_i$ and the $w_i$ are distinct,
then $\P(S_{\pi}=x)$ must lie between $n^{-3/2}$ and $n^{-5/2}$. A sharp bound
in this regime would serve as a permutation-space analog of the
Erd\H{o}s--Moser result in the product-space setting. However, the techniques
of~\cite{APZ} and \cite{Paw} do not resolve this problem. We will address it in
Section~\ref{section:applications}.

\section {New results in the space of random permutations}

We start with highlights of our new results in this paper:

\begin{itemize}  
\item In the discrete setting, we prove an analog of Theorem \ref{thm:ILO}, namely Theorem \ref{thm:ILO:prod}, in the space of random permutations. In fact, we can prove a more general inverse result concerning two-dimensional arrays (matrices), Theorem \ref{thm:ILO:array}. The permutation setting is the special case when the array (matrix) has rank~1.

\vskip2mm 

\item In the continuous setting, 
we deal with the 
small-ball probability $\P (|S_{\pi} -L| \le \delta) $.  A remarkable feature here (compared to discrete case) is the dependence of the bound on the parameter $L$. 
Our new results Theorems \ref{thm:cont:1}, \ref{thm:cont:3/2}, and \ref{thm:cont:poly:1} establish  sub-gaussian behavior with respect to $L$. This is optimal and  answers a question posed by  S{\"o}ze \cite{Soze2}. Another key result in this direction, Theorem \ref{thm:comparison:2vs1}, sheds light on the behavior of joint distribution, showing that under certain conditions, two small ball events behave almost independently. 

\vskip2mm 

\item As direct  applications of Theorem \ref{thm:ILO:prod}, we obtain 
 several  {\it forward} results,  such as Theorems \ref{thm:discrete:1/2} , \ref{thm:discrete:3/2} and \ref{thm:discrete:5/2}.  In particular, we 
 succesfully address the range $n^{-3/2}$ and $n^{-5/2}$ discussed earlier, answering a question of Alon et al.~\cite{APZ}.

 \vskip2mm 

\item Anti-concentration results in the product setting have a large number of applications in different fields; see \cite{NVsur} for a survey. We expect that our new results in the random permutation setting will have a similar impact. We present two illustrations concerning random polynomials and random matrices.

\vskip2mm

\item We consider random polynomials of the form 
$P_{\pi}(x)= \sum_{i=1}^n w_{\pi(i) } x^{i-1}$, where $\BMw=(w_1,\dots,w_n)$ is a deterministic real vector and $\pi$ is a random permutation. A well known theorem of S{\"o}ze~\cite{Soze2} shows, under a mild assumption on $\BMw$,  that the number of real roots of $P_{\pi}$ is of order 
$O(\log n)$. We strengthen this result by showing that the number of 
extremal points (real roots of any fixed order derivative) of $P_{\pi}(x)$
is also of order $O(\log n)$; see Theorem \ref{thm:poly}.

\vskip2mm 

\item Regarding random matrices, we consider the matrix $Q_{n\times n}$ whose rows are $\pi_1(\BMw),\dots,\pi_n(\BMw)$, where $\BMw$ is a deterministic vector and $\pi_1,\dots,\pi_n$ are independent random permutations. We show, under a mild assumption on $\BMw$, that $Q_{n\times n}$ is non-singular with high probability; see Theorem~\ref{thm:singularity}. This theorem can be viewed as the permutation analog of the classical matrix singularity problem in product spaces.

\end{itemize}

\subsection{A new inverse theorem}
One of the main results of our paper is the following  analog of Theorem \ref{thm:ILO},  providing  a characterization of  those vectors $\BMw$ and $\BMv$ for which $S_{\pi}$ admits polynomially large concentration. 

\begin{theorem}[Inverse result for permutation sums]\label{thm:ILO:prod}
Let \( C>0\) and \( \eps \in (0,1) \) be constants. Assume that
\[
\rho := \sup_{x} \P_\pi\Big( \sum_{i=1}^n w_{i}\, v_{\pi(i)} = x \Big) \ge n^{-C}.
\]
Then for any $n^{\eps} \le n' \le n$ there exists a {proper symmetric} GAP of rank $r = O_{C, \eps}(1)$ and size $O_{\eps}(\rho^{-1}(n')^{-r/2})$ that contains $(w_i-w_j)(v_k-v_l)$  for at least $(1-n'/n)n^4$ quadruples $(i,j,k,l)$. In particular, there exists a {proper symmetric} GAP of rank $r = O_{C, \eps}(1)$ and size $O_{\eps}(\rho^{-1}n^{-r/2})$ that contains $(w_i-w_j)(v_k-v_l)$ for at least 
$(1-\eps)n^4$ quadruples.
\end{theorem}

The bound is sharp; we refer the reader to Lemma~\ref{lemma:optimal} for a proof. In Section~\ref{section:applications}, we will use this theorem to derive and
refine various \emph{forward} results. In particular, we will resolve the
question posed by Alon et al. discussed in the previous subsection.

\subsection {A more general theorem on two dimensional arrays} 
 Given an $n \times n$ array (matrix) with real entries $(a_{ij})_{1 \le i,j \le n}$, define the random sum
\[
S_\pi = \sum_{i=1}^n a_{i\pi(i)},
\]
where $\pi$ is a uniformly random permutation.

Such sums are referred to as {\it random combinatorial sums} in probability theory. 
The behavior of this sum, including CLT, has been studied by many authors \cite{Bolt, ChenFang, HoChen, CGS,  Hoeffding, Roos-CLT}. However, as far as we know, no anti-concentration result has been proved.   We can generalize Theorem \ref{thm:ILO:prod} as follows

\begin{theorem}[Inverse result for $2$D arrays]\label{thm:ILO:array}
Let $C>0$ and $0<\eps<1$ be constants. Let $(a_{ij})_{1 \le i,j \le n}$ be an $n \times n$ array of real-valued entries. Assume that
\[
\rho := \sup_{x} \P_\pi\Big( \sum_i a_{i\pi(i)} = x \Big) \ge n^{-C}.
\]
Then for any $n^{\eps} \le n' \le n$ there exists a {proper symmetric} GAP of rank $r = O_{C, \eps}(1)$ and size $O_{\eps}(\rho^{-1}(n')^{-r/2})$ that contains $a_{i k} - a_{j k} - a_{i l} + a_{j l}$ for at least
$(1-n'/n)n^4$ quadruples $(i,j,k,l)$. 
\end{theorem}

Theorem~\ref{thm:ILO:prod} corresponds to the special case $a_{ij}=w_i v_j$.

\subsection{New bounds on the small ball probability: The subgaussian behavior}

We assume that $\|\BMv\|_{\infty}=O(1)$. 
We inverstigate the small-ball probability 
$\P (
|S_{\pi}-L|\le \delta)$. 
 It is clear that this continuous problem
 contains the discrete 
problem considered earlier as a special case
with $\delta=0$. 

In the product-space setting, this problem has been studied extensively, with
early foundational results due to Rogozin~\cite{Ro}, Kolmogorov~\cite{Kol}, and
Hal\'asz~\cite{Halasz}, dated back to the 1960s and 1970s. There have also been many recent developments, especially
in the inverse direction; see, for instance,
\cite{NgV-Adv, RV, TVinverse} and the survey~\cite{NVsur}.
For $S_{\pi}$, CLT results and Berry--Esseen--type
estimates have been investigated beginning with the classical works of
Wald--Wolfowitz~\cite{GN,WW} and Hoeffding~\cite{Hoeffding} in the early 1950s.
More recent approaches based on the Lindeberg exchange method or Stein’s method
can be found in
\cite{AChW, Bolt, ChenFang, HoChen, Roos-CLT}. We refer the reader to
\cite{CGS} and the references therein for further contributions concerning this
important statistic. The best anti-concentration bound obtainable from these
classical approaches is of order $O(n^{-1/2})$ (independent of the length of the
interval), arising from the rate of convergence in the CLT.

A key feature of this part is the dependence of the bounds on the parameter $L$. We show that these bounds exhibit sub-gaussian decay in $L$. This behavior is consistent with the central limit theorem and answers a question raised by S\"oze in \cite{Soze2}.

In what follows, we restrict our attention to one of the most natural choices for $\BMv$, namely
sequences arising from polynomials. We use this setting in one of our main applications. Our method applies in a more general setting, but the treatment is more technical and will appear elsewhere.

\subsubsection{A starting point: the linear case $v_i = i/n$}

Our starting point is the following result of S{\"o}ze~\cite[Lemma~4]{Soze2}, in
which he studied $S_{\pi}$ from~\eqref{eqn:S_p} for the special choice of
$\BMv$ given by $v_i=i/n$.

\begin{theorem}\label{lemma:Koze:0}
Assume that $\sum_i w_i = 0$ and $\sum_i w_i^2 = 1$. Then for every $L\in\R$,
\[
\P\Bigl(\bigl|S_{\pi}-L\bigr|\le \frac{1}{n}\Bigr)
= O\Bigl(\frac{1}{n} e^{-\Theta(|L|)}\Bigr).
\]
\end{theorem}

An interesting feature of this theorem is the appearance of the parameter $L$
in the bound, which implies that $S_{\pi}$ exhibits exponential decay. S{\"o}ze
\cite{Soze2} conjectured a sharper estimate of order
$O(n^{-1} e^{-\Theta(L^{2})})$, corresponding to sub-gaussian
decay. Notice that once $L$ appears on the right-hand side, the magnitude of the $v_{i}$ becomes
relevant.

The proof in~\cite{Soze2} is clever but rather involved. Roughly speaking, the
author compares the sum $\sum_i w_i v_{i}$ with $\sum_i w_i u_i$, where the
$u_i$ are i.i.d.\ uniformly distributed on $(0,1)$, and then exploits certain
ad hoc unimodality properties of the resulting sum. Nonetheless, the resulting
bound is not optimal.

\begin{problem}\label{Problem:Soze}
Can one achieve a sub-gaussian bound in Theorem~\ref{lemma:Koze:0}?
\end{problem}

Other natural questions include weakening the  restriction
$\BMv=(1,\dots,n)/n$ and treating scales smaller than $1/n$. We achieve
these goals under a modest assumption on $\BMw$. In what follows, we impose the following assumption.

\begin{condition}[Non-degeneracy]\label{cond:separation}
We say that a sequence $w_1,\dots,w_n$ satisfying
\[
\sum_{i=1}^n w_i = 0
\quad\text{and}\quad
\sum_{i=1}^n w_i^2 = 1
\]
is \emph{non-degenerate} if
\begin{equation}\label{eqn:separation}
|w_i-w_j| \;\le\; \frac{1}{A\sqrt{\log n}}
\end{equation}
for all distinct $i,j$ and for some constant $A$.
\end{condition}

Roughly speaking, this condition corresponds to $\BMw$ being a unit vector
orthogonal to $\mathbf{1}$ and satisfying
$\|\BMw\|_{\infty}=o(1/\sqrt{\log n})$.

\begin{theorem}[New result at scale $1/n$]\label{thm:cont:1}
Let $\delta>0$ be given. Suppose that the sequence $(w_1,\ldots,w_n)$ satisfies
Condition~\ref{cond:separation} for some sufficiently large constant $A>0$.
Let $I\subset[n]$ be any subset with $|I|\ge \delta n$, and consider a sequence
$(v_1,\ldots,v_n)$ that is partially specified by
\[
v_i = i/n \quad \text{for all } i\in I.
\]
Then, for every $L\in\R$, we have the uniform bound
\begin{equation}\label{eqn:1:u}
\P\Bigl(\bigl|S_{\pi}-L\bigr|\le \frac{1}{n}\Bigr)
= O_{A}\Big(\frac{1}{n}\Big).
\end{equation}
If, in addition, there exists a constant $\widetilde{B}>0$ such that
$|v_i|\le \widetilde{B}$ for all $i\in[n]$, then we obtain the sub-gaussian bound
\begin{equation}\label{eqn:1:L}
\P\Bigl(\bigl|S_{\pi}-L\bigr|\le \frac{1}{n}\Bigr)
= O\Big(\frac{1}{n} e^{-\Theta(L^{2})}\Big).
\end{equation}
Here, the implied constants depend only on $A$ and $\widetilde{B}$.
\end{theorem}

\begin{corollary}
The answer to Problem~\ref{Problem:Soze} is affirmative for any vector $\BMw$
satisfying Condition~\ref{cond:separation}.
\end{corollary}

The next theorem allows us to treat the smaller scale $n^{-3/2}$ under the
additional assumption that most of the coefficients $w_i$ are not squeezed
into a very small interval. We refer the reader to
Remark~\ref{rmk:optimal:3/2} for a discussion of the sharpness of this
assumption.

\begin{theorem}[New treatment at scale $n^{-3/2}$]\label{thm:cont:3/2}
Let $0<\varepsilon<1/2$ and $\delta>0$ be given. Suppose that the sequence
$(w_1,\ldots,w_n)$ satisfies Condition~\ref{cond:separation} for some
sufficiently large constant $A>0$. Furthermore, assume that no interval of
length $\varepsilon/\sqrt{n}$ contains more than $(1-\varepsilon)n$ of the
values $w_i$. Let $I\subset[n]$ be any subset with $|I|\ge \delta n$, and
consider a sequence $(v_1,\ldots,v_n)$ that is partially specified by
\[
v_i = i/n \quad \text{for all } i\in I.
\]
Then, for every $L\in\R$, we have the uniform bound
\begin{equation}\label{eqn:3/2:u}
\P\Bigl(\bigl|S_{\pi}-L\bigr|\le \frac{1}{n^{3/2}}\Bigr)
= O_{A}\Big(\frac{1}{n^{3/2}}\Big).
\end{equation}
If, in addition, there exists a constant $\widetilde{B}>0$ such that
$|v_i|\le \widetilde{B}$ for all $i\in[n]$, then we obtain the sub-gaussian bound
\begin{equation}\label{eqn:3/2:L}
\P\Bigl(\bigl|S_{\pi}-L\bigr|\le \frac{1}{n^{3/2-\eps}}\Bigr)
= O\Big(\frac{1}{n^{3/2-\eps}} e^{-\Theta(L^{2})}\Big).
\end{equation}
Here, the implied constants depend only on $A$ and $\widetilde{B}$.
\end{theorem}

This result may be viewed as a continuous analog of
Theorem~\ref{thm:discrete:3/2}. We emphasize that the condition excluding
intervals of length $1/\sqrt{n}$ that contain most of the $w_i$ is essential
for obtaining small-ball estimates at the $n^{-3/2}$ scale. In the
$L$-dependent bound~\eqref{eqn:3/2:L}, we assume the slightly larger radius
$n^{-3/2+\varepsilon}$ purely for technical convenience; see our treatment of
the ``very large~$|t|$'' regime in the proof of~\eqref{eqn:3/2:L} in
Section~\ref{sect:linear}.

Moreover, Theorem~\ref{thm:cont:5/2} shows that even finer approximations—down
to the scale $n^{-5/2+o(1)}$—are possible under stronger assumptions on the
coefficients $w_i$.

Finally, we remark that Theorem~\ref{thm:cont:1} may still hold without
Condition~\ref{cond:separation}. However, for
Theorems~\ref{thm:cont:3/2} and~\ref{thm:cont:5/2}, additional assumptions (of
the type imposed above) are necessary, since these results fail dramatically
when only a few of the $w_i$ are nonzero.
 
\subsubsection{More general conditions on $v_i$}

In this subsection, we generalize the above theorems by allowing the coefficients
$v_i$ to depend polynomially on $i$.

\begin{theorem}[Treatment at scale $1/n$]\label{thm:cont:poly:1}
Let $d\ge2$ be a fixed integer, and let $\delta>0$, $b\neq0$, and $B>0$ be
constants. Suppose that the sequence $(w_1,\ldots,w_n)$ satisfies
Condition~\ref{cond:separation} for some sufficiently large constant $A>0$.
Let $I\subset[n]$ be any subset with $|I|\ge \delta n$, and consider a sequence
$(v_1,\ldots,v_n)$ that is partially specified by
\[
v_i=\frac{P_d(i)}{n^{d}} \quad \text{for all } i\in I,
\]
where $P_d(i)$ is a real polynomial of degree $d$ with fixed leading coefficient
$b$, and whose remaining coefficients are allowed to depend on $n$, subject to
the bound
\[
|v_i|\le B \quad \text{for all } i\in I.
\]
Then, for any $L\in\R$, we have the uniform estimate
\begin{equation}\label{eqn:poly:u}
\P\Bigl(\bigl|S_{\pi}-L\bigr|\le \frac{1}{n}\Bigr)
= O_{A}\Bigl(\frac{1}{n}\Bigr).
\end{equation}
If, in addition, there exists a constant $\widetilde{B}>0$ such that
$|v_i|\le \widetilde{B}$ for all $i\in[n]$, then
\begin{equation}\label{eqn:poly:L}
\P\Bigl(\bigl|S_{\pi}-L\bigr|\le \frac{1}{n}\Bigr)
= O\Bigl(\frac{1}{n} e^{-\Theta(L^2)}\Bigr).
\end{equation}
Here, the implied constants depend on $A,B$ and $\widetilde{B}$.
\end{theorem}

It is also possible to treat the next scale $n^{-3/2}$ in this polynomial
setting. However, the argument becomes significantly more involved, and we do
not pursue this direction here.

\subsection{Joint distributions}

We now turn to the more difficult problem of studying joint distributions.
Specifically, we consider three vectors $\BMw$, $\BMv$, and $\BMv'$, and the
event
\[
\Bigl\{ \bigl|S_{\pi}(\BMw,\BMv)-L_{1}\bigr|\le \frac{1}{n} \Bigr\}
\;\wedge\;
\Bigl\{ \bigl|S_{\pi}(\BMw,\BMv')-L_{2}\bigr|\le \frac{1}{n} \Bigr\}.
\]

\begin{theorem}[Joint distribution]\label{thm:comparison:2vs1}
Let $d\ge2$ be a fixed integer, and let $\delta>0$, $b\neq0$, $c\neq0$, and
$B>0$ be given. Suppose that the sequence $(w_1,\ldots,w_n)$ satisfies
Condition~\ref{cond:separation} for some sufficiently large constant $A>0$.
Let $I\subset[n]$ be any subset with $|I|\ge \delta n$, and consider the
sequences $(v_1,\ldots,v_n)$ and $(v_1',\ldots,v_n')$ partially specified by
\[
v_i=\frac{P_d(i)}{n^{d}}
\quad\text{and}\quad
v_i'=\frac{P_{d-1}(i)}{n^{d-1}}
\quad \text{for all } i\in I,
\]
where $P_d(i)$ and $P_{d-1}(i)$ are real polynomials of degrees $d$ and $d-1$,
respectively, with fixed leading coefficients $b$ and $c$. The remaining
coefficients may depend on $n$, subject to the bound
\[
|v_i|,\,|v_i'|\le B \quad \text{for all } i\in I.
\]
Then, for any $L_1,L_2\in\R$, we have
\begin{equation}\label{eqn:comparison:u}
\P\Bigl(
\bigl|S_{\pi}(\BMw,\BMv)-L_1\bigr|\le \frac{1}{n}
\;\wedge\;
\bigl|S_{\pi}(\BMw,\BMv')-L_2\bigr|\le \frac{1}{n}
\Bigr)
= O\Big(\frac{1}{n^2}\Big).
\end{equation}
If, in addition, there exists a constant $\widetilde{B}>0$ such that
$|v_i|,|v_i'|\le \widetilde{B}$ for all $i\in[n]$, then
\begin{equation}\label{eqn:comparison:L}
\P\Bigl(
\bigl|S_{\pi}(\BMw,\BMv)-L_1\bigr|\le \frac{1}{n}
\;\wedge\;
\bigl|S_{\pi}(\BMw,\BMv')-L_2\bigr|\le \frac{1}{n}
\Bigr)
= O\Big(\frac{1}{n^2} e^{-\Theta(L_1^2+L_2^2)}\Big).
\end{equation}
Here, the implied constants depend on $A,B$ and $\widetilde{B}$.
\end{theorem}

It may be possible to extend our approach to sequences of polynomials of
non-consecutive degrees; however, we do not pursue this direction in the
present paper. We note that the weaker bound
$O(\frac{1}{n} e^{-\Theta(L_1^2+L_2^2)})$,
involving $O(1/n)$ rather than $O(1/n^2)$, follows immediately from the
one-dimensional estimates. We also refer the reader to
Theorem~\ref{thm:comparison:2vs1'} for an application of the above result
comparing the quantities $S_{\pi}(\BMw,\BMv)$ and $S_{\pi}(\BMw,\BMv')$.

\subsection{An inverse theorem via 
discretization}  We can obtain an inverse theorem in the continuous setting via a simple
discretization argument. For simplicity, assume that
$\|\BMw\|_{\infty}=\|\BMv\|_{\infty}=1$.
Fix a parameter $\alpha>0$ and round each $w_i$ and $v_i$ to the nearest integer
multiple of $\alpha$; denote the resulting vectors by $\BMw'$ and $\BMv'$,
respectively. This rounding procedure changes the value of $S_{\pi}$ by at most
$3n\alpha$. Moreover, the new sum $S_{\pi}(\BMw',\BMv')$ takes values in integer
multiples of $\alpha^2$.

If $S_{\pi}$ lies in an interval of length $2\delta$, then
$S_{\pi}(\BMw',\BMv')$ must belong to a discrete set $D$ of size at most
\(
m := \Big\lceil \frac{1}{\alpha^2}\bigl(2\delta+3n\alpha\bigr)\Big\rceil.
\)
Thus, if the (continuous) anti-concentration probability is $\rho$, then for
some $x\in D$ we have
\(
\P\bigl(S_{\pi}(\BMw',\BMv')=x\bigr)\ge m^{-1}\rho,
\)
placing us in a position to apply Theorem~\ref{thm:ILO:prod}.

\begin{theorem}[Inverse result for permutation sums: continuous setting]
\label{thm:ILO:prod1}
Let $\BMw,\BMv$ be unit vectors. Let $C,C'>0$ and $\varepsilon\in(0,1)$ be
constants. Assume that for some $L\in\R$ and $\delta\le1$,
\[
\rho := \P\bigl(|S_{\pi}-L|\le\delta\bigr)\ge n^{-C}.
\]
For any $\alpha\ge n^{-C'}$ and $n^{\eps}\le n'\le n$, there exists a proper symmetric GAP $Q$ of rank
$r_0=O_{C,C',\eps}(1)$ and size
$O_{\eps}\bigl(m\rho^{-1}(n')^{-r_0/2}\bigr)$ such that for at least
$(1-n'/n)n^4$ quadruples $(i,j,k,l)$, the quantity
$(w_i-w_j)(v_k-v_l)$ lies within distance at most $\alpha$ of a point in $Q$.
\end{theorem}

To complete this section, we introduce a notion--closely related to the Least Common Denominator (LCD) concept of Rudelson and Vershynin~\cite{RV}--that captures small-ball probabilities rather efficiently.

\begin{definition}\label{defn:LCD}  
Let $\kappa \ge n^{3/2}$ and $0 < \gamma < 1$. The \emph{Essential Least Common Divisor} of a pair of vectors $\BMw, \BMv \in \R^n$ is defined by  
\[
\LCD_{\gamma, \kappa} (\BMw, \BMv)
\;=\;
\inf\Big\{ D > 0 \colon \dist\big(D \BMu, \Z^{n^4}\big)
< \min\{\gamma D \|\BMu\|_{2}, \, \kappa\} \Big\},
\]  
where $\BMu \in \R^{n^4}$ is the vector whose $(i,j,k,l)$-th coordinate is
\[
(v_{i} - v_{j})(w_{k} - w_{l}), \qquad 1\le i,j,k,l\le n.
\]
\end{definition}

\begin{theorem}\cite[Theorem~3.2]{Tran}\label{thm:LCD}  
Under the notation of Definition~\ref{defn:LCD}, assume that $\|\BMu\|_{2} \ge n^{3/2}$. Then, for any
\(
\delta \;\ge\; 1/\LCD_{\gamma, \kappa} (\BMw, \BMv),
\)
we have
\[
\sup_{x \in \R} \P\big( |S_{\pi} - x| \le \delta \big)
=
O\Big(\frac{\delta}{\gamma} \,+\, e^{-\kappa^{2} / 2n^{3}}\Big).
\]
\end{theorem}

We will include a proof of this result in Appendix \ref{sect:LCD} for the reader's convenience.

\section{Applications}\label{section:applications}

For a vector $\BMv=(v_1,\dots,v_n)$, define
$m_{\BMv}=\max\limits_a \#\{i:1\le i\le n,\ v_i=a\}$,
the maximum multiplicity of an entry of $\BMv$.

\subsection{Forward theorems}

Our first result is an analog of the Erd\H{o}s--Littlewood--Offord bound, Theorem \ref{thm:ELO}, with the assumption that the multiplicity of any value among the $w_i$ and $v_{i}$ are not too close to $n$.

\begin{corollary}\label{thm:discrete:1/2}
Let $0<c<1$ be a constant. Suppose that $m_{\BMw}, m_{\BMv} \le cn$.
Then
\[
\sup_{x}\P\bigl(S_{\pi}=x\bigr)
= O_{c}\Big(\frac{1}{n^{1/2}}\Big).
\]
\end{corollary}

The rate $O(n^{-1/2})$ is optimal. Indeed, for $\BMw=\BMv=(-1,\dots, -1,1,\dots, 1)$, the left-hand side equals $\Theta(n^{-1/2})$.

Our next corollary improves Theorems~\ref{thm:soze_main}
and~\ref{thm:Ber} under the additional assumption that the multiplicity of any
value among the $w_i$ is not too close to $n$.

\begin{corollary}\label{thm:discrete:3/2}
Let $0<c<1$ be a constant. Suppose that
$m_{\BMw} \le cn$ and that the $v_i$ are all distinct. Then
\[
\sup_{x}\P\bigl(S_{\pi}=x\bigr)
= O_{c}\Big(\frac{1}{n^{3/2}}\Big).
\]
\end{corollary}

Next, we address a problem raised in~\cite{APZ} and discussed in
Subsection~\ref{subsection:discrete}. We show that if both the $w_i$ and the
$v_j$ are distinct, then the optimal decay rate is $n^{-5/2}$, up to a
logarithmic factor. This result may be viewed as an analog of
Theorem~\ref{thm:SSz} in the product-space setting.

\begin{corollary}\label{thm:discrete:5/2}
If all the $w_i$ are distinct and all the $v_j$ are distinct, then
\[
\sup_{x}\P\bigl(S_{\pi}=x\bigr)
= O\Big(\frac{\log n}{n^{5/2}}\Big).
\]
\end{corollary}

We allow the possibility that $w_i=v_j$ for some $i,j$.
In particular, taking $w_i=v_i=i$ yields a bound of order $n^{-5/2}\log n$. An independent and slightly weaker bound of order $n^{-5/2}(\log n)^2 $ was posted by Hunter et al. \cite{HPZ} two weeks after the initial upload of this manuscript. Under stronger assumptions on $\BMv$ and $\BMw$, our method (see Theorem~\ref{thm:ILO:prod}) can establish
arbitrarily strong bounds on the concentration probability $\P(S_{\pi}=x)$, whereas their approach does not seem to yield any bound stronger than $n^{-5/2+o(1)}$. 
We will derive 
Corollary~\ref{thm:discrete:3/2} and Corollary~\ref{thm:discrete:5/2} using Theorem~\ref{thm:ILO:prod} in
Section~\ref{sect:discrete}.

\medskip

We also observe that the continuous theorems can be used to obtain sharper
bounds in the discrete setting. For instance, the following result is a direct
corollary of Theorem~\ref{thm:cont:1}.

\begin{corollary}
Under the assumptions of Theorem~\ref{thm:cont:1}, we have
\begin{equation}
\P(S_{\pi}=x)
= O\Big(\frac{1}{n} e^{-\Theta(x^2)}\Big).
\end{equation}
\end{corollary}

\subsection{Critical points of random polynomials}

A random polynomial is a function of the form
\[
P(x)=\sum_{i=0}^n \xi_i x^i,
\]
where the $\xi_i$ are random variables. This is a central subject in both
probability theory and analysis, with a long and rich history beginning with
the foundational works of Littlewood--Offord and Kac in the 1940s. One of the
main questions in the theory concerns the number and distribution of the real
roots and critical points of $P$.

Let us first discuss the most basic class of random polynomials, namely the
\emph{Kac polynomials}, where the $\xi_i$ are i.i.d.\ copies of a random
variable $\xi$ with mean zero and unit variance. Using his celebrated formula,
Kac showed that when $\xi\sim N(0,1)$, the expected number of real roots
satisfies
\begin{equation}\label{expectation}
\E N_{\R}(P)=\Bigl(\frac{2}{\pi}+o(1)\Bigr)\log n .
\end{equation}

It took more than a decade until Erd\H{o}s and Offord extended this result to
the case where $\xi$ is Rademacher (taking values $\pm1$ with probability
$1/2$), using completely different methods. About ten years later,
Ibragimov and Maslova~\cite{IM} showed that~\eqref{expectation} holds for any
$\xi$ with zero mean and unit variance.

The problem of counting critical points is even more delicate. Observe that
for a differentiable function $F$, between any two consecutive real roots of
$F$ there must be a real root of $F'$. Consequently, for any fixed $d$,
\[
\E N_{\R}\bigl(P^{(d)}\bigr)
\;\ge\;
\Bigl(\frac{2}{\pi}+o(1)\Bigr)\log n .
\]
However, equality does \emph{not} hold. Maslova~\cite{Mas} famously proved that
for any fixed $d$, the Kac polynomial satisfies
\begin{equation}\label{eqn:critical}
\E N_{\R}\bigl(P^{(d)}\bigr)
=
\frac{1+\sqrt{1+2d}+o(1)}{\pi}\log n .
\end{equation}

Far less is known about random polynomials with dependent coefficients. In
fact, the only available results appear to be
\cite{Soze1,Soze2} and~\cite{ONg}, which treat models in which the coefficients
are exchangeable or weakly stationary. In this section, we focus on the family
of random polynomials with exchangeable coefficients introduced in
\cite{Soze1,Soze2}.

Given a real vector $\BMw=(w_1,\dots,w_n)$, we consider the random polynomial
\[
P_{\pi}(x)=\sum_{i=1}^n w_{\pi(i)} x^i,
\]
where $\pi$ is a uniformly random permutation. We start the index at $i=1$ to
be consistent with our convention that $\BMw$ is a vector of length $n$.
Obviously, the same results hold for
$P_{\pi}(x)=\sum_{i=0}^n w_{\pi(i)}x^i$.

Assuming $\BMw\neq0$, S{\"o}ze~\cite[Theorem~1]{Soze2} proved that the expected
number of \emph{nonzero} real roots $N_{\R}^\ast(P_{\pi})$ satisfies
\[
\E N_{\R}^\ast(P_{\pi})=O(\log n).
\]

It is natural to conjecture that the same bound holds for the number of
nonzero critical points of $P_{\pi}$. Unfortunately, the approach
in~\cite{Soze2} does not extend to derivatives, since exchangeability breaks
down for the sequence $(i\,w_{\pi(i)})$. In this section, using the new
anti-concentration results developed in this paper, we answer this question
affirmatively under some additional (but natural) conditions on the weights
$w_i$.

Let\footnote{Conceptually, it would be more natural to define
$\sigma(\BMw)=\sqrt{\frac{1}{n}\sum_{i=1}^n (w_i-\overline{\BMw})^2}$. However,
we retain the present normalization to remain consistent with the literature.}
\[
\overline{\BMw}=\frac{1}{n}\sum_{i=1}^n w_i,
\qquad
\sigma(\BMw)=\sqrt{\sum_{i=1}^n (w_i-\overline{\BMw})^2}.
\]
The following condition is analogous to Condition~\ref{cond:separation}, but
is invariant under shifts and rescaling.

\begin{condition}[Non-degeneracy: $K$-balanced]\label{cond:shift}
Let $K>1$. For each $k\in\Z^{+}$, define
\[
M_k(\BMw):=\frac{1}{n}\sum_{i=1}^n (w_i-\overline{\BMw})^k .
\]
We say that $\BMw=(w_1,\dots,w_n)$ is $K$-balanced if
\[
M_4(\BMw)\le K\, M_2(\BMw)^2 .
\]
\end{condition}

Equivalently, the rescaled squares
\(
X_i:=n\,(w_i-\overline{\BMw})^2/\sigma(\BMw)^2
\)
satisfy
\[
\frac{1}{n}\sum_{i=1}^n X_i^2 \le K.
\]
By the Cauchy--Schwarz inequality we always have
$M_2(\BMw)^2\le M_4(\BMw)$ (and $\sum_{i=1}^n X_i^2/n\ge1$).

\begin{example}
One simple example is when $|w_i|=1$ for all $i$ and
$\bigl|\sum_{i=1}^n w_i\bigr|\le (1-\varepsilon)n$ for some constant
$\varepsilon>0$; in this case, $\BMw$ is $K$-balanced for some
$K=K(\varepsilon)$. Another family of examples is given by vectors of the form
$w_i=a+t\,b_i$, where $a\in\R$, $t\neq0$, and
$b_1,\dots,b_n\in\{0,\dots,n\}$ have maximal multiplicity at most $0.99\,n$.
This includes, for instance,
$\BMw=(a,a+t,\dots,a+nt)$ or
$\BMw=(\underbrace{a,\dots,a}_{(1-\delta)n},
\underbrace{a+t,\dots,a+t}_{\delta n})$.
\end{example}

\begin{theorem}\label{thm:poly}
Assume that $\BMw$ satisfies Condition~\ref{cond:shift}. Then for any
nonnegative integer $d$,
\[
\E N_{\R}\bigl(P_{\pi}^{(d)}\bigr)
= O_{K,d}(\log n).
\]
\end{theorem}

Finally, let $P_{{\rm Rad},k}$ denote a Kac polynomial with Rademacher coefficients,
conditioned on $\sum_{i=0}^n \xi_i=k$, where
$-(1-\varepsilon)n\le k\le (1-\varepsilon)n$. For example, $P_{{\rm Rad},0}$
has the uniform distribution over polynomials with $\pm1$ coefficients having
exactly half of the coefficients equal to $1$. Theorem~\ref{thm:poly} implies
the following.

\begin{corollary}[Critical points of conditional Kac polynomials]
For any nonnegative integer $d$,
\[
\E N_{\R}\bigl(P_{{\rm Rad},k}^{(d)}\bigr)
= O_{d,\varepsilon}(\log n).
\]
\end{corollary} 

\subsection{Singularity of random row-permutation matrices}\label{sub:sing} Let $d\le n$ 
be positive integers such that
$\min(d,n-d)=\Omega(n)$.
Consider the random $\{0,1\}$-matrix $Q_{n,d}$ whose rows are independent
uniform random vectors with exactly $d$ ones.
Nguyen~\cite{Nguyen} proved that for every $C>0$, if $n$ is sufficiently large, then
\[
\P\bigl(Q_{n,d}\ \text{is singular}\bigr)\le n^{-C}.
\]
This polynomial upper bound was later strengthened by Ferber et al.~\cite{FJLS},
and subsequently by Tran~\cite{Tran}, to an exponential bound.
Related analogs for the least singular value were studied by Nguyen--Vu~\cite{NgV-Comb},
Jain~\cite{Jain}, Tran~\cite{Tran}, and Jain--Sah--Sawhney~\cite{JSS22}; see also~\cite{AChW}.
We also refer the reader to Section~\ref{section:comments} for related discussion.
Above all, all existing approaches rely heavily on the $\{0,1\}$ nature of the matrix entries.

By applying Theorem~\ref{thm:ILO:prod}, we show that the above model
admits a substantial generalization. 

\begin{theorem}\label{thm:singularity}
Let $0<\eps<1$. For every $C>0$ there exists $n_0=n_0(\eps,C)$ such that the following holds for all $n\ge n_0$.
Let $\BMv=(v_1,\dots,v_n)\in\R^n$ satisfy $m_{\BMv}\le (1-\eps)n$ and $\sum_{i=1}^n v_i\neq 0$.
Let $\pi_1,\dots,\pi_n$ be independent uniformly random permutations of $[n]$, and let $Q_{n \times n}$ be the $n\times n$ matrix whose $i$-th row is $\pi_i(\BMv)$.
Then
\[
\P(Q_{n\times n}\text{ is singular})\le n^{-C}.
\]
\end{theorem}

Our proof also extends to perturbed matrices $Q_{n\times n}+F$, where $F$ is a deterministic matrix. It is plausible that the above polynomial bound can be strengthened to an exponential bound
(or even smaller when the entries of $\BMv$ are sufficiently spread out, i.e., $m_{\BMv}$ is small), or to the setting where $m_{\BMv}\le n-n^\eps$, but we do not pursue such refinements here. Our main goal is to provide a clean general framework.

\section{Some lemmas} \label{sect:lemmas}

\subsection{Bound for the characteristic function}

Consider the sum $S_{\pi}$ in its general form:
\[ 
S_{\pi} = \sum_{k=1}^{n} a_{k \pi(k)}.\]
Its characteristic function $\varphi(t)$ can be expressed as
\[
\varphi(t) = \mathbb{E} \, e^{it S_{\pi}} = \frac{1}{n!}\sum_{\pi \in \mathbb{S}_n} \prod_{k=1}^{n} e^{it a_{k\pi(k)}}= \frac{1}{n!} \perm(M),
\]
where $\perm(M)$ denotes the permanent of the $n \times n$ matrix $M$ with entries $(e^{it a_{kl}})_{1 \le k,l \le n}$.

By establishing upper bounds on the permanent of such matrices, Roos \cite[Theorem 1.4]{Roo} was able to prove the following result:

\begin{theorem}[Roos]\label{thm:Roo}
Let $S_{\pi}=\sum_{k=1}^{n} a_{k\pi(k)}$ and define 
$y_{i,j,k,l}=a_{ik}-a_{jk}-a_{il}+a_{jl}$ for $i\ne j$ and $k\ne l$.
Set $d=\lfloor n/2\rfloor$. Then the characteristic function 
$\varphi(t)=\mathbb{E}e^{itS_{\pi}}$ satisfies
\[
|\varphi(t)|\le \Big(\frac{1}{n^2(n-1)^2}
\sum_{i\ne j,\,k\ne l}
\cos^2\Big(\frac{t y_{i,j,k,l}}{2}\Big)\Big)^{d/2}.
\]
\end{theorem}

The proof of Theorem \ref{thm:ILO:array} will use the following corollary of this result.

\begin{corollary}\label{cor:Roos}
Assume that $S_{\pi} = \sum_{k=1}^{n} a_{k \pi(k)}$.
Then its characteristic function $\varphi(t) = \E e^{it S_{\pi}}$ satisfies
\[
|\varphi(2\pi t)| \le \exp\Big(-\frac{1}{2n^{3}} \sum_{1\le i,j,k,l\le n} \|t (a_{ik} - a_{jk} -a_{il} + a_{jl})\|^2_{\R/\Z}\Big).
\] 
\end{corollary}
\begin{proof}
By convexity, we have $|\sin(\pi x)|\ge 2 \|x\|_{\R/\Z}$ for all $x\in \R$. Hence
$\cos^2 (\pi x)=1-\sin^2 (\pi x) \le 1-4\|x\|^2_{\R/\Z}$. Applying this to Theorem~\ref{thm:Roo}, we obtain
\begin{align*}
|\varphi(2\pi t)| &\le \Big(  \frac{1}{n^{2}(n-1)^{2}} \sum_{i\ne j, \, k\ne l} \cos^{2} (\pi t (a_{ik} - a_{jk} -a_{il} + a_{jl}))\Big)^{ \lfloor n/2 \rfloor/2}\\
&\le \Big( 1- \frac{4}{n^{2}(n-1)^{2}} \sum_{i\ne j, \, k\ne l} \|t (a_{ik} - a_{jk} -a_{il} + a_{jl})\|^2_{\R/\Z}\Big)^{n/8}\\
&\le \exp\Big(-\frac{1}{2n^{3}} \sum_{i\ne j, \, k\ne l} \|t (a_{ik} - a_{jk} -a_{il} + a_{jl})\|^2_{\R/\Z}\Big)\\
&=\exp\Big(-\frac{1}{2n^{3}} \sum_{1\le i,j,k,l\le n} \|t (a_{ik} - a_{jk} -a_{il} + a_{jl})\|^2_{\R/\Z}\Big),
\end{align*}
where in the third inequality we used the fact that $1-x\le \exp(-x)$ for any $0\le x\le 1$.
\end{proof}

\subsection{Large deviation result}

The following result, which follows from \cite[Theorem 3.1]{AChW} (or \cite[Corollary 2.3]{Vuetal}) via Talagrand's concentration inequality, will be crucial.

\begin{lemma}\label{lemma:deviation}
Let $S_{\pi} = \sum_{i=1}^{n} w_i v_{\pi(i)}$. Then, for some positive constant $C_0$, we have
\[
\P\Big( | S_{\pi} - \mathbb{E} S_{\pi} | \ge \lambda \, \sigma(\BMw) \, \| \BMv \|_{\infty} \Big) \le C_0 e^{-C_0 \lambda^2}.
\]
In other words, the normalized random variable $(S_{\pi} - \mathbb{E} S_{\pi})/\sigma(\BMw) \| \BMv \|_{\infty}$ is subgaussian. 

As a consequence, if $|v_{i}|\le \widetilde{B}$ for all $i\in [n]$, then
    \[
    \P\Big( \Big| \sum_{i=1}^{n} w_{i} v_{\pi(i)} - \Big( \sum_{i=1}^{n} v_{i} \Big) \overline{\BMw} \Big| \ge \lambda  \sigma(\BMw) \Big) \le C_0 e^{-C_0 \lambda^2},
    \]
where $C_0$ is a positive constant depending only on $\widetilde{B}$.
\end{lemma}

An immediate consequence of the above lemma is the following bound on the moment generating function of the normalized variable
\[
\bar{S} := \frac{S_{\pi} - \mathbb{E} S_{\pi}}{\sigma(\BMw)\|\BMv\|_{\infty}}.
\]
There exists a constant $C_0'>0$ such that (see, for example, \cite[Proposition 2.6.1]{VerHDP})
\begin{equation}\label{eqn:mgf:1}
m_{\bar{S}}(t) := \mathbb{E} e^{t\bar{S}}
\le C_0' e^{C_0' t^2}, \qquad t\in\R.
\end{equation}

\subsection{Diophantine Properties}

We begin with linear forms.

\begin{lemma}[Wrapping around for linear forms]\label{lemma:wraparound}
For any $0<\delta<1$, there exists a constant $C>0$ such that the following holds. 
Let $I\subset\{-n,\dots,n\}$ with $|I|\ge \delta n$. Then, for $\frac{C}{n}\le |b|\le \frac{1}{C}$ and any $b_0\in\R$,
\[
\sum_{r\in I} \|br+b_0\|_{\R/\Z}^2=\Theta_{\delta}(n).
\]
\end{lemma}

The above immediately yields the following simple result.

\begin{corollary}\label{cor:wraparound} 
Let $0 < \delta < 1$ and $I \subset \{-n,\dots,n\}$ with $|I| \geq \delta n$. For every constant $C > 0$ sufficiently large in terms of $\delta$, and for all $\frac{1}{Cn} \leq |b| \leq \frac{1}{C}$ and for any $b_0 \in \R$ we have
$$\sum_{r \in I} \| b r +b_0 \|_{\R/\Z}^2 =\Theta_{C} (n).$$
\end{corollary}

\begin{proof}(of Corollary \ref{cor:wraparound})
For $\tfrac{C}{n} \le |b| \le \tfrac{1}{C}$, the result follows directly from Lemma \ref{lemma:wraparound}. Now suppose $\tfrac{1}{Cn} \le |b| \le \tfrac{C}{n}$. Since $|b|$ is too small in this range, we amplify it slightly. For $k \in \Z^{+}$, Cauchy--Schwarz gives
\[
\| br+b_0 \|_{\R/\Z}^{2} \;\ge\; \frac{1}{k^{2}} \, \| k (b r+b_0) \|_{\R/\Z}^{2}=\frac{1}{k^{2}} \, \| (kb)r+(kb_0) \|_{\R/\Z}^{2}.
\]
Taking $k = \lceil C^{2} \rceil$ and applying Lemma \ref{lemma:wraparound} completes the proof.
\end{proof}

Lemma \ref{lemma:wraparound} is a special case of the following more general result on polynomial sequences.

\begin{lemma}[Wrapping around for polynomial sequences]\label{lemma:wraparound:square}
Let $\delta > 0$ and $d \in \Z^+$ be given. There exists a constant $C>0$ such that the following holds. Let $I \subset \{-n,\dots,n\}$ with $|I| \geq \delta n$. Then, for any $\frac{C}{n} \leq |b| \leq \frac{n^{d-1}}{C}$,
\[
\sum_{r \in I} \Big\| \frac{b r^d + b' r^{d-1} + \dots}{n^{d-1}} \Big\|_{\R/\Z}^2 =\Theta_{\delta}(n),
\]
where $b', b'', \dots \in \R$ are arbitrary.
\end{lemma}

To prove this result, we will use a very nice inverse-type Weyl estimate by Tao \cite[Corollary 5]{Tao}, which is stated below (and proved in Appendix \ref{sect:Dio} for the reader's convenience) for positive density form
(instead of full density form, $I=\{-n,\ldots,n\}$, as in \cite{Tao}).

\begin{lemma}\label{lemma:Weyl:inverse}  Let $I$ be a subset of $\{-n,\dots, n\}$ with $|I|\ge \delta n$ for some $\delta>0$. Let $P(k) = \sum_{i \leq d} \alpha_i k^i$
be a polynomial of degree at most $d \geq 0$, where $\alpha_0,\dots,\alpha_d \in \R/\Z$. If
\[
\frac{1}{n} \Big| \sum_{k \in I} e(P(k)) \Big| \geq \delta,
\]
then there is a natural number $q= O_d(\delta^{-O_d(1)})$ such that
\[
\| q\alpha_i \|_{\R/\Z} =O_d(\delta^{-O_d(1)} n^{-i})
\quad \text{for all } i=0,\dots,d.
\]
\end{lemma}

\begin{proof}(of Lemma \ref{lemma:wraparound:square})  Using the inequality $|\sin(\pi x)| \le 4 \|x\|_{\R/\Z}^2$, we obtain
\[
\cos(2\pi x)=1-2 \sin^2(\pi x) \ge 1-32 \|x\|_{\R/\Z}^2.
\]
Suppose, for contradiction, that
\[
\sum_{r\in I} \Big\| \frac{b r^d + b' r^{d-1} + \dots}{n^{d-1}} \Big\|_{\R/\Z}^2 \le (\delta/64) n.
\]
Then it follows that
\[ 
\sum_{r\in I} \cos\Big(2\pi \,\frac{b r^d + b' r^{d-1} + \dots}{n^{d-1}}\Big) \ge |I|-(\delta/2)n \ge (\delta/2)n.
\]
Thus
\[ 
\frac{1}{n}\Big| \sum_{r\in I} e\Big(\frac{b r^d + b' r^{d-1} + \dots}{n^{d-1}}\Big) \Big| \ge \frac{1}{n}(\delta/2)n \ge \delta/2. 
\]
By Lemma \ref{lemma:Weyl:inverse}, there exists a positive integer $q \le \delta^{{-O(1)}}$ such that
\[
\Big\| q \cdot \frac{b}{n^{d-1}} \Big\|_{\R/\Z} \le \frac{\delta^{{-O(1)}}}{n^d}.
\]
On the other hand, since $C/n\le |b| \le n^{d-1}/C$, we have 
\[
\Big\| q \cdot \frac{b}{n^{d-1}} \Big\|_{\R/\Z}=\Big| q \cdot \frac{b}{n^{d-1}} \Big|>\frac{\delta^{{-O(1)}}}{n^d},
\]
provided that $C$ is sufficiently large in terms of $\delta$ and $d$. This yields a contradiction.
\end{proof}

\section{Discrete settings: proof of Theorem \ref{thm:ILO:array}, Theorem \ref{thm:discrete:3/2}, and Theorem \ref{thm:discrete:5/2}} \label{sect:discrete}


We begin by showing that the bound in Theorem \ref{thm:ILO:array} is sharp.

\begin{lemma}\label{lemma:optimal} The conclusion of Theorem \ref{thm:ILO:array} is optimal, in the sense that 
if $Q$ is a {proper symmetric GAP} of rank $r=O(1)$ and size $O(\rho^{-1}n^{-r/2})$, and $a_{i k} - a_{j k} - a_{i l} + a_{j l} \in Q$ for all quadruples, then 
\[
\sup_{x} \P_\pi (S_{\pi}=x) =\Omega(\rho).
\]
\end{lemma}
\begin{proof} Assume that $Q$ has the form
\[
Q = \{\sum_{s=1}^{r} q_{s} g_{s}: |q_{s}| \le N_{s} \}.
\]
Let $a_{ij}' := a_{ij}-a_{1j}-a_{i1} + a_{11}$. Since $a_{ij}'\in Q$, there exist integers $a_{ij,s}$, bounded by $N_s$, such that
\begin{equation}\label{eqn:a:l}
a_{ij}' = \sum_{s=1}^{r} a_{ij;s} g_{s}.
\end{equation}
Write $S_{\pi}((a_{ij}'))= \sum_i a_{i \pi(i)}'$ for a uniform permutation $\pi$. Then $S_{\pi}((a_{ij}'))=S_{\pi}+c$ with $c=na_{11}-\sum_{i=1}^n(a_{i1}+a_{1i})$. As $c$ is independent of $\pi$,  
\[
\sup_{x} \P_\pi (S_{\pi}((a_{ij}'))=x) =\sup_{x} \P_\pi (S_{\pi}=x).
\]
Consider an array $(b_{ij})$, which plays the role of $a_{ij;s}$ for each $1\le s\le r$. Let
$$\tilde b_{ik} =  b_{ik} - \frac{1}{n} \sum_{l=1}^{n} b_{il} - \frac{1}{n} \sum_{j=1}^{n} b_{jk}  + \frac{1}{n^{2}}\sum_{j,l=1}^na_{jl}.$$
Observe that 
\[
b_{ik}-b_{jk} - b_{il}+b_{jl} = \tilde b_{ik}- \tilde b_{jk} - \tilde b_{il}+ \tilde b_{jl}.
\]
Let $S_{\pi}((b_{ij}))=\sum_i b_{i\pi(i)}$, where $\pi$ is a uniform permutation. Then
$\E S_{\pi}((b_{ij}))=(1/n)\sum_{i,j} b_{ij}$. It is also known \cite[formula (89)]{Gold} that
\begin{equation}\label{eqn:Var}
\Var S_{\pi}((b_{ij}))=\frac{1}{4n^2(n-1)}\sum_{i\ne j;\,k\ne l}
(b_{ik}-b_{jk}-b_{il}+b_{jl})^2.
\end{equation}
Hence, if $b_{ik}-b_{jk}-b_{il}+b_{jl}\in[-N,N]$, then
$\Var(S_{\pi}((b_{ij})))\le nN^2$. By Chebyshev's inequality, with probability at least
$1-16/C^2$,
\begin{equation}\label{eqn:Chebyl}
|S_{\pi}((b_{ij}))-\E S_{\pi}((b_{ij}))|\le (C/4)\sqrt{n}\,N .
\end{equation}

For each fixed $1\le s\le r$, apply this bound to the sequence
$b_{ij}=a_{ij;s}$ from \eqref{eqn:a:l}. Taking the intersection over $s$,
we obtain an event $\CE$ (over the random permutation $\pi$) with
$\P(\CE)\ge 1-16r/C^2$ such that, simultaneously for all $s$,
\[
|S_{\pi}((a_{ij;s}))-\mu_s|\le C\sqrt{n}\,N_s,
\]
where $\mu_s=\E S_{\pi}((a_{ij;s}))$.

Under this event $\CE$, the original sum $\sum_i a_{i\pi(i)}'$ lies in the shifted GAP
\[
\big([\mu_1-C\sqrt{n}N_1,\mu_1+C\sqrt{n}N_1]\cap\Z\big)g_1
+\dots+\big([\mu_r-C\sqrt{n}N_r,\mu_r+C\sqrt{n}N_r]\cap\Z\big)g_r.
\]
By the pigeonhole principle, there exists $x$ in this GAP such that
\[
\P(S_{\pi}(a_{ij}')=x)
\ge
\frac{\P(\CE)}{(3C)^r n^{r/2}\prod_i N_i}
\ge
\rho
\]
for appropriate choices of the constants.
\end{proof}

We next turn to the proof of our inverse theorem by relying on the method of \cite{NgV-Adv}.
\begin{proof} (of Theorem \ref{thm:ILO:array}) The proof consists of several steps.

\textit{Embedding.}
The following theorem (see \cite[Lemma 5.25]{TVbook}, \cite[Theorem 4.3]{NgV-Adv}) allows us to assume that $a_{ij}$ are elements of $\F_p$ for some large prime $p$. 
\begin{theorem}
Let $V$ be a finite subset of a torsion-free additive group $G$. Then, for any integer $k$, there is a map $\phi : V \to \phi(V)$ into some finite subset $\phi(V)$ of the integers $\Z$ such that
$$
v_1 + \cdots + v_i=v'_1 + \cdots + v'_j 
\;\;\Longleftrightarrow\;\;
\phi(v_1) + \cdots + \phi(v_i)=\phi(v'_1) + \cdots + \phi(v'_j)
$$
for all $i,j \le k$. The same holds if we replace $\Z$ by $\F_p$, provided $p$ is sufficiently large, depending on $V$.
\end{theorem}

\textit{Fourier analysis.}
We view elements of $\F_p$ as integers between $0$ and $p-1$. Let $S=\sum_{i=1}^na_{i\pi(i)}$, and suppose 
$$    
\rho=\P(S=x)
$$
for some $x \in \F_p$. Using the standard notation $e_p(z)=\exp(2\pi i z/p)$, we have
$$
\rho = \P(S=x)=\E\frac{1}{p}\sum_{t \in \F_p}e_p(t(S-x))=\E\frac{1}{p}\sum_{t \in \F_p}e_p(tS)e_p(-tx).
$$
Denote by $A$ the multiset $\{a_{i k} - a_{j k} - a_{i l} + a_{j l}:1\le i,j,k,l\le n\}$. Then
$$
\rho \le \frac{1}{p}\sum_{t\in \F_p}|\E e_p(tS)| \le \frac{1}{p}\sum_{t \in \F_p}\exp \Big(-\frac{1}{2n^3}\sum_{a\in A}\|\frac{at}{p}\|_{\R/\Z}^2\Big),
$$
where the second inequality follows from Corollary \ref{cor:Roos}. 

\textit{Level sets.}
For $m\in \BBN$, define $L_m:=\big\{t \in \F_p: \sum_{a\in A}\|\frac{at}{p}\|_{\R/\Z}^2\in [4(m-1)n^3,4mn^3]\big\}$. Then
$$
n^{-C}\le \rho \le \frac{1}{p}\sum_{t \in \F_p}\exp \Big(-\frac{1}{2n^3}\sum_{a\in A}\|\frac{at}{p}\|_{\R/\Z}^2\Big) \le \frac{1}{p}+\frac{1}{p}\sum_{m\ge 1}\text{exp}(-2(m-1))|L_m|.
$$
Since $\sum_{m\ge 1}\text{exp}(-m)<1$ and $p\ge n^{2C}$, there exists a level set $L_m$ such that 
$$
|L_m|\exp(-m+2)\ge\rho p.
$$ 
As $\rho \ge n^{-C}$ and $|L_m|\le p$, it follows that $m=O(\log n)$.

\textit{Double counting and the triangle inequality.}
We have 
$$
\sum_{a\in A}\sum_{t \in L_m}\|\frac{at}{p}\|_{\R/\Z}^2=\sum_{t \in L_m}\sum_{a\in A}\|\frac{at}{p}\|_{\R/\Z}^2\le  (4mn^3)|L_m|.
$$

Let $n^{\eps} \le n' \le n$. By averaging, at least a $(1-n'/n)$-fraction of $a\in A$ (i.e. $(n^{4}-n'n^{3})$ quadruples) satisfy
$$
\sum_{t \in L_m}\|\frac{at}{p}\|_{\R/\Z}^2\le\frac{4m}{n'}|L_m|.
$$
Denote this set by $A'$. We will show that $A'$ is a dense subset of a proper GAP.

Applying the triangle inequality to the norm $\|\cdot\|_{\R/\Z}$, we obtain, for any $a \in l A'$,
\begin{equation}\label{eq:sumset}
\sum_{t \in L_{m}}\|\frac{a t}{p}\|^2_{\R/\Z} \leq l^{2} \frac{4m}{n'}|L_{m}|.
\end{equation}

\textit{Dual sets.} Define $L_{m}^{*}:=\big\{a|\sum_{t \in L_{m}}\|\frac{a t}{p}\|^2_{\R/\Z} \leq \frac{1}{40}|L_{m}|\big\}$. We claim
\begin{equation}\label{eq:dual}
|L_{m}^{*}| \leq \frac{8 p}{|L_{m}|}.
\end{equation}
Indeed, set $T_{a}=\sum_{t \in L_{m}} \cos \frac{2 \pi a t}{p}$. Using $\cos 2 \pi z \geq 1-20\|z\|^2_{\R/\Z}$, we see that for $a \in L_{m}^{*}$
$$
T_{a} \geq \sum_{t \in L_{m}}(1-20\|\frac{a t}{p}\|^2_{\R/\Z}) \geq \frac{1}{2}|L_{m}|.
$$
On the other hand, since $\sum_{a \in \F_p} \cos \frac{2 \pi a z}{p}=p \mathbf{1}_{z=0}$,
$$
\sum_{a \in \F_p} T_{a}^{2} \leq 2 p|L_{m}|.
$$
The bound \eqref{eq:dual} follows by averaging.

Set $k=\sqrt{\frac{n'}{160m}}$. By \eqref{eq:sumset}, we have $\bigcup_{l=1}^k lA' \subset L_m^*$. Setting $A''=A'\cup\{0\}$, this implies $kA''\subset L_m^*\cup\{0\}$. Hence
\begin{equation}\label{eq:critical}
|kA''|=O\Big(\frac{p}{|L_m|}\Big)=O(\rho^{-1}e^{-m+2}).
\end{equation}

The ambient field $\F_p$ is no longer important, so we may view the $a_{ij}$ as integers. We now invoke the following long-range inverse theorem (see \cite[Theorem 3.2]{NgV-Adv}).

\begin{theorem}\label{thm:Freiman-inverse}
Let $\gamma > 0$ be a constant. Assume that $X$ is a subset of a torsion-free group such that $0 \in X$ and $|kX| \le k^{\gamma} |X|$ for some integer $k \ge 2$ that may depend on $|X|$. Then there is a proper symmetric GAP $Q$ of rank $r = O(\gamma)$ and cardinality $O_{\gamma}(k^{-r} |kX|)$ such that $X \subset Q$.
\end{theorem}

Since 
$$k=\Omega\Big(\sqrt{\frac{n'}{m}}\Big)=\Omega\Big(\sqrt{\frac{n'}{\log n}}\Big), \quad \rho^{-1} \leq n^{C} \le k^{2 C/\eps+1},
$$
\eqref{eq:critical} allows us to apply Theorem \ref{thm:Freiman-inverse} with $\gamma=2C/\eps+1$ and $X$ the set of distinct elements of $A^{''}$ (note that $k X=k A^{''}$ for $k \geq 2$). Hence $X$ is contained in a proper symmetric GAP $Q$ of rank $r = O_{C, \eps}(1)$ and size
$$
O_{C, \eps}(k^{-r}|kX|)= O_{C, \eps}(k^{-r}|kA^{''}|)= O_{C, \eps}\Big(\rho^{-1}e^{-m+2}(\sqrt{\frac{n'}{m}})^{-r}\Big)= O_{C,\eps}(\rho^{-1}(n')^{-r/2}),
$$
which completes the proof.
\end{proof}

To complete the section we quickly deduce Theorem \ref{thm:discrete:1/2}, Theorem \ref{thm:discrete:3/2} and Theorem \ref{thm:discrete:5/2}.

\begin{proof}(of Theorem \ref{thm:discrete:1/2}) Assume for contradiction that $\rho:=\sup_{x} \P\big( \sum_{i=1}^n w_i v_{\pi(i)} = x \big) \ge A n^{-1/2}$ for some sufficiently large constant $A$. Set $\eps=.1$ and $C =1/2$. By Theorem \ref{thm:ILO:prod}, there is a GAP $Q$ of rank $r$ and size $O(\rho^{-1}n^{-r/2})$ that contains a $(1-\eps)$-portion of the set of quadruples $\{(w_{i}-w_{j})(v_{k} -v_{l}): 1\le i,j,k,l \le n\}$. This set must have at least $\Theta(1)$ (distinct) elements, because by assumption there are $\Theta_{c}(n^{2})$ pairs $i,j$ where $w_{i}-w_{j} \neq 0$ and $\Theta_{c}(n^{2})$ pairs $k,l$ where $v_{k}-w_{l} \neq 0$.
Since $Q$ has rank at least $1$, we have $\Omega(n)=|Q|=O(\rho^{-1}n^{-1/2}) = O(n/A)$, a contradiction if $A$ is sufficiently large.
\end{proof}

\begin{proof}(of Theorem \ref{thm:discrete:3/2}) Assume for contradiction that $\rho:=\sup_{x} \P\big( \sum_{i=1}^n w_i v_{\pi(i)} = x \big) \ge A n^{-3/2}$ for some sufficiently large constant $A$. Set $\eps=.1$ and $C =
3/2$. By Theorem \ref{thm:ILO:prod}, there is a GAP $Q$ of rank $r$ and size $O(\rho^{-1}n^{-r/2})$ that contains a $(1-\eps)$-portion of the set of quadruples $\{(w_{i}-w_{j})(v_{k} -v_{l}): 1\le i,j,k,l \le n\}$. This set must have at least $\Theta(n)$ (distinct) elements, because by assumption there are $\Theta_{c}(n^{2})$ pairs $i,j$ where $w_{i}-w_{j} \neq 0$, and on average, each such pair $(i,j)$ is associated with $\Theta(n^{2})$ pairs $(k,l)$. Among these pairs, we just chose $\Theta(n)$ pairs of the form $(k_{0},l_{1}),\dots, (k_{0},l_{m}), m = \Theta(n)$. 
Since $Q$ has rank at least $1$, we have $\Omega(n)= |Q|=O(\rho^{-1}n^{-1/2}) = O(n/A)$, a contradiction.
\end{proof}

\begin{proof}(of Theorem \ref{thm:discrete:5/2}) Assume for contradiction that $\rho:=\sup_{x} \P\big( \sum_{i=1}^n w_i v_{\pi(i)} = x \big) \ge A n^{-5/2} \log n$ for some sufficiently large constant $A$. Set $\eps=.1$ and $C =5/2$. By Theorem \ref{thm:ILO:prod}, there is a GAP $Q$ of rank $r$ and size $O(\rho^{-1}n^{-r/2})$ that contains a $(1-\eps)$-portion of the set of quadruples $\{(w_{i}-w_{j})(v_{k} -v_{l}): 1\le i,j,k,l \le n\}$. Denote this set by $S$. We claim that $S$ must have at least  $\Theta(n^{2}/\log n)$ elements. Since $Q$ has rank at least $1$, this would imply $\Omega(n^2/\log n)=|Q|=O(\rho^{-1}n^{-1/2})=O(n^2/A\log n)$, a contradiction.

It remains to prove $|S|=\Omega(n^2/\log n)$. Let $A=\{v_1,\ldots,v_n,w_1,\ldots,w_n\}$. For $x\in A-A$, let $r(x)$ denote the number of representations $a-b=x$ with $(a,b)\in A^2$.  
By \cite[Theorem 3]{BW},\footnote{The statement in \cite{BW} is given for $A \subset \Z$, but the proof works verbatim for $A \subset \R$.} we have
\[ \sum_{z}\Big|\sum_{xy=z}r(x)r(y)\Big|^2=\sum_{xy=x'y'}r(x)r(x')r(y)r(y')=O(|A|^6\log |A|).
\]
Applying the Cauchy-Schwarz inequality gives
\[
\frac{1}{|S|}\Big(\sum_{xy\in S}r(x)r(y)\Big)^2=O(|A|^6\log |A|).
\]
Moreover, $\sum_{xy\in S}r(x)r(y)$ counts at least the number of selected quadruples, namely $(1-\eps)n^4$. Therefore, $|S|=\Omega(n^8/|A|^6\log |A|)=\Omega(n^2/\log n)$, as claimed.
\end{proof}

\section{Continuous setting: proof of Theorems \ref{thm:cont:1} and \ref{thm:cont:3/2}}\label{sect:linear}

We will use the following simple fact.

\begin{fact}\label{fact:R}
Let $ I \subset [n]$ with $ |I| \geq \delta n $ for some constant $ \delta > 0 $. Then, there exists $ \CR \subset \{-n, \dots, n\} $ of size $ |\CR|= \Omega_{\delta}(n)$ such that for every $ r \in \CR $, there are $ \Omega_{\delta}(n)$ pairs $x,y \in I$ with $x - y = r $.
\end{fact}

Throughout this section, let $\mathcal{R} \subset \{-n, \dots, n\}$ be a set of size $|\mathcal{R}| = \Theta_{\delta}(n)$, as defined in Fact~\ref{fact:R}. Let $C$ and $A$ be positive constants, with $C$ sufficiently large relative to $\delta$ and $\varepsilon$, and $A$ sufficiently large relative to $C$.

We will break down the proof of Theorem \ref{thm:cont:1} into two parts, the uniform bound \eqref{eqn:1:u} and the $L$-dependent bound \eqref{eqn:1:L}.

\begin{proof}[Proof of \eqref{eqn:1:u} from Theorem \ref{thm:cont:1}]
Let 
\[
S = n\sum_{i} w_{i} v_{\pi(i)}, \quad \text{and} \quad X = S - L n.
\]
Using Esseen's estimate, we can write
\begin{equation}\label{eqn:kt}
\P(|X| \le 1) 
=O\Big( \Big| \int_{|t| \le 1} k(t)\varphi_{X}(t) \, dt \Big| \Big)
=O\Big( \int_{|t| \le 1} |\varphi_{X}(t)| \, dt\Big),
\end{equation}
where $k(t) = 1_{[-1/2, 1/2]} * 1_{[-1/2, 1/2]}(t)$ (see for instance \cite{NVsur}), and $\varphi_{X}(t)$ is the characteristic function of $X = S - L n$:
\[
\varphi_{X}(t) = \mathbb{E} \, e^{i t (S - L n)} = \mathbb{E} \, e^{i t S} e^{- i L n t} = \varphi_{S}(t) e^{- i L n t}.
\]
Thus,
\[
|\varphi_{X}(t)| = |\mathbb{E} \, e^{it S}| = |\varphi_{S}(t)|.
\]

We first note that by Corollary \ref{cor:Roos},
\[
|\varphi_{S}(2\pi t)| \le \exp\Big\{ - \frac{1}{2n^{3}} \sum_{i,j,k,l} \| tn (w_{i} - w_{j}) (v_{k} - v_{l}) \|_{\R/\Z}^{2} \Big\}.
\]

The exponent on the right-hand side can be bounded from below by
\[
\frac{c_{\delta}}{n^{2}} \sum_{\substack{1 \le i,j \le n \\ r\in \CR}} \| t (w_{i} - w_{j}) r \|_{\R/\Z}^{2},
\]
and so
\begin{equation}\label{eqn:int:1}
\P(|X| \le 1) =O\Big(\int_{|t| \le 1} |\varphi_{S}(t)| \, dt\Big)
=O\Big(\int_{|t| \le 1} \exp\Big\{ -\frac{c_{\delta}}{n^{2}} \sum_{\substack{1 \le i,j \le n \\ r\in \CR}} \| t (w_{i} - w_{j}) r \|_{\R/\Z}^{2} \Big\} dt\Big).
\end{equation}

We will split the integral depending on whether $|t| \le (\sqrt{A\log n})/n$ or $(\sqrt{A\log n})/n \le |t| \le 1$. 

\underline{\bf Large $|t|$.} Assume that $$\frac{\sqrt{A\log n}}{n} \le |t| \le 1.$$

We first discard those $(i,j)$ for which $|w_{i} - w_{j}|$ is
smaller than $1/\sqrt{n}$. Let
\[
\CG = \{ (i,j) : |w_{i} - w_{j}| \ge 1/\sqrt{n} \}.
\]
Since $\sum_{1\le i,j\le n}(w_{i}-w_{j})^2=2n$, we see that
\[
n \le \sum_{(i,j) \in \CG} (w_{i} - w_{j})^{2} \le 2n.
\]
For $0 \le k \le \log \Big( \frac{\sqrt{n}}{A \sqrt{\log n}} \Big)+1$ (so that $\frac{2^{k-1}}{\sqrt{n}} \le \frac{1}{A\sqrt{\log n}}$ and $\frac{2^{k}}{\sqrt{n}}\ge \frac{1}{\sqrt{n}}$), let $\CG_{k}$ be the set of pairs $(i,j)$ such that
\[
D_{k-1} := \frac{2^{k-1}}{\sqrt{n}} < |w_{i} - w_{j}| \le \frac{2^{k}}{\sqrt{n}} =: D_{k}.
\]
Then $\CG=\bigcup_{k}\CG_{k}$ and
\[
n \le \sum_{k} D_{k}^{2} |\CG_{k}| \le 8n.
\]

Now for each fixed pair $(i,j)\in \CG_{k}$, we consider 
the sum $\sum_{r\in \CR} \| t (w_{i} - w_{j}) r \|_{\R/\Z}^{2}$.

If $|t|D_{k} \le \frac{1}{Cn}$, then $|t (w_{i} - w_{j}) r|\le \frac{1}{Cn}n=\tfrac{1}{C}<1$, so 
\[
\sum_{r \in \CR} \| t (w_{i} - w_{j}) r \|_{\R/\Z}^{2} = \sum_{r\in \CR} |t (w_{i} - w_{j}) r|^{2} =\Theta_{\delta}(t^{2} n^{3} D_k^{2}).
\]
On the other hand, if $|t|D_{k}>\frac{1}{Cn}$, then $\tfrac{1}{2Cn}< |t(w_i-w_j)| \le \frac{1}{A\sqrt{\log n}} \le \frac{1}{2C}$, so by Corollary~\ref{cor:wraparound}, we have
\[
\sum_{r} \| t (w_{i} - w_{j}) r \|_{\R/\Z}^{2} =\Theta_{C}(n).
\]
Thus, we obtain
\[
\sum_{(i,j)\in \CG,r\in \CR} \| t (w_{i} - w_{j}) r \|_{\R/\Z}^{2} =\Omega_{\delta, C}(1)\cdot \Big(\sum_{k:\, |t|D_{k} \le 1/(Cn)} t^{2} n^{3} D_{k}^{2} |\CG_{k}| + \sum_{k:\, |t|D_{k}  > 1/(Cn)} n |\CG_{k}|\Big).
\]

To further estimate the right-hand side, we divide into two cases.

{\bf Case 1:} $\sum_{k:\, |t| \le 1/(CnD_{k})} D_{k}^{2} |\CG_{k}| \ge (1/10)n$. 

Since $t^2 \ge \frac{A \log n}{n^2}$,
\[
\sum_{k:\, |t| \le 1/(CnD_{k})} t^{2} n^{3} D_{k}^{2} |\CG_{k}|=\Omega(A n^{2} \log n).
\]

{\bf Case 2:} $\sum_{k:\, |t| > 1/(CnD_{k})} D_{k}^{2} |\CG_{k}| > (3/4 - 1/10) n$. 

Since $D^2_{k} \le \frac{1}{A^2 \log n}$, we get
\[
\sum_{k:\, |t| > 1/(CnD_{k})} n|\CG_{k}| =\Omega(A^{2} n^{2} \log n) =\Omega(A n^{2}\log n).
\]

In both cases, we always have
\[
\sum_{i,j,r} \| t (w_{i} - w_{j}) r \|_{\R/\Z}^{2} =\Omega_{\delta,C}(A n^{2} \log n).
\]
Therefore, assuming that $A$ is sufficiently large relative to $\delta$ and $C$, for $\frac{\sqrt{A\log n}}{n} \le |t| \le 1$, we have
\begin{equation}\label{eqn:phiS:larget} 
\frac{c_{\delta}}{n^{2}}\sum_{i,j,r} \| t (w_{i} - w_{j}) r \|_{\R/\Z}^{2} \ge 2\sqrt{A} \log n, \qquad |\varphi_{S}(t)| \le n^{-2\sqrt{A}}.
\end{equation}

\underline{\bf Small $|t|$.} It remains to consider 
\[
|t| \le \frac{\sqrt{A\log n}}{n}.
\]

Since $|w_{i} - w_{j}| \le \frac{1}{A \sqrt{\log n}}$, we have $|t(w_{i}-w_{j})r| \le \frac{1}{\sqrt{A}}<1$, which implies $\| t (w_{i} - w_{j}) r \|_{\R/\Z} = | t (w_{i} - w_{j}) r |$.
Thus, we obtain
\[
\frac{c_{\delta}}{n^{2}} \sum_{i,j,r} \| t (w_{i} - w_{j}) r \|_{\R/\Z}^{2} 
= \frac{c_{\delta}}{n^{2}} \sum_{i,j,r}(t (w_{i} - w_{j}) r)^{2} =\Theta_{\delta}(t^{2} n^{2}).
\]
Therefore,
\[
\int_{|t| \le (\sqrt{A\log n})/ n} |\varphi_{S}(t)| \, dt \le \int_{|t| \le (\sqrt{A\log n})/ n} \exp(-\Theta(t^{2} n^{2})) \, dt 
=O\Big(\frac{1}{n} \int_{\R} \exp(-\Theta(x^{2})) \, dx\Big) 
=O\Big(\frac{1}{n}\Big).
\]
\end{proof}

Next, we modify the above approach to deal with the extra factor $e^{-c L^{2}}$, paying more attention to how the characteristic functions depend on $L$. We will mainly establish the following bound:
\begin{equation}\label{eqn:modification:1}
\P\Big( \Big| \sum_{i} w_{i} v_{\pi(i)} - L \Big| \le \frac{1}{n} \Big) 
= O\left( \max\Big\{ n^{-\sqrt{A}}, \ \frac{1}{n} e^{-\Theta(L^{2})} \Big\} \right).
\end{equation}
In fact, for large $L$, we can apply Lemma \ref{lemma:deviation}: if 
\[
L \ge \frac{2}{\sqrt{C_{0}}} \sqrt{\log n},
\]
then
\[
\P\Big( \Big| n\sum_{i=1}^{n} v_{i} w_{\pi(k)} -n \Big( \sum_{i=1}^{n} v_{i} \Big) \overline{\BMw} \Big| 
\ge L n \cdot \sigma(\BMw) \Big) 
\le C_{0} e^{-C_{0} L^{2}} 
\le \frac{1}{n} e^{-C_{0} L^{2} / 2}.
\]
Thus, for Theorem \ref{thm:cont:1} it suffices to assume
\begin{equation}\label{eqn:L<}
L < \frac{2}{\sqrt{C_{0}}} \sqrt{\log n}.
\end{equation}

\begin{proof}(of \eqref{eqn:1:L} (or more precisely  \eqref{eqn:modification:1}) of Theorem \ref{thm:cont:1}) As mentioned above, assume $L< (2 /\sqrt{C_{0}})\sqrt{ \log n}$. We first note that the treatment for large $t$ in the above proof (i.e., $|\varphi_S(t)|\le n^{-2\sqrt{A}}$ for $(\sqrt{A\log n})/n \le |t| \le 1$) can be extended all the way to 
$(\sqrt{A\log n})/n\le |t| \le \sqrt{A\log n}$.

\underline{\bf Very Large $|t|$.} We assume now that 
$$1 \le |t| \le  \sqrt{A\log n}.$$

We first throw away those $|w_{i}-w_{j}|$ that are smaller than $1/\sqrt{n}$, and set 
$$\CG=\{(i,j): |w_{i}-w_{j}| \ge 1/\sqrt{n}\}.$$
Then we have 
$$\sum_{(i,j)\in \CG} (w_{i}-w_{j})^{2} \ge n.$$
Since $|w_{i}-w_{j}| \le 1/A\sqrt{\log n}$, it follows that 
$$|\CG| \ge A^{2} n \log n.$$
Now we consider the sum $\sum_{i,j,r} \|t (w_{i}-w_{j})r\|_{\R/\Z}^{2}$, where $(i,j) \in \CG$. As $1/\sqrt{n}  \le |w_{i}-w_{j}| \le 1/A\sqrt{\log n}$ and $1 \le |t| \le  \sqrt{A\log n}$, we have 
$\frac{1}{\sqrt{n}}  \le |t (w_{i}-w_{j})| \le \frac{1}{\sqrt{A}}<\frac{1}{C}$.
Thus, by Lemma \ref{lemma:wraparound}, it follows that 
$$  \sum_{r \in \CR} \|t (w_{i}-w_{j}) r\|_{\R/\Z}^{2} =\Theta_{\delta}(n).$$
This implies 
$$\sum_{i,j,r} \|t (w_{i}-w_{j})r\|_{\R/\Z}^{2} =\Omega_{\delta}(n |\CG|) =\Omega_{\delta}(A^2n^{2}\log n).$$
Therefore, for $1\le |t|\le \sqrt{A\log n}$, we also 
have
\begin{equation}\label{eqn:L:large:t}
|\varphi_{S}(t)| \le n^{-A}.
\end{equation}

Our plan in the proof of \eqref{eqn:1:L} is to replace $k(t)$ in \eqref{eqn:kt} by some smoother function that can be extended holomorphically to $\C$. Our starting point is that
\begin{equation}\label{eqn:identity:1d}
\int_{\R} e^{-\pi t^{2}} e^{i t x} dt = e^{-\pi x^{2}/2}.
\end{equation}
Hence
$$\E \int_{\R} e^{-\pi t^{2}} e^{i t X} dt = \E e^{-\pi X^{2}/2}.$$

For any $K$ (noting here and later that the integrals are real valued due to the symmetry of the range of $t$),
$$-\int_{|t|\ge K} e^{-\pi  t^{2}}dt  \le \int_{|t|>K} e^{-\pi  t^{2}} e^{i t x} dt  \le \int_{|t|\ge K} e^{-\pi  t^{2}} dt \le e^{-\Theta(K^{2})}.$$ 
Thus, for sufficiently large $A$, with $X=S-Ln$
$$|\E \int_{|t| \ge \sqrt{A\log n}} e^{-\pi  t^{2}} e^{i t X} dt| \le \int_{|t| \ge \sqrt{A\log n}} e^{-\pi  t^{2}} dt  \le n^{-2\sqrt{A}}.$$
We thus have 
\begin{align}\label{eqn:X<1:1}
\P(|X|\le 1) \le e^{\pi/2}  \E e^{-\pi X^{2}/2} &\le e^{\pi/2} \Big[   \E \int_{|t| \le \sqrt{A\log n}} e^{-\pi  t^{2}} e^{i t X} dt + n^{-2\sqrt{A}} \Big] \nonumber \\
&\le  e^{\pi/2} \int_{|t| \le \sqrt{A\log n}} e^{-\pi  t^{2}} \E e^{i t X} dt + e^{\pi/2} n^{-2\sqrt{A}}.
\end{align}
At this point, if $\P(|X| \le 1) \le 2e^{\pi /2} n^{-2\sqrt{A}}$, then there is nothing to prove.
In the remaining case, from \eqref{eqn:X<1:1} we have reached that  
$$\P(|X|<1) \le 2 e^{\pi/2} \int_{|t| \le \sqrt{A\log n}} e^{-\pi  t^{2}} \varphi_{X}(t) dt.$$
Combining \eqref{eqn:L:large:t} with \eqref{eqn:phiS:larget} from the proof of \eqref{eqn:1:u}, we get: for $(\sqrt{A\log n})/n \le |t| \le \sqrt{A\log n}$, 
$$|\varphi_{X}(t)| \le n^{-2\sqrt{A}}.$$
It follows that
$$
\int_{(\sqrt{A\log n})/n\le |t| \le \sqrt{A\log n}} e^{-\pi  t^{2}} \varphi_{X}(t) dt \le (2\sqrt{A\log n})n^{-2\sqrt{A}} = o(n^{-\sqrt{A}}).
$$
It remains to bound $ \int_{|t| \le (\sqrt{A\log n})/n} e^{-\pi  t^{2}} \varphi_{X}(t) dt$.
We will decompose the integral to
\begin{align*}
\int_{|t| \le (\sqrt{A\log n})/n} e^{-\pi  t^{2}} \varphi_{X}(t) dt & =  \int_{  L/n< |t| \le (\sqrt{A\log n})/n} e^{-\pi  t^{2}} \varphi_{X}(t) dt +\int_{|t| \le L/n} e^{-\pi  t^{2}} \varphi_{X}(t) dt.
\end{align*}
For the first integral, recall from the proof of \eqref{eqn:1:u} in the case of ``small $|t|$'' that for $|t|\le (\sqrt{A\log n})/n$, $|\varphi_{X}(t)| \le \exp(-\Theta (t^{2}n^{2}))$, so
$$\int_{  L/n< |t| \le (\sqrt{A\log n})/n} e^{-\pi  t^{2}} \varphi_{X}(t) dt \le \frac{1}{n}\int_{L<|x|} e^{-\Theta(x^{2})} dx \le  \frac{1}{n}e^{-\Theta(L^{2})}.$$

It remains to work with the second integral 
\[
\int_{|t| \le L/n} e^{-\pi t^{2}} \varphi_{X}(t) \, dt,
\]
which, by the change of variable \(x = nt\), can be rewritten as
\[
\frac{1}{n} \int_{|x| \le L} e^{-\pi x^{2}/n^{2}} \, \varphi_{X/n}(x) \, dx,
\]
where
\[
\varphi_{X/n}(x) = \mathbb{E} \, e^{i x X/n} = \mathbb{E} \, e^{i x (S/n - L)} = e^{-i xL}\varphi_{S/n}(x).
\]
Note that if we extend \(\varphi_{S/n}(t)\) (or  \(\varphi_{X/n}(t)\)) to the complex plane, since \(S\) is bounded (for each fixed \(n\)), we obtain a holomorphic function
\[
\varphi_{S/n}(z) =  \mathbb{E} \, e^{i z (S/n)}.
\]
Let
\[
h(t) = e^{-\pi t^{2}/n^{2}} \, \varphi_{S/n}(t).
\]
This function can be extended holomorphically to
\[
h(z) = e^{-\pi z^{2}/n^{2}} \, \varphi_{S/n}(z).
\]
Since $|\mathbb{E} Y| \le \mathbb{E}|Y|$ for any complex-valued random variable $Y$, inequality \eqref{eqn:mgf:1} gives  \begin{equation}\label{eqn:CS:1} 
|\varphi_{S/n}(z)| = \Big|\mathbb{E} \, e^{i z S/n} \Big|= \Big|\mathbb{E} \, e^{i (t+ i s) S/n} \Big|= \Big|\mathbb{E} \, e^{i (t/n) S} e^{-s (S/n)} \Big|\le \mathbb{E} \, e^{-s (S/n)} \le C_{0}'e^{C_{0}' s^{2}}.
\end{equation}

Now we establish the bound \(O(\frac{\exp(-c L^{2})}{n})\) for some sufficiently small constant \(c\) (such as \(c = 1/(4C_{0}')\)). Write 
$$ \frac{1}{n} \int_{|x| \le L} e^{-\pi x^{2}/n^{2}} \, \varphi_{X/n}(x) dx=\frac{1}{n} \int_{|x| \le L}  e^{-i xL} e^{-\pi x^{2}/n^{2}} \, \varphi_{S/n}(x)dx= \frac{1}{n} \int_{|x| \le L} e^{-i xL} h(x)dx.$$
First, using contour integration, we pass to the line \(\R - i c L \):
\begin{align}\label{eqn:contour}
\frac{1}{n} \int_{|x| \le L} e^{-i x L} h(x) \, dx
&= \Re \Big[ \frac{1}{n} \int_{\substack{z \in \R - i c L \\ |\Re(z)| \le L}} e^{-i z L} h(z) \, dz \Big] \nonumber \\
&= \frac{1}{n} \Re \Big[ \int_{|t| \le L} e^{-i (t - i c L) L} h(t - i c L ) \, dt \Big] \nonumber \\
&= \frac{e^{-c L^{2}}}{n} \Re \Big[ \int_{|t| \le L} e^{-i t L} h(t - i c L ) \, dt \Big],
\end{align}
where it is crucial to notice that the first integral is real-valued because \(h(-t) = \overline{h(t)}\), and the real parts of the integrals (with opposite orientation) on the lines \(\Re(z) = -L\) and \(\Re(z) = L\) cancel each other.  
More specifically,
\[
\Re \int_{z = -L - i t,\, 0 \le t \le c L} e^{-i z L} h(z) \, dt
= \Re \int_{z = L - i t,\, 0 \le t \le c L} e^{-i z L} h(z) \, dt
\]
since they are conjugates of each other.  
This follows from \(S \in \R\) and
\[
h(-x + i y) = e^{-\pi (-x + i y)^{2}/n^{2}} \, \mathbb{E} \, e^{i (-x + i y) S / n}
= e^{-\pi (x^{2} - y^{2} - 2 i x y)/n^{2}} \, \mathbb{E} \, e^{-i x S / n} e^{-y S / n},
\]
while
\[
h(x + i y) = e^{-\pi (x + i y)^{2}/n^{2}} \, \mathbb{E} \, e^{i (x + i y) S / n}
= e^{-\pi (x^{2} - y^{2} + 2 i x y)/n^{2}} \, \mathbb{E} \, e^{i x S / n} e^{-y S / n}.
\]

Thus, by \eqref{eqn:contour},
\[
\P(|X| \le 1) \le \frac{e^{-c L^{2}}}{n} \Big| \int_{|t| \le L} e^{-i t L} h(t - i c L) \, dt \Big|.
\]
Note that
\[
\Big| e^{-\pi (t - i c L)^{2}/n^{2}} \Big|
= e^{-\pi (t^{2} - c^{2} L^{2})/n^{2}} = \Theta(1),
\quad \text{as } n \to \infty \text{ and } |t| \le L = O(\sqrt{\log n}),
\]
and by \eqref{eqn:CS:1},
\[
|\varphi_{S/n}(t - i c L)| \le C_{0}' e^{C_{0}' c^{2} L^{2}}.
\]

Putting these together, by choosing \(c = 1/(4 C_{0}')\), we obtain (in the case \(\P(|X| \le 1) \ge e^{\pi/2}n^{-A/2}\))
\[
\P(|X| \le 1) \le \frac{e^{-c L^{2}}}{n} \cdot e^{C_{0}' c^{2} L^{2}} \cdot 2L
=O\Big(\frac{e^{-L^{2}/(16 C_{0}')} \, L}{n}\Big) =O\Big(\frac{e^{-\Theta(L^{2})}}{n}\Big).
\]
\end{proof}

\begin{remark}\label{rmk:optimal}
We observe that the assumption $|w_i - w_j| \le \frac{1}{A \sqrt{\log n}}$ in Theorem~\ref{thm:cont:1} cannot be relaxed to $|w_i - w_j| = O(\frac{1}{\sqrt{\log n}})$ when using only the characteristic function method (i.e., relying solely on Theorem~\ref{thm:Roo}). Indeed, suppose we partition $[n]$ into two disjoint sets $I \cup J = [n]$ with $|I| = 2c \log n$ for some constant $c > 0$. Let $w_i = \frac{1}{\sqrt{2c \log n}}$ and $- \frac{1}{\sqrt{2c \log n}}$ for half of the $i \in I$ respectively, and $w_j = 0$ for all $j \in J$. Then,
\begin{align*}
\frac{1}{n^2} \sum_{i,j,r} \|t(w_i - w_j) r\|_{\R/\Z}^2 
&= \frac{1}{n^2} \Big[ \sum_{(i,j) \in I \times J} \sum_r \|t(w_i - w_j) r\|_{\R/\Z}^2 + \sum_{(i,j) \in I \times I} \sum_r \|t(w_i - w_j) r\|_{\R/\Z}^2 \Big] \\
&\le \frac{1}{n^2} \Big(2 c \log n \cdot n + (c \log n)^2 \Big) \cdot n \le 3c \log n.
\end{align*}
Therefore,
\[
\int_{|t| \le 1} \exp\Big( - \frac{1}{n^2} \sum_{i,j,r} \|t(w_i - w_j) r\|_{\R/\Z}^2 \Big) dt 
\ge \int_{|t| \le 1} e^{-3c \log n} \, dt 
= \Theta(n^{-3c}),
\]
which is too large to obtain meaningful decay.
\end{remark}

We next move to a finer scale, proving Theorem \ref{thm:cont:3/2}.

\begin{proof}[Proof of \eqref{eqn:3/2:u} of Theorem \ref{thm:cont:3/2}]
Let $\Delta>0$ be a constant chosen sufficiently large in terms of $\eps$ and $\delta$.\footnote{Only the small-$|t|$ regime considered below requires $\Delta$ to be a constant. The other regimes hold for \emph{any} $\Delta = \Delta(n)$ sufficiently large with respect to $\eps$ and $\delta$.} Define
\[
S' = \sum_{i} \frac{n^{3/2}}{\Delta}w_{i} v_{\pi(i)} =\frac{n^{1/2}}{\Delta} S
\quad\text{and}\quad
X' = S' - \frac{n^{3/2}}{\Delta}L.
\]
To prove \eqref{eqn:3/2:u}, it suffices to show $\P(|X'| \le 1)=O_{\eps,\delta,\Delta,A}(\frac{1}{n^{3/2}})$.
Using Esseen's estimate together with Corollary \ref{cor:Roos}, we can write
\[
\sup_{x} \P(|S' - x| \le 1)
= O\Big(\int_{|t| \le 1} |\varphi_{S'}(t)| \, dt\Big)
=O\Big(\int_{|t| \le 1} \exp\Big\{- \frac{1}{2n^{3}} \sum_{i,j,k,l} \Big\|\frac{t n^{3/2}}{\Delta} (w_{i}-w_{j})(v_{k}-v_{l})\Big\|_{\R/\Z}^{2}\Big\} \, dt\Big).
\]
Recalling that $v_i = i/n$ for $i\in I$, and following the argument from the proof of \eqref{eqn:1:u}, the exponent on the right-hand side can be bounded from below by
\[
\frac{c_{\delta}}{n^{2}} \sum_{\substack{1 \le i,j \le n \\ r\in \CR}}
\Big\| \frac{t n^{1/2}}{\Delta} (w_{i} - w_{j}) r \Big\|_{\R/\Z}^{2},
\]
and so
\begin{equation}\label{eqn:int:3/2}
\P(|X'| \le 1)  
= O\Big(\int_{|t| \le 1} \exp\Big\{ -\frac{c_{\delta}}{n^{2}} \sum_{\substack{1 \le i,j \le n \\ r\in \CR}} \Big\| \frac{t n^{1/2}}{\Delta}(w_{i} - w_{j}) r \Big\|_{\R/\Z}^{2} \Big\} dt\Big).
\end{equation}

We note that this differs from \eqref{eqn:int:1} in that we have an extra factor of $n^{1/2}$ in the exponent.  
As such, our case analysis for $t$ will be different, and we will need more information in addition to \eqref{eqn:separation} about the $w_{i}$ (as assumed in Theorem \ref{thm:cont:3/2}).

We discard pairs $(i,j)$ for which $|w_{i}-w_{j}|$ is much smaller than $1/\sqrt{n}$, and set
\[
\CG = \{ (i,j) : |w_{i}-w_{j}| \ge \eps / 2\sqrt{n} \}.
\]
Since there is no interval of length $\eps/\sqrt{n}$ containing at least $(1-\eps)n$ elements from $w_{1},\dots,w_{n}$, the number of pairs $(i,j)$ such that $|w_{i}-w_{j}| < \eps/2\sqrt{n}$ is at most $(1-\eps)n^{2}$.
Thus, we have
\begin{equation}\label{eqn:3/2:u:starting}
|\CG|\ge \eps n^2, \qquad \sum_{(i,j)\in \CG} (w_{i}-w_{j})^{2} =\Theta_{\eps}(n).
\end{equation}

\underline{\bf Intermediate $|t|$.}  
We first focus on the range
\[
\frac{(\sqrt{A\log n})\Delta}{n^{3/2}} \le |t| \le \frac{\Delta}{n^{1/2}}.
\]

Now, for $\log \big(\frac{2}{\eps}\big)\le k \le \log\Big( \frac{\sqrt{n}}{A \sqrt{\log n}} \Big)+1$ (so that $\frac{2^{k-1}}{\sqrt{n}} \le \frac{1}{A\sqrt{\log n}}$ and $\frac{2^{k}}{\sqrt{n}} \ge \frac{\eps}{2\sqrt{n}}$), let $\CG_{k}$ be the collection of pairs $(i,j)$ for which
\[
D_{k-1} = \frac{2^{k-1}}{\sqrt{n}} < |w_{i}-w_{j}| \le \frac{2^{k}}{\sqrt{n}} = D_{k}.
\]
Then we see that $\CG=\bigcup_{k}\CG_{k}$ and
\[
\sum_{k} D_{k}^{2} |\CG_{k}| =\Theta_{\eps}(n).
\]

Given a fixed pair $(i,j)\in \CG_{k}$, we consider the sum $\sum_{r\in \CR} \|\frac{t n^{1/2}}{\Delta} (w_{i}-w_{j})r\|_{\R/\Z}^{2}$. For $\frac{t n^{1/2}}{\Delta}D_{k}\le \frac{1}{C n}$, we have $|\frac{t n^{1/2}}{\Delta} (w_{i}-w_{j})r|\le \frac{1}{C}<1$, so
\[
\sum_{r\in \CR}\Big\|\frac{t n^{1/2}}{\Delta} (w_{i}-w_{j})r\Big\|^{2}_{\R/\Z} = \sum_{r\in \CR}\Big|\frac{t n^{1/2}}{\Delta} (w_{i}-w_{j})r\Big|^2 =\Theta_{\delta} \Big(\frac{t^{2}n^{4}D^{2}_{k}}{\Delta^{2}}\Big).
\]
On the other hand, if $\frac{tn^{1/2}}{\Delta}D_{k}> \frac{1}{Cn}$, then $\frac{1}{2Cn}\le |\frac{tn^{1/2}}{\Delta}(w_{i}-w_{j})|\le \frac{1}{A\sqrt{\log n}}$, and by Corollary \ref{cor:wraparound}, we get
\[
\sum_{r\in \CR} \Big\|\frac{tn^{1/2}}{\Delta} (w_{i}-w_{j})r\Big\|_{\R/\Z}^{2} =\Theta_{C}(n).
\]
From the above discussion, we conclude that
\[
\sum_{i,j,r} \Big\|\frac{tn^{1/2}}{\Delta} (w_{i}-w_{j})r\Big\|_{\R/\Z}^{2}
=\Omega_{\delta,C}(1)\cdot
\Big(
\sum_{k: \, \frac{|t|n^{1/2}}{\Delta}D_{k} \le \frac{1}{Cn}} \frac{t^2 n^{4} D_{k}^{2}}{\Delta^2} |\CG_{k}|
+ \sum_{k: \, \frac{|t|n^{1/2}}{\Delta}D_{k}>\frac{1}{Cn}} n |\CG_{k}|\Big).
\]

To lower bound the right-hand side, we distinguish two cases.

\medskip
\noindent\textbf{Case 1:} $\sum_{k: \, \frac{|t|n^{1/2}}{\Delta}D_{k} \le \frac{1}{Cn}} D_{k}^{2} |\CG_{k}| =\Theta_{\eps}(n)$.

Since $t^2 \ge \frac{(A\log n)\Delta^{2}}{n^{3}}$, we have
\[
\sum_{k: \, \frac{|t|n^{1/2}}{\Delta}D_{k} \le \frac{1}{Cn}} \frac{t^2 n^{4} D_{k}^{2}}{\Delta^{2}} |\CG_{k}|
=\Omega_{\eps}(An^{2}\log n).
\]

\medskip
\noindent\textbf{Case 2:} $\sum_{k:\, \frac{|t|n^{1/2}}{\Delta}D_{k} > \frac{1}{Cn}} D_{k}^{2} |\CG_{k}| =\Theta_{\eps}(n)$.

As $D^{2}_{k}\le 1/(A^2\log n)$, we obtain $\sum_{k:\, \frac{|t|n^{1/2}}{\Delta}D_{k} > \frac{1}{Cn}} |\CG_{k}|= \Omega_{\eps}(A^2 n \log n)$,
which implies
\[
\sum_{k:\, \frac{|t|n^{1/2}}{\Delta}D_{k} > \frac{1}{Cn}} n |\CG_{k}|
=\Omega_{\eps}(A^2 n^{2}\log n).
\]

Therefore, in both cases,
\[
\sum_{i,j,r} \|t n^{1/2} (w_{i}-w_{j})r\|_{\R/\Z}^{2} =\Omega_{\eps,\delta,C}(A n^{2} \log n).
\]
Assuming $A$ is sufficient large relative to $\eps,\delta$, and $C$, this implies that for $\frac{ \sqrt{A\log n}}{n^{3/2}} \le |t| \le \frac{1}{n^{1/2}}$, we have
\[
\frac{c_{\delta}}{n^{2}}\sum_{i,j,r} \| t (w_{i} - w_{j}) r \|_{\R/\Z}^{2} \ge 2\sqrt{A} \log n, \qquad |\varphi_{S'}(t)| \le n^{-2\sqrt{A}}.
\]

\underline{\bf Large $|t|$.} We now focus on the range 
$$\frac{\Delta}{n^{1/2}} \le |t|\le 1.$$  

Let $\CG_{0}$ be the set obtained from $\CG$ by removing pairs $(i,j)$ with $|w_{i} - w_{j}| > 2/\sqrt{\eps n}$. Since $\sum_{i,j}(w_{i}-w_{j})^2=2n$, at most $(\eps/2)n^2$ pairs were removed. From \eqref{eqn:3/2:u:starting}, we see that $|\CG_{0}|\ge |\CG|-(\eps/2)n^2\ge (\eps/2)n^2$. For every $(i,j)\in \CG_{0}$, as $\frac{\eps}{2\sqrt{n}} \le |w_{i} - w_{j}| \le \frac{2}{\sqrt{\eps n}}$, we have
\[
\frac{\eps}{2\sqrt{n}}\le \Big|\frac{t n^{1/2}}{\Delta} (w_{i}-w_{j})\Big|\le \frac{2}{\Delta \sqrt{\eps}}.
\]
Assuming $\Delta$ is sufficiently large in terms of $\eps$ and $\delta$, Corollary \ref{cor:wraparound} then gives
\[
\frac{c_{\delta}}{n^{2}}
\sum_{\substack{(i,j)\in \CG_{0}\\ r\in \CR}}
\Big\| \frac{t n^{1/2}}{\Delta} (w_i-w_j) r \Big\|_{\R/\Z}^{2}
= \Theta_{\delta}\!\Big(\frac{|\CG_{0}|}{n}\Big)
= \Theta_{\delta,\varepsilon}(n).
\]

\underline{\bf Small $|t|$.}  
It remains to consider 
$$|t| \le \frac{(\sqrt{A\log n})\Delta}{n^{3/2}}.$$

In this case, as $|w_{i}- w_{j}| \le 1/(A \sqrt{\log n})$, we have
\[
\Big|\frac{t n^{1/2}}{\Delta} (w_{i}-w_{j})r\Big|\le \frac{1}{\sqrt{A}}<1.
\]
Therefore,
\[
\frac{c_{\delta}}{n^{2}} \sum_{i,j,r} \Big\|\frac{t n^{1/2}}{\Delta} (w_{i}-w_{j})r\Big\|_{\R/\Z}^{2}
= \frac{c_{\delta}}{n^{2}} \sum_{i,j,r} \Big|\frac{t n^{1/2}}{\Delta}(w_{i}-w_{j})r\Big|^{2}
=\Theta_{\delta,\Delta}(t^{2} n^{3}).
\]
Hence,
\[
\int_{|t|\le \frac{(\sqrt{A\log n})\Delta}{n^{3/2}}} \exp\Big\{- \frac{c_{\delta}}{n^{2}} \sum_{\substack{1\le i,j \le n \\r \in \CR}} \Big\| \frac{t n^{1/2}}{\Delta} (w_{i}-w_{j}) r\Big\|_{\R/\Z}^{2}\Big\}\Big)
\le \int_{\R} e^{-\Theta(t^{2} n^{3})}\, dt
= O(n^{-3/2}).
\]
\end{proof}

\begin{remark}\label{rmk:optimal:3/2}
Consider the case where $w_1 = \dots = w_{(1-t)n} =0$ and $w_{(1-t)n+1} = \dots = w_n = \pm 1/\sqrt{ \delta n}$, which satisfies the condition of Theorem~\ref{thm:cont:1}. In this setting, one can show that the weighted sum $\sqrt{ \delta n} \sum_{i} w_i \pi(i)$ spreads rather evenly over the interval $[-C t^{1/2} n^{3/2}, Ct^{1/2}n^{3/2}]$ for some constant $C > 0$.\footnote{This follows since the variance is of order $t n^3$.} By the pigeonhole principle, $\sup_{x} \P(\sqrt{\delta n} \sum_{i} w_i \pi(i)=x)= \Omega(t^{-1/2}n^{-3/2})$. Hence for any $\delta>0$
\[
\sup_{x} \P(| \sum_{i} w_i \pi(i)/n-x| \le \delta)  =\Omega(t^{-1/2}n^{-3/2}).
\]
\end{remark}

In what follows, we turn to the proof of the $L$-dependent estimate in this finer scale.  
Our general method is similar to the proof of \eqref{eqn:3/2:L}, although the details are slightly different.

\begin{proof}[Proof of \eqref{eqn:3/2:L} of Theorem \ref{thm:cont:3/2}]
Define
\[
S' = \sum_{i} n^{3/2-\eps} w_i v_{\pi(i)} = n^{1/2-\eps} S
\quad \text{and} \quad
X' = S' - n^{3/2-\eps} L.
\]
These random variables correspond to those defined in the proof of~\eqref{eqn:3/2:u}, obtained by setting $\Delta = n^{\varepsilon}$. It suffices to assume $L = O(\sqrt{\log n})$.  

We first observe that our treatment of the intermediate- and large-$|t|$ regimes above, namely
\[
|\varphi_{S'}(t)| \le n^{-2\sqrt{A}} \quad \text{for} \quad \frac{\sqrt{A\log n}}{n^{3/2-\eps}} \le t \le 1,
\]
in fact extends all the way to $|t| \le \sqrt{A\log n}$.

\underline{\bf Very large $|t|$.}  We now assume that
\[
1 \le |t| \le \sqrt{A\log n}.
\]

Similarly to the treatment of the large-$|t|$ regime in the proof of \eqref{eqn:3/2:u}, there exists a set $\CG_{0}$ containing at least $(\eps/2)n^{2}$ pairs $(i,j)$ for which $\frac{\eps}{2\sqrt{n}} \le |w_{i} - w_{j}| \le \frac{2}{\sqrt{\eps n}}$. For each such pair, we have
\[
\frac{\eps}{2 n^{\eps}}\le |t n^{1/2-\eps} (w_{i}-w_{j})|\le \frac{2\sqrt{A\log n}}{\sqrt{\eps}n^{\eps}}.
\]
It then follows from Corollary \ref{cor:wraparound} that
\[
\frac{c_{\delta}}{n^{2}} \sum_{\substack{(i,j) \in \CG_{0}\\ r \in \CR}} 
\Big\| t n^{1/2-\eps} (w_{i}-w_{j}) r \Big\|_{\R/\Z}^{2} 
=\Theta_{\delta}\Big(\frac{|\CG_{0}|}{n}\Big) =\Theta_{\delta,\eps}(n).
\]
As such, in the case $1 \le |t| \le \sqrt{A\log n}$, we also have
\[
|\varphi_{S'}(t)| \le e^{-\Omega(n)}.
\]

\medskip
\noindent
Our next step is similar to the proof of \eqref{eqn:1:L} of Theorem \ref{thm:cont:1}, so we will be brief.  
Starting from
\[
\int_{t \in \R} e^{-\pi t^{2}} e^{i t x} \, d\xi = e^{-\pi x^{2}/2},
\]
and assuming that $\P(|X| \le 1) \le 2e^{\pi /2} n^{-2\sqrt{A}}$, we arrive at
\[
\P(|X'|\le 1) \le 2 e^{\pi/2} \int_{|t| \le \sqrt{A\log n}} e^{-\pi t^{2}} \varphi_{X'}(t) \, dt.
\]
Since we have shown $|\varphi_{X'}(t)| \le n^{-2\sqrt{A}}$ for $\frac{\sqrt{A\log n}}{n^{3/2-\eps}} \le |t| \le \sqrt{A\log n}$, it follows that
\[
\P(|X'|\le 1) \le 2 e^{\pi} \int_{|t| \le \frac{\sqrt{A\log n}}{n^{3/2-\eps}}} e^{-\pi t^{2}} \varphi_{X'}(t) \, dt + (2\sqrt{A\log n})n^{-2\sqrt{A}}.
\]
Similarly to the proof of Theorem \ref{thm:cont:1}, we decompose the integral on the right-hand side into
\[
\int_{|t| \le L/n^{3/2-\eps}} e^{-\pi t^{2}} \varphi_{X'}(t) \, dt
\;+\;
\int_{L/n^{3/2-\eps} < |t| \le (\sqrt{A\log n})/n^{3/2-\eps}} e^{-\pi t^{2}} \varphi_{X'}(t) \, dt.
\]

The treatment of the second integral is similar to the method used for small-$|t|$ regime (i.e., $|t| \le \frac{\sqrt{A \log n}}{n^{3/2}}$) in the proof of \eqref{eqn:3/2:u}, where we obtained a bound of the type
\[
\frac{1}{n^{3/2}}\int_{L \le |x|} \exp(-c x^{2}) \, dx
= \frac{\exp(-cL^{2})}{n^{3/2}}.
\]

For the first integral, by the change of variables $x = n^{3/2-\eps}t$, we can rewrite it as
\[
\frac{1}{n^{3/2-\eps}}\int_{|x| \le L} e^{-\pi x^{2}} \varphi_{X'/n^{3/2-\eps}}(x) \, dx,
\]
where
\[
\varphi_{X'/n^{3/2-\eps}}(x) = \mathbb{E} \, e^{i x (S'/n^{3/2-\eps} - L)}.
\]
If we extend this to the complex plane, then, as $S'$ is bounded, we obtain a holomorphic function
\[
\varphi_{S'/n^{3/2-\eps}}(z) = \mathbb{E} \, e^{i z S'/n^{3/2-\eps}}.
\]

Let
\[
h(t) = e^{-\pi t^{2}/n^{3/2-\eps}} \, \varphi_{S'/n^{3/2-\eps}}(t),
\]
which extends holomorphically to
\[
h(z) = e^{-\pi z^{2}/n^{3/2-\eps}} \, \varphi_{S'/n^{3/2-\eps}}(z).
\]
By \eqref{eqn:mgf:1},
\begin{equation}\label{eqn:CS:3/2}
|\varphi_{S'}(z/n^{3/2-\eps})| = |\varphi_{S}(z/n)| \le C_{0}' e^{C_{0}' s^{2}}.
\end{equation}
Next, by using contour integration, we pass to the line $\R + icL$:
\begin{align}\label{eqn:contour:3/2}
\P(|X'|\le 1)
&\le \frac{1}{n^{3/2-\eps}} \int_{|t| \le L} e^{-i t L} h(t) \, dt \\
&= \Re \Big( \int_{\substack{z\in \R - icL \\ |\Re(z)| \le L}} e^{-i z L} h(z) \, dz \Big) \nonumber \\
&= \frac{1}{n^{3/2-\eps}} \Re \Big( \int_{|t| \le L} e^{-i(t-icL)L} h(t - icL) \, dt \Big) \nonumber \\
&= \frac{e^{-cL^{2}}}{n^{3/2-\eps}} \, \Re \Big( \int_{|t| \le L} e^{itL} h(t - icL) \, dt \Big). \nonumber
\end{align}
By \eqref{eqn:CS:3/2},
\[
|\varphi_{S'/n^{3/2-\eps}}(t - icL)| = O\Big(e^{C_{0}' c^{2} L^{2}}\Big).
\]

Putting this together, by choosing $c = 1/(4C_{0}')$, we obtain the bound
\[
\P(|X'|\le 1) \le \frac{e^{-cL^{2}}}{n^{3/2-\eps}} \cdot e^{C_{0}' c^{2} L^{2}} \cdot 2A L
= O\Big( \frac{e^{-\Theta(L^{2})}}{n^{3/2-\eps}} \Big).\qedhere
\]
\end{proof}

To conclude this section, we further refine the scaling under
an additional---yet still quite generic---condition on the $w_{i}$, as stated below. The proof will be given in Appendix \ref{sec:proof-5/2}.

\begin{theorem}\label{thm:cont:5/2}
Let $0 < \eps < 1/2$ and $\delta > 0$ be given.  
Suppose that the sequence $(w_1, \ldots, w_n)$ satisfies Condition~\ref{cond:separation} for some sufficiently large constant $A > 0$.  
In addition, assume that no interval of length $\eps / \sqrt{n}$ contains more than $(1 - \eps)n$ of the values $w_i$, and that there are at least $A n \log n$ pairs $(i, j)$ for which  
\[
\frac{1}{n^{3/2}} \le |w_i - w_j| \le \frac{1}{n^{3/2 - \eps}}.
\]
Then, for every $L \in \R$, we have  
\[
\P \Big( \Big|\sum_i w_i \, \frac{\pi(i)}{n} - L \Big| \le \frac{1}{n^{5/2 - \eps}} \Big) = O\Big( \frac{1}{n^{5/2 - \eps}} \Big).
\]
\end{theorem}

We can also obtain an $L$-dependent bound, as in Theorems~\ref{thm:cont:1} and \ref{thm:cont:3/2}. This result can be seen as a continuous analog of Theorem \ref{thm:discrete:5/2}. Here, roughly speaking, if the $w_{i}$ are spread out evenly over the interval $[-C/\sqrt{n},\, C/\sqrt{n}]$, then the average consecutive spacing is $1/n^{3/2}$. The above condition requires that most of the consecutive spacings asymptotically attain this bound.

\section{Continuous setting: proof of Theorem \ref{thm:cont:poly:1} for the polynomial sequences}\label{sect:poly}

We will use the method 
in a similar way to that in Theorem \ref{thm:cont:1}. 

\begin{proof}(of \eqref{eqn:poly:u} Theorem \ref{thm:cont:poly:1}) From Esseen's estimate and Corollary \ref{cor:Roos}, we can write
\[
\sup_{L} \P(|n\sum_{i} w_{i} v_{\pi(i)} -Ln| \le 1) =O\Big(\int_{|t| \le 1} |\varphi(t)|\, dt\Big)=O\Big(\int_{|t| \le 1} e^{- \frac{1}{n^{3}} \sum_{i,j,k,l} \|t n (w_{i}-w_{j}) (v_{k}-v_{l})\|_{\R/\Z}^{2}}\Big).
\]
Here we recall that $v_i = P_d(i)/n^{d}$ for  $i \in I$, where $P_{d}(i) = b i^{d} + b_{n}' i^{d-1} +b_{n}'' i^{d-2}+ \dots$ is a real polynomial of degree $d$ with fixed leading coefficient $b$ such that 
\begin{equation}\label{eqn:bound:P_2}
|v_{i}| \le  B \quad \text{for all} \quad i \in I.
\end{equation}
The exponent of the 
right-hand side can be bounded from below by
\begin{equation}\label{eqn:int:1/2}
\frac{1}{n^{3}} \sum_{1\le i,j\le n; \, k,l \in I} \|t (w_{i}-w_{j})(P_{d}(k) - P_{d}(l))/n^{d-1}\|_{\R/\Z}^{2}.
\end{equation}

As in the proof of Theorem \ref{thm:cont:1}, we will break down the integral depending on whether $|t|\le (\sqrt{A\log n})/n$ or  $(\sqrt{A\log n})/n \le  |t| \le 1$.

\underline{\bf  Large $|t|$.} We assume now that 
$$\frac{\sqrt{A\log n}}{n} \le |t| \le 1.$$

 We
 first throw away those $|w_{i}-w_{j}|$ that are smaller than $1/\sqrt{n}$, and set 
$$\CG=\{(i,j): |w_{i}-w_{j}| \ge 1/\sqrt{n}\}.$$
Then we have 
$$\sum_{(i,j)\in \CG} (w_{i}-w_{j})^{2} \asymp n.$$
Now for $0 \le k \le \log\Big( \frac{\sqrt{n}}{A \sqrt{\log n}} \Big)+1$, we let $\CG_{k}$ be the collection of pairs $(i,j)$ for which
$$ D_{k-1}=\frac{2^{k-1}}{\sqrt{n}} <  |w_{i}-w_{j}| \le \frac{2^{k}}{\sqrt{n}} = D_{k}.$$
Then we see that 
\[
\sum_{k} D_{k}^{2} |\CG_{k}| =\Theta(n).
\]

We use the following corollary of Lemma \ref{lemma:wraparound:square} to estimate the expression in \eqref{eqn:int:1/2}.\footnote{Part \textbf{(ii)} of Corollary \ref{cor:bigt:square} follows from Lemma \ref{lemma:wraparound:square} in the same way that Corollary \ref{cor:wraparound} follows from Lemma \ref{lemma:wraparound}.}

\begin{corollary}\label{cor:bigt:square} Let $r_{0} \in [n]$ be fixed.
\begin{itemize}
\item[\textbf{(i)}] We have 
\[
\sum_{r\in I} (P_{d}(r) - P_{d}(r_{0}))^2=\Theta_{C}(n^{2d+1}).
\]
\item[\textbf{(ii)}] For $ D_{k-1} \le |w|  \le D_{k}$ and $\frac{1}{C n D_{k}} \le  |t| \le \frac{n^{d-1}}{CD_{k}}$, we have
$$ \sum_{r\in I} \|t w (P_{d}(r) - P_{d}(r_{0}))/n^{d-1}\|_{\R/\Z}^{2} =\Theta_{C}(n).$$
\end{itemize}
\end{corollary}

As a consequence, for 
$$\frac{1}{C n D_{k}} \le  |t| \le 1$$ 
we obtain 
\[
\frac{1}{n^{3}} \sum_{(i,j) \in \CG_{k}} \sum_{r,r_{0}\in I} \|t (w_{i}-w_{j})(P_{d}(r) - P_{d}(r_{0}))/n^{d-1}\|_{\R/\Z}^{2} =\Theta_{C}\Big(\frac{|\CG_{k}|}{n}\Big).
\]
Consider 
\[
|t|\le \frac{1}{C n D_{k}}.
\]
In this case, for $(i,j) \in \CG_{k}$, using \eqref{eqn:bound:P_2}
$$\|t (w_{i}-w_{j}) (P_{d}(r) - P_{d}(r_{0}))/n^{d-1}\|_{\R/\Z} = |t (w_{i}-w_{j}) (P_{d}(r) - P_{d}(r_{0}))/n^{d-1}|.$$
Thus
\[
\frac{1}{n^{3}} \sum_{(i,j) \in \CG_{k}} \sum_{r,r_{0}\in I} \|t (w_{i}-w_{j})(P_{d}(r) - P_{d}(r_{0}))/n^{d-1} \|_{\R/\Z}^{2}=\frac{1}{n^{3}} \sum_{(i,j) \in \CG_{k}} \sum_{r,r_{0}\in I} \big| t (w_{i}-w_{j}) (P_{d}(r) - P_{d}(r_{0}))/n^{d-1}\big|^{2}
\]
\[
=\Theta\Big(\frac{1}{n^{3}} t^{2} D_{k}^{2} |\CG_{k}| n^{4}\Big)=   \Theta(t^{2} nD_{k}^{2} |\CG_{k}|).
\]

Putting together, 
\[
\frac{1}{n^{3}} \sum_{(i,j) \in \CG} \sum_{r,r_{0}\in I} \|t (w_{i}-w_{j})(P_{d}(r) - P_{d}(r_{0}))/n \|_{\R/\Z}^{2} =\Omega(1)\cdot \Big(\sum_{k:\, |t|D_{k}\le 1/(Cn)}   t^{2} n D_{k}^{2} |\CG_{k}| + \frac{1}{n}\sum_{k:\, |t|D_{k}> 1/(Cn)}  |\CG_{k}|\Big).
\]

{\bf Case 1:} $\sum_{k:\, |t|D_{k}\le 1/(Cn)} D_{k}^{2} |\CG_{k}|\ge (1/10)n$. 

Then as $t^2 \ge \frac{A\log n}{n^2}$  
\[
\sum_{k:\, |t|D_{k}\le 1/(Cn)}  t^2  n D_{k}^{2} |\CG_{k}| =\Omega(A \log n).
\]

{\bf Case 2:}  $\sum_{k:\, |t|D_{k}> 1/(Cn)} D_{k}^{2} |\CG_{k}|> (9/10)n$. 

Then as $D^{2}_{k}\le 1/A^2\log n$, we have that
\[
\sum_{k:\, |t|D_{k}\ge 1/(Cn)} |\CG_{k}|=\Omega(A^2n \log n).
\]
We thus conclude that in the case of large $|t|$, similarly to \eqref{eqn:phiS:larget}
\begin{equation}\label{eqn:phiS:larget:poly} 
|\varphi_{S}(t)| \le n^{-2\sqrt{A}},
\end{equation}
provided that $A$ is sufficiently large with respect to $C$.

\underline{\bf Small $|t|$.}  It remains to focus on 
\[
|t| \le \frac{\sqrt{A\log n}}{n}.
\]
As $|w_{i}- w_{j}| \le 1/(A \sqrt{\log n})$ and $|P_{d}(r)/n^{d-1}|\le Bn$, if we choose $A$ sufficiently large,
$$\|t (w_{i}-w_{j}) (P_{d}(r) - P_{d}(r_{0}))/n^{d-1}\|_{\R/\Z} = |t (w_{i}-w_{j}) (P_{d}(r) - P_{d}(r_{0}))/n^{d-1}|.$$
Then 
\begin{align*}
\frac{1}{n^{3}} \sum_{1\le i,j \le n; \, r,r_{0} \in I} \|t (w_{i}-w_{j}) (P_{d}(r) - P_{d}(r_{0}))/n^{d-1}\|_{\R/\Z}^{2}&= \frac{1}{n^{3}} \sum_{1\le i,j \le n; \, r,r_{0} \in I} |t (w_{i}-w_{j}) (P_{d}(r) - P_{d}(r_{0}))/n^{d-1}|^{2}\\ & =\Theta(t^{2}n^{2}),
\end{align*}
where we used the fact that $\sum_{1\le i,j \le n}(w_{i}-w_{j})^2=2n$, and that $\sum_{r,r_{0} \in I} \big(P_{d}(r) - P_{d}(r_{0})\big)^{2} =\Theta(n^{2d+2})$.
We then
have
\[
\int_{|t| \le (\sqrt{A\log n})/n} \exp(-\Theta(t^{2} n^{2}))dt \le \frac{1}{n}\int_{\R} \exp(-\Theta(x^{2})dx = O(1/n).
\]
\end{proof}

\begin{proof}(of \eqref{eqn:poly:L} of Theorem \ref{thm:cont:poly:1}) We begin by considering the case where $|t|$ is very large, and show that the characteristic function is very small in this regime.

\underline{\bf Very Large $|t|$.} We assume now that 
$$1 \le |t| \le  \sqrt{A\log n}.$$

We
again throw away those $|w_{i}-w_{j}|$ that are smaller than $1/\sqrt{n}$. 
Let 
$$\CG=\{(i,j): |w_{i}-w_{j}| \ge 1/\sqrt{n}\}.$$
Then we have 
$$\sum_{(i,j)\in \CG} (w_{i}-w_{j})^{2} \ge n.$$
It thus follows that, as $(w_{i}-w_{j})^2 \le 1/A^2\log n$ 
$$|\CG| \ge A^{2} n \log n.$$
Now consider the sum $\sum_{(i,j)\in \CG; \, r,r_{0} \in I} \|t (w_{i}-w_{j}) (P_{d}(r) - P_{d}(r_{0}))/n^{d-1}\|_{\R/\Z}^{2}$. 
As $\frac{1}{\sqrt{n}} \le |w_{i}-w_{j}| \le \frac{1}{A\sqrt{\log n}}$ and $1 \le |t| \le  \sqrt{A\log n}$, we have $\frac{1}{\sqrt{n}}  \le |t(w_{i}-w_{j})| \le \frac{1}{\sqrt{A}}\le \frac{1}{C}$.
As such, Corollary \ref{cor:bigt:square} implies that 
\[
\sum_{r\in I} \|t (w_{i}-w_{j}) (P_{d}(r) - P_{d}(r_{0}))/n^{d-1}\|_{\R/\Z}^{2} = \Theta(n).
\]
Hence we have 
\[
\sum_{1\le i,j\le n; \, r,r_{0} \in I} \|t (w_{i}-w_{j}) (P_{d}(r) - P_{d}(r_{0}))/n^{d-1}\|_{\R/\Z}^{2}  = \Omega(n^{2} |\CG|)= \Omega(A^2n^{3}\log n).
\]
As such, in the case that 
$1\le |t|\le \sqrt{A\log n}$ we also have
\begin{equation}\label{eqn:L:large:t:quad}
|\varphi_{S}(t)| \le n^{-A},
\end{equation}
provided that $A$ is sufficiently large. 

The rest of the proof is almost identical to that of \eqref{eqn:1:L} and of \eqref{eqn:3/2:L}, and hence we omit the details. 
 \end{proof}

\section{Proof of Theorem \ref{thm:comparison:2vs1} for the joint distributions}\label{sect:joint}
We will justify for the case $d=2$ only, the case $d\ge 3$ can be treated similarly. We restate the result below (after scaling up by a factor of $n$).

\begin{theorem}[Smoothness of the joint vector, $d=2$]\label{thm:comparison:2vs1:d=2}
Let $b, B>0$ be constants. Suppose that the sequence $(w_{1}, \ldots, w_{n})$ satisfies Condition~\ref{cond:separation} for some sufficiently large constant $A > 0$. Let $I\subset [n]$ be any subset with $|I|\ge \delta n$, and consider the sequences $(v_{1},\dots,v_{n})$ and  $(v'_{1},\dots,v'_{n})$ partially defined by
\[
v_{i}=\frac{b i^2 + b_{n}' i + b_{n}''}{n^{2}} \quad \text{and} \quad v'_{i}=\frac{i}{n} \quad \text{for all} \quad i\in I,
\]
such that
\[
|v_{i}|,|v'_{i}|\le B \quad \text{for all} \quad i\in I.
\]
Then, for any given $L_1, L_2\in \R$, we have:
\begin{itemize}
\item (Uniform bound)
\begin{equation}\label{eqn:comparison:2D:u}
\P\Big( \Big| n\sum_{i=1}^{n} v_{i} w_{\pi(i)} - L_1 n \Big| \le 1 
\;\wedge\; 
\Big| n\sum_{i=1}^{n} v'_i w_{\pi(i)} - L_2 n \Big| \le 1 \Big) 
= O_{A}\Big( \frac{1}{n^2} \Big),
\end{equation}
\end{itemize}
If, additionally $|v_{i}|,|v'_{i}|\le \widetilde{B}$ for all $i\in [n]$, for some constant $\widetilde{B}>0$, then we have:
\begin{itemize}
\item ($L$-dependent bound)
\begin{equation}\label{eqn:comparison:2D:L}
\P\Big( \Big|n \sum_{i=1}^{n} v_{i} w_{\pi(i)} - L_1 n \Big| \le 1 
\;\wedge\; 
\Big| n\sum_{i=1}^{n} v'_{i} w_{\pi(i)} - L_2 n \Big| \le 1 \Big) 
= O\Big( \frac{1}{n^2} e^{-\Theta(L_1^2 + L_2^2)} \Big).
\end{equation}
\end{itemize}
Here the implied constants are allowed to depend on $A$ and $\widetilde{B}$.
\end{theorem}

Let 
\[
S_1 =n \sum_{i=1}^{n} v_{i} w_{\pi(i)}, \quad 
S_2 =n \sum_{i=1}^{n} v'_{i} w_{\pi(i)},
\]
and define the vector $\BMS = (S_1, S_2)$. We are interested in the event
\[
\Big| \frac{S_1}{n} - L_1 \Big| \le \frac{1}{n} \quad \text{and} \quad \Big| \frac{S_2}{n} - L_2 \Big| \le \frac{1}{n},
\]
or equivalently,
\[
|S_1 - nL_1| \le 1 \quad \text{and} \quad |S_2 - nL_2| \le 1.
\]
For convenience, we will also let $\BMX = (X_1, X_2) := (S_1 - nL_1, S_2 - nL_2)$. 

\begin{proof}(of Equation~\eqref{eqn:comparison:2D:u}) 
We consider the characteristic function of $\BMX$: for any $\BMt=(t_{1},t_{2}) \in \R^{2}$,
\[
\varphi_{\BMX}(\BMt) = \E e^{i (t_{1}X_{1} + t_{2}X_{2})} 
= \E e^{i t_{1}(S_{1} - L_{1}n) + i t_{2} (S_{2} - L_{2}n)} = \E e^{i (t_{1}S_{1} + t_{2} S_{2})} e^{{-it_{1} L_{1}n - i t_{2}L_{2}n}}.
\]
We first establish the following estimate.
\begin{lemma}[large $\|\BMt\|_{2}$]\label{lemma:quad:1}
For $\BMt=(t_{1},t_{2})$ such that $(\sqrt{A\log n})/n \le \|\BMt\|_{2} \le 1$ we have
\[
|\varphi_{\BMX}(\BMt)| \le n^{-2\sqrt{A}}.
\]
\end{lemma}

We remark that this generalizes \eqref{eqn:phiS:larget} of Theorem \ref{thm:cont:1} (where $t_{2}=0$) and \eqref{eqn:phiS:larget:poly} of Theorem \ref{thm:cont:poly:1} (where $t_{1}=0$).
 
\begin{proof}(of Lemma \ref{lemma:quad:1}) Given $v_{k} = t_{1}\frac{b k^2 + b_{n}' k + b_{n}''}{n^{2}} + t_{2}  \frac{k}{n}$ for $k\in I$, we start with
\[
|\varphi_{\BMX}(\BMt)|  \le \exp\Big(- \frac{1}{2n^{3}} \sum_{i,j,k,l} \|n (w_{i}-w_{j}) (v_{k}-v_{l})\|_{\R/\Z}^{2}\Big).
\]
It boils down to bound from below the following
\[
\frac{1}{n^{3}} \sum_{1\le i,j\le n; \, r,s \in I} \big\| t_{1} (w_{i}-w_{j})\frac{b r^2 + b_{n}' r-b s^2 - b_{n}' s}{n} + t_{2}(w_{i}-w_{j}) (r - s) \big\|_{\R/\Z}^{2}.
\]


\medskip
\noindent\textbf{Case 1.}
Assume that
\[
\frac{\sqrt{A\log n}}{2n} \le |t_{1}| \le 1.
\]
Our goal is to use Lemma \ref{lemma:wraparound:square}, but for this, we will first need to simplify the term involving $t_{2}$. We use the following claim. 

\begin{claim} Let $w_{i}, w_{j}$ be fixed. Then there exists a subset $I_{0} \subset [n]$, and for each $s\in I_{0}$, a corresponding subset $J_{s} \subset \{-n,\dots,n\}$, such that all of these sets have size $\Theta_{\delta}(n)$, and that
\begin{align*}
& \sum_{r,s \in I} \big\| t_{1} (w_{i}-w_{j})\frac{b r^2 + b_{n}' r-b s^2 - b_{n}' s}{n} + t_{2}(w_{i}-w_{j}) (r - s) \big\|_{\R/\Z}^{2} \\
& \ge \sum_{s\in I_{0}} \sum_{h\in J_{s}} \big\| t_{1} (w_{i}-w_{j})\frac{b (h+2s)h + b_{n}'h}{n} + t_{2}(w_{i}-w_{j})h \big\|_{\R/\Z}^{2}.
\end{align*}
\end{claim}

\begin{proof}
By Fact \ref{fact:R}, there exists a subset $\CR\subset \{-n,\dots,n\}$ with $|\CR|=\Theta_{\delta}(n)$, such that for every $h \in \CR $, there are $ \Theta_{\delta}(n) $ pairs $r, s\in I$ with $ r - s = h $. For each $s\in I$, let $J_{s}$ denote the set of $h\in \CR$ such that $s+h\in I$. By double counting the pairs $r,s\in I$ with $r-s\in \CR$, we deduce that $|J_{s}|=\Theta_{\delta}(n)$ for $\Theta_{\delta}(n)$ many $s\in I$. Denote the set of such $s$ by $I_{0}$.

The sum is bounded from below by
$$\sum_{h\in \CR}\sum_{r,s \in I: \, r-s=h} \big\| t_{1} (w_{i}-w_{j})\frac{b r^2 + b_{n}' r-b s^2 - b_{n}' s}{n} + t_{2}(w_{i}-w_{j}) (r - s) \big\|_{\R/\Z}^{2},$$
which simplifies to
$$\sum_{s\in I}\sum_{h\in J_{s}} \big\| t_{1} (w_{i}-w_{j})\frac{b (h+2s)h + b_{n}'h}{n} + t_{2}(w_{i}-w_{j})h \big\|_{\R/\Z}^{2},$$
by the substitution $r=s+h$. Restricting the outer sum to $s\in I_{0}$, we obtain the desired inequality.
\end{proof}

To complete the treatment in this case, we just proceed as how we proved \eqref{eqn:phiS:larget:poly} for each of the sum $\sum_{h\in J_{s}} \big\| t_{1} (w_{i}-w_{j})\frac{b (h+2s)h + b_{n}'h}{n} + t_{2}(w_{i}-w_{j})h \big\|_{\R/\Z}^{2}$, using Lemma \ref{lemma:wraparound:square}.

\medskip

\noindent\textbf{Case 2.}
Assume that $|t_{1}| < \frac{\sqrt{A\log n}}{2n}$. Then since $ \frac{\sqrt{A\log n}}{n} \le \|\BMt\|_{2} \le 1$, we must have
\[
\frac{\sqrt{A\log n}}{2n} \le |t_{2}| \le 1.
\]
As in the proof of \eqref{eqn:phiS:larget}, we consider $\CG$ to be the collection of pairs $i,j$ where $|w_{i}-w_{j}|\ge 1/2\sqrt{n}$. First notice that because $|t_{1}| < \frac{C\sqrt{\log n}}{2n}$ and $|w_{i} - w_{j}| \le \frac{1}{A \sqrt{\log n}}$, we have
\begin{equation}\label{eqn:C/A} 
|t_{1} (w_{i}-w_{j})\frac{b r^2 + b_{n}' r -b s^2 - b_{n}' s}{n}| \le \frac{2}{\sqrt{A}}.
\end{equation}
For any fixed $s$, applying Corollary \ref{cor:wraparound} with $b = t_{2} (w_{i}-w_{j})$ and  
$b_{0}=-t_{2}(w_{i}-w_{j})s$ and assuming that $|t_{2}| D_{k} \ge 1/Cn$, gives 
\[
\sum_{r} \| t_{2}(w_{i}-w_{j}) (r - s)\|_{\R/\Z}^{2} =\Theta_{C}(n),
\]
and, together with \eqref{eqn:C/A}, we obtain
\[
\sum_{r} \big\| t_{1} (w_{i}-w_{j}) \frac{b r^2 + b_{n}' r -b s^2 - b_{n}' s}{n}   + t_{2}(w_{i}-w_{j}) (r - s) \big\|_{\R/\Z}^{2} =\Theta_{C}(n).
\]
Now if $|t_{2}| D_{k} < 1/Cn$ , then 
$$\| t_{1} (w_{i}-w_{j}) \frac{b r^2 + b_{n}' r -b s^2 - b_{n}' s}{n}  + t_{2}(w_{i}-w_{j}) (r - s) \big\|_{\R/\Z}^{2} =(w_{i}-w_{j})^{2} |t_{1} \frac{b r^2 + b_{n}' r -b s^2 - b_{n}' s}{n} + t_{2}(r - s)|^{2}.$$
To finish, we use the following fact.
\begin{claim}\label{claim:span}
For any fixed $s \in I$, the $2$-dimensional vectors $(\frac{b r^2 + b_{n}' r -b s^2 - b_{n}' s}{n^{2}},\frac{r-s}{n}), r\in I$ completely span $\R^{2}$ in the sense that for any unit vector $(t_{1},t_{2})$ we have 
\[
\sum_{r\in I} |(t_{1},t_{2}) \cdot (\frac{b r^2 + b_{n}' r -b s^2 - b_{n}' s}{n^{2}},  \frac{r-s}{n})  |^{2} =\Theta(n).
\]
\end{claim}
\begin{proof} We only restrict to $r$ for which $r-s$ has order $n$. We rewrite $(\frac{b r^2 + b_{n}' r -b s^2 - b_{n}' s}{n^{2}},  \frac{r-s}{n})$ as $(\frac{b(h+2s)h+b_{n}'h}{n^2},\frac{h}{n})$, where $h:=r-s$. Choose any pairs $h,h'$ of order $n$ (i.e. $|r-s|,|r'-s|$ is of order $n$) so that 
\[
|h-h'|=\Theta(n).
\]
We see that the vectors $(\frac{b(h+2s)h+b_{n}'h}{n^{2}},\frac{h}{n})$ and  $(\frac{b(h'+2s)h'+b_{n}'h'}{n^{2}}, \frac{h'}{n})$ both have norm of order 1, and area of the parallelogram formed by them is  
\[
|\frac{b(h+2s)h+b_{n}'h}{n^{2}}\frac{h'}{n}-\frac{b(h'+2s)h'+b_{n}'h'}{n^{2}}\frac{h}{n}| = |\frac{b(h-h') h h'}{n^{3}}| =\Theta(1).
\]
As such, for any unit vector $(t_{1},t_{2})$
\[
|(t_{1},t_{2}) \cdot (\frac{b(h+2s)h+b_{n}'h}{n^{2}},\frac{h}{n}) |^{2} +  |(t_{1},t_{2}) \cdot (\frac{b(h'+2s)h'+b_{n}'h'}{n^{2}}, \frac{h'}{n})|^{2} =\Theta(1).
\]
To finish the proof, we just choose
$\Theta(n^{2})$ distinct pairs $(r,r')$ from $I^{2}$ satisfying the above properties, and sum up the estimates.
\end{proof}
As a corollary, for each fixed $s$, by the claim
\[
\sum_{r} (w_{i}-w_{j})^{2} |t_{1}\big(\frac{b r^2 + b_{n}' r -b s^2 - b_{n}' s}{n^{2}}\big) + t_{2} (r - s)|^{2} =\Theta\Big((w_{i}-w_{j})^{2} (t_{1}^{2}+t_{2}^{2}) n^{3}\Big).
\]
The rest of the proof from this point on is similar to our proof of  \eqref{eqn:phiS:larget}, we omit the details. This completes our proof of Lemma \ref{lemma:quad:1}.
\end{proof}

\begin{lemma}[Very large $\|\BMt\|_{2}$]\label{lemma:quad:2}
For $\BMt=(t_{1},t_{2})$ such that $1 \le \|\BMt\|_{2} \le  \sqrt{A\log n}$ we have
\[
|\varphi_{\BMX}(\BMt)| \le n^{-2\sqrt{A}}.
\]
\end{lemma}

\begin{proof}(of Lemma \ref{lemma:quad:2})
If $1 \le |t_{1}| \le \sqrt{A\log n}$, then we can argue as in the proof of \eqref{eqn:poly:L} for very large $|t|$.  
If $(\sqrt{A\log n})/n \le |t_{1}| \le 1$, we can also argue as in the proof of \eqref{eqn:poly:u} for large $|t|$. It remains to assume $|t_{1}| \le (\sqrt{A\log n})/n$, in which case $1/2 \le |t_{2}| \le \sqrt{A\log n}$.  

To this end, we recall \eqref{eqn:C/A} that $|t_{1} (w_{i}-w_{j})\frac{b r^2 + b_{n}' r -b s^2 - b_{n}' s}{n}| \le \frac{2}{\sqrt{A}}$. We can then argue as in the proof of \eqref{eqn:poly:L}, again applying Corollary \ref{cor:wraparound}.
\end{proof}

By Lemmas \ref{lemma:quad:1} and \ref{lemma:quad:2}, it suffices to focus on $\|\BMt\|_{2} \le \frac{\sqrt{A\log n}}{n}$. In this case,
\[
|\varphi_{\BMX}(\BMt)| \le \exp\Big(-\frac{1}{2n^{3}} \sum_{1\le i,j\le n; \, r,s\in I} \big| t_{1} (w_{i}-w_{j}) \frac{b r^2 + b_{n}' r -b s^2 - b_{n}' s}{n} + t_{2}(w_{i}-w_{j}) (r - s) \big|^{2} \Big).
\]
Since $\sum_{1\le i,j\le n}(w_{i}-w_{j})^2=2n$, the exponent simplifies to
\[
\frac{1}{n^{2}} \sum_{r,s \in I} \big| t_{1} \frac{b r^2 + b_{n}' r -b s^2 - b_{n}' s}{n} + t_{2} (r - s)  \big|^{2}.
\] 
Again by Claim \ref{claim:span}, 
\[
\frac{1}{n^{2}} \sum_{r,s \in I} \big| t_{1} \frac{b r^2 + b_{n}' r -b s^2 - b_{n}' s}{n} + t_{2} (r - s)  \big|^{2} =\Theta \Big((t_{1}^{2}+t_{2}^{2})n^{2}\Big).
\]
Using this, we are done with the proof of \eqref{eqn:comparison:2D:u}.
\end{proof}

\begin{proof}(of Equation \eqref{eqn:comparison:2D:L})
For simplicity, we replace $b k^{2} + b'_{n} k + b''_{n}$ with $k^{2}$; the argument extends without change to the general case. Our starting point is a two-dimensional variant of \eqref{eqn:identity:1d}:
\begin{equation}\label{eqn:identity:2d}
 \int_{\R^{2}} e^{-\pi\|\BMt\|_{2}^{2}} e^{i \BMt \cdot \BMx} \, d\BMt = e^{-\pi \|\BMx\|_{2}^{2}/2}.
\end{equation}

Hence
\[
\mathbb{E} \int_{\R^{2}} e^{-\pi \|\BMt\|_{2}^{2}} e^{i \BMt \cdot \BMX} \, d\BMt
= \mathbb{E} e^{-\pi \|\BMX\|_{2}^{2}/2}.
\]

For any $K$ (noting here and later that the integrals are real-valued because of the symmetry of the range of $t$),
\[
-\int_{\|\BMt\|_{2}\ge K} e^{-\pi  \|\BMt\|_{2}^{2}}\, d\BMt
 \le \int_{\|\BMt\|_{2}>K} e^{-\pi  \|\BMt\|_{2}^{2}} e^{i \BMt \cdot \Bx} \, d\BMt
 \le \int_{\|\BMt\|_{2} \ge K} e^{-\pi  \|\BMt\|_{2}^{2}} \, d\BMt
 \le e^{-C' K^{2}}.
\]
Thus, for sufficiently large $A$, with
\[
\BMX=\BMS-n\BML = (S_{1}-n L_{1}, \; S_{2} -n L_{2}),
\]
\[
\Big| \mathbb{E} \int_{\|\BMt\|_{2} \ge \sqrt{A\log n}} e^{-\pi  \|\BMt\|_{2}^{2}} e^{i \BMt \cdot \BMX} \, d\BMt \Big|
 \le \int_{\|\BMt\|_{2} \ge \sqrt{A\log n}} e^{-\pi  \|\BMt\|_{2}^{2}} \, d\BMt
 \le n^{-2\sqrt{A}}.
\]
We thus have (note that the integral is real because of the symmetry of the range of $t$)
\begin{align*}
\P(\|\BMX\|_{2}\le 1) &\le e^{\pi} \, \mathbb{E} e^{-\pi \|\BMX\|^{2}/2} \\
&\le e^{\pi} \Big[ \mathbb{E} \int_{\|\BMt\|_{2} \le \sqrt{A\log n}} e^{-\pi  \|\BMt\|_{2}^{2}} e^{i \BMt \cdot \BMX} \, d\BMt + n^{-2\sqrt{A}} \Big] \\
&\le e^{\pi} \int_{\|\BMt\|_{2} \le \sqrt{A\log n}} e^{-\pi  \|\BMt\|_{2}^{2}} \mathbb{E} e^{i \BMt \cdot \BMX} \, d\BMt + e^{\pi} n^{-2\sqrt{A}}.
\end{align*}

Next, if $\P(\|\BMX\|_{2} \le 1)\le 2e^{\pi /2} n^{-2\sqrt{A}}$, then there is nothing to prove. Assume otherwise, then from the above we have
\[
\P(\|\BMX\|_{2}\le 1) \le 2 e^{\pi} \int_{\|\BMt\|_{2} \le \sqrt{A\log n}} e^{-\pi  \|\BMt\|_{2}^{2}} \varphi_{\BMX}(\BMt) \, d\BMt.
\]

It remains to bound the RHS. As we assume that $L_{1}, L_{2} = O(\sqrt{\log n})$, by Lemma \ref{lemma:quad:1} and Lemma \ref{lemma:quad:2}, it suffices to focus on $\|\BMt\|_{2} \le (\sqrt{A\log n})/n$. Without loss of generality \footnote{The case $L_{1} > L_{2}$ can be treated similarly by applying~\eqref{eqn:mgf:1} to the sequence of squares corresponding to the sum $S_{1}$ in~\eqref{eqn:CS:1} instead.}, let us assume
\[
L_{2} \ge L_{1}.
\]
We will establish the bound
\[
\int_{\|\BMt\|_{2} \le (\sqrt{A\log n})/n} e^{-\pi  \|\BMt\|_{2}^{2}} \varphi_{\BMX}(\BMt) \, d\BMt = O\Big( \frac{e^{-c L_{2}^{2}}}{n^{2}} \Big).
\]
We first write
\begin{align*}
\int_{\|\BMt\|_{2} \le (\sqrt{A\log n})/n} e^{-\pi  \|\BMt\|_{2}^{2}} \varphi_{\BMX}(\BMt) \, d\BMt &\quad \le \int_{|t_{1}| \le L_{2}/n, \; |t_{2}| \le L_{2}/n} e^{-\pi  \|\BMt\|_{2}^{2}} \varphi_{\BMX}(\BMt) \, d\BMt \\
&\quad + \int_{\|\BMt\|_{2} \le (\sqrt{A\log n})/n, \; L_{2}/n \le |t_{1}| \le (\sqrt{A\log n})/n} e^{-\pi  \|\BMt\|_{2}^{2}} \varphi_{\BMX}(\BMt) \, d\BMt \\
&\quad + \int_{\|\BMt\|_{2} \le (\sqrt{A\log n})/n, \; L_{2}/n \le |t_{2}| \le (\sqrt{A\log n})/n} e^{-\pi  \|\BMt\|_{2}^{2}} \varphi_{\BMX}(\BMt) \, d\BMt.
\end{align*}

For the second integral, recall that for $\|\BMt\|_{2} \le (\sqrt{A\log n})/n$, we have $|\varphi_{\BMX}(\BMt)| \le \exp(-\Theta(\|\BMt\|_{2}^{2}n^{2}))$, so
\[
\int_{\|\BMt\|_{2} \le (\sqrt{A\log n})/n, \, L_{2}/n\le |t_{1}| \le (\sqrt{A\log n})/n}
 e^{-\pi  \|\BMt\|_{2}^{2}} \varphi_{\BMX}(\BMt) \, d\BMt
 \le \frac{1}{n^{2}}\int_{L_{2}\le |x_{1}|} e^{-\Theta(x_{1}^{2} + x_{2}^{2})} dx_{1}dx_{2}
 \le \frac{1}{n^{2}}e^{-\Theta(L_{2}^{2})}.
\]
The third integral is bounded similarly:
\[
\int_{\|\BMt\|_{2} \le (\sqrt{A\log n})/n, \, L_{2}/n\le |t_{2}| \le (\sqrt{A\log n})/n}
 e^{-\pi  \|\BMt\|_{2}^{2}} \varphi_{\BMX}(\BMt) \, d\BMt
 \le \frac{1}{n^{2}}\int_{L_{2}\le |x_{2}|} e^{-\Theta(x_{1}^{2} + x_{2}^{2})} dx_{1}dx_{2}
 \le \frac{1}{n^{2}}e^{-\Theta(L_{2}^{2})}.
\]

It remains to work with the first integral. By the change of variables $x_{1} = n t_{1}$, $x_{2} = n t_{2}$, it becomes
\[
\frac{1}{n^{2}}\int_{|x_{1}| \le L_{2},\; |x_{2}| \le L_{2}} e^{-\pi  \|\BMx\|_{2}^{2}/n^{2}} \varphi_{\BX/n}(\BMx) \, dx_{2}dx_{1},
\]
where (recalling that $S_{1}= \sum_{k\in I} k^{2}w_{\pi(k)}/n, S_{2}=\sum_{k\in I} k w_{\pi(k)}$)
\[
\varphi_{\BMX/n}(\BMx) = \mathbb{E} e^{i\BMx \cdot \BMX/n}
= \mathbb{E} e^{i x_{1} X_{1} + i x_{2} (S_{2}/n -L_{2})} =  e^{-i x_{2}L_{2}}\mathbb{E} e^{i x_{1} X_{1} + i x_{2} S_{2}/n}.
\]
In what follows we will fix $x_{1} \in [-L_{2},L_{2}]$ and consider only the inner integral with respect to $x_{2}$, 
\begin{align*}
f(x_{1})&=\int_{|x_{2}| \le L_{2}} e^{-\pi  \|\BMx\|_{2}^{2}/n^{2}} (\varphi_{\BMX/n}(x_{1},x_{2})+\varphi_{\BMX/n}(-x_{1},x_{2})) \, dx_{2}\\
&=:\int_{|x_{2}| \le L_{2}} e^{-\pi  \|\BMx\|_{2}^{2}/n^{2}} (\varphi_{\BMX/n,x_{1}}(x_{2})+\varphi_{\BMX/n,-x_{1}}(x_{2})) \, dx_{2}
\end{align*}
To treat with this integral, we first extend $x_{2}$ to complex numbers via
\begin{align*}
\varphi_{\BMX/n,x_{1}}(z)& = e^{-i zL_{2}}\mathbb{E} e^{i x_{1} X_{1} + i z (S_{2}/n)}\\
\varphi_{X/n,-x_{1}}(z)& = e^{-i zL_{2}}\mathbb{E} e^{-i x_{1} X_{1} + i z (S_{2}/n)}.
\end{align*}
For short, let 
\begin{align}\label{eqn:hx_{1}}
h_{x_{1}}(x_{2}) &=  e^{-\pi  (x_{1}^{2} + x_{2}^{2})/n^{2}} (\mathbb{E} e^{i x_{1} X_{1} + i x_{2} (S_{2}/n)} +  \mathbb{E} e^{-i x_{1} X_{1} +i x_{2} (S_{2}/n)})\nonumber \\
 &=e^{-\pi  (x_{1}^{2} + x_{2}^{2})/n^{2}} \mathbb{E} (e^{i x_{1} X_{1}}+ e^{-i x_{1} X_{1}})e^{i x_{2} (S_{2}/n)}.
\end{align}

This function can be extended holomorphically to 
\begin{align*}
h_{x_{1}}(z)&=e^{-\pi  (x_{1}^{2}+ z^{2})/n^{2}}  (\mathbb{E} e^{i x_{1} X_{1} + i z (S_{2}/n)} +  \mathbb{E} e^{-i x_{1} X_{1} +i z (S_{2}/n)}).
\end{align*}
Since $|\mathbb{E} Y| \le \mathbb{E} |Y|$ for any complex-valued random variable $Y$, for $z=t+is$, inequality \eqref{eqn:mgf:1} gives 
\begin{align}\label{eqn:CS:2} 
|\mathbb{E} e^{i x_{1} X_{1} + i (t+is) (S_{2}/n)} +  \mathbb{E} e^{-i x_{1} X_{1} +i (t+is) (S_{2}/n)} | \le 2 \E  e^{{-sS_{2}/n}}  \le 2C_{0}'e^{C_{0}'s^{2}}.
\end{align}
We write
$$f(x_{1})=\int_{|x_{2}| \le L_{2}} e^{-\pi  \|\BMx\|_{2}^{2}/n^{2}} (\varphi_{\BMX/n,x_{1}}(x_{2})+\varphi_{X/n,-x_{1}}(x_{2})) \, dx_{2} =\int_{|x_{2}| \le L_{2}} e^{-i x_{2}L_{2}} h_{x_{1}}(x_{2}) dx_{2}.$$
By using contour integral, we pass to the line $\R-icL$
\begin{align}\label{eqn:contour:joint}
\int_{|x_{2}| \le L_{2}} e^{-i x_{2}L_{2}}h_{x_{1}}(x_{2})dx_{2} &=\Re (\int_{z\in \R-icL_{2}, \, |\Re(z)| \le L_{2}} e^{-i zL_{2}} h_{x_{1}}(z)dz) \nonumber  \\
&= \Re (\int_{|t| \le L_{2}} e^{-i(t-icL_{2})L_{2}} h_{x_{1}}(t-icL_{2})dt) \nonumber \\
&=(e^{-cL_{2}^{2}}) \Re (\int_{|t| \le L_{2}} e^{-itL_{2}} h_{x_{1}}(t-icL_{2})dt),
\end{align}
where it is crucial to notice that the first integral is real-valued because $h_{x_{1}}(-x_{2}) = \overline{{h}_{x_{1}}(x_{2})}$ from \eqref{eqn:hx_{1}} and the real part of the integrals (with opposite orentation) on the lines $\Re(z) = -L_{2}$ and $\Re(z) = L_{2}$ cancel each other. More specifically
$$\Re \int_{z=-L_{2} - it, \, 0\le t\le cL} e^{-i zL_{2}}h_{x_{1}}(z)dt = \Re \int_{z=L_{2} - it, \, 0\le t\le cL_{2}} e^{-i zL_{2}}h_{x_{1}}(z)dt$$
as they are conjugate to each other: this follows from that fact that  $S_{1},S_{2} \in \R$ and $h_{x_{1}}(-x+it) = \overline{{h}_{x_{1}}(x+it)}$, which can be  seen from 
$$h_{x_{1}}(-x+iy) =e^{-\pi  x_{1}^{2}/n^{2}-\pi  (-x+iy)^{2}/n^{2}} \E (e^{i x_{1} X_{1}/n}+ e^{-i x_{1} X_{1}/n} ) e^{i(-x+iy)S_{2}/n}$$
$$ =e^{-\pi (x_{1}^{2}+x^{2}-y^{2})/n^{2}} e^{2\pi i xy/n^{2}} \E (e^{i x_{1} X_{1}/n}+ e^{-i x_{1} X_{1}/n} ) e^{-i x S_{2}/n} e^{-yX_{2}/n},$$ 
while 
$$h_{x_{1}}(x+iy) =e^{-\pi (x_{1}^{2}+x^{2}-y^{2})/n^{2}} e^{-2\pi i xy/n^{2}} \E (e^{i x_{1} X_{1}/n}+ e^{-i x_{1} X_{1}/n} ) e^{i x X_{2}/n} e^{-yS_{2}/n}.$$
To continue \eqref{eqn:contour:joint}, note that
$$|e^{-\pi  (t-icL_{2})^{2}/n^{2}}| = e^{-\pi (t^{2} - c^{2}L_{2}^{2})/n^{2}} \approx 1, \mbox{ as $n\to \infty$ and $L_{2}=O(\sqrt{\log n})$}$$
and so by \eqref{eqn:CS:2},
\begin{align*}
|h_{x_{1}}(t-icL_{2})| &= |e^{-\pi  (t-icL_{2})^{2}/n^{2}}| |\mathbb{E} e^{i x_{1} X_{1} + i (t+icL_{2}) (S_{2}/n)} +  \mathbb{E} e^{-i x_{1} X_{1} +i (t-icL_{2}) (S_{2}/n)} | \\
&= O(e^{C_{0}'c^{2}L_{2}^{2}}).
\end{align*}
Putting together, by choosing $c=1/8C_{0}'$, we have obtained a bound
$$|f(x_{1})| \le (e^{-cL_{2}^{2}}) \times e^{C_{0}'c^{2}L_{2}^{2}} \times 2L_{2} =O(e^{-\Theta(L_{2}^{2})}).$$
All together, in the case $\P(\|\BMX\|_{2} \le 1)\le 2e^{\pi /2} n^{-2\sqrt{A}}$ we have 
$$\P(\|\BMX\|_{2}\le 1)  \le \frac{1}{n^{2}}\int_{|x_{1}|\le  L_{2}} |f(x_{1})| dx_{1} =O\Big(\frac{1}{n^{2}} L_{2} e^{-\Theta(L_{2}^{2})}\Big) = O\Big(\frac{1}{n^{2}}e^{-\Theta(L_{2}^{2})}\Big),$$
completing the proof.
\end{proof}

To conclude this section, we present below a comparison estimate, which will be useful for the next section.

\begin{theorem}\label{thm:comparison:2vs1'}
Let $d \ge 2$ be a fixed integer, and let $\delta > 0$, $b \ne 0$, $c \ne 0$, and $\widetilde{B} > 0$ be constants. Suppose\footnote{Note that $\sum_i w_i$ is not necessarily zero.} that the sequence $(w_1, \dots, w_n)$ satisfies $\sigma(\BMw) = 1$, and that
\[
|w_i - w_j| \le \frac{1}{A \sqrt{\log n}} \quad \text{for all } i, j,
\]
for some sufficiently large constant $A > 0$.
Let $I \subset [n]$ be any subset with $|I| \ge \delta n$, and consider the sequences $(v_1, \dots, v_n)$ and $(v_1', \dots, v_n')$ partially specified by
\[
v_i = \frac{P_d(i)}{n^{d}} \quad \text{and} \quad v_i' = \frac{P_{d-1}(i)}{n^{d-1}} \quad \text{for all } i \in I,
\]
where $P_d$ and $P_{d-1}$ are real polynomials of degrees $d$ and $d - 1$, respectively, with fixed leading coefficients $b$ and $c$, and whose remaining coefficients may depend on $n$.
Assume that
\[
|v_i|,\, |v_i'| \le \widetilde{B} \quad \text{for all } i \in [n], \qquad \text{and} \qquad \Big| \sum_{i=1}^{n} v'_{i} \Big| \le \widetilde{B}\Big| \sum_{i=1}^{n} v_{i} \Big|.
\]
Then,
\begin{equation}
\P\Big( \Big| \sum_{i=1}^n w_i v_{\pi(i)} \Big| \le \frac{1}{n} \Big| \sum_{i=1}^n w_i v_{\pi(i)}' \Big| \Big) = O_{A}\Big( \frac{1}{n} \Big).
\end{equation}
\end{theorem}
Our proof shows that we can actually relax the condition $|\sum_{i}v'_i| =O(|\sum_{i}v_i|)$ to 
$| \overline{\BMv'} \cdot \overline{\BMw} | = O(| \overline{\BMv} \cdot \overline{\BMw}|)$. In the centered case, where $\overline{\BMw} = 0$, the latter condition holds for all $v_i$ and $v'_i$. Note also that a weaker version of this result, in which the right-hand side is 
$O(\frac{\sqrt{\log n}}{n})$, can be handled by a much simpler method 
(by not relying on Theorem~\ref{thm:comparison:2vs1} for the joint event, but instead 
using Theorem~\ref{thm:cont:poly:1} together with Lemma~\ref{lemma:deviation}).  
We leave the details to the reader.

\begin{proof}[Proof of Theorem~\ref{thm:comparison:2vs1'}]
Let
\[
X_1 = \sum_{i=1}^{n} w_i v_{\pi(i)}, \quad X_2 = \sum_{i=1}^{n} w_i v'_{\pi(i)}.
\]
Then, define $a:=\E X_1 = \overline{\BMv'} \cdot \overline{\BMw}$ and $b:=\E X_2 = \overline{\BMv} \cdot \overline{\BMw}$. By the assumption, 
\[
|a| \leq \widetilde{B} |b|.
\]
Let $ I_k = [k-1, k] $ if $ k $ is a positive integer, and $ I_k = [-k, -k+1] $ if $ k $ is a negative integer. We will consider the joint events that $|X_{2}| \in I_{k}$ and $|X_{1}| \in I_{l} / n$ for some integer $l$ with $1 \le |l| \le |k|$. 

By Theorem~\ref{thm:comparison:2vs1}, and by decomposing $I_{k}$ into $n$ intervals of length $1/n$ each, we have 
\[
\P\big(|X_{2}| \in I_{k},\, |X_{1}| \in I_{l}/n \big)
= O\Big(\frac{1}{n} e^{-\Theta((k-a)^{2})} e^{-\Theta((l/n-b)^{2})}\Big).
\]
Summing over $l$ with $1\le |l| \le |k|$ and over $k$ gives the bound
\[
\P\big(|X_{1}| \le |X_{2}|/n\big) = \frac{O(1)}{n}\sum_{1\le |l|\le |k|}e^{-\Theta((k-a)^{2})} e^{-\Theta((l/n-b)^{2})}.
\]
To estimate the above double sum, we consider two regimes: $|k| \ge 2|a|$ and $|k|<2|a|$. Since $(k-a)^2 \ge k^2/4$ for $|k|\ge 2|a|$, we see that 
\begin{align*}
\frac{1}{n}\sum_{1\le |l|\le |k|; \, |k|\ge 2|a|}e^{-\Theta((k-a)^{2})} e^{-\Theta((l/n-b)^{2})} &\le \frac{1}{n}\sum_{1\le |l|\le |k|; \, |k|\ge 2|a|}e^{-\Theta(k^{2})}\\
&\le \frac{1}{n}\sum_{k}2|k|e^{-\Theta(k^{2})}=O\Big(\frac{1}{n}\Big).
\end{align*}
For $1\le |l|\le |k|<2|a|$, we have $|l|/n<(2|a|)/n<(2B'|b|)/n\le |b|/2$, so $(l/n-b)^2\ge b^2/4$. It follows that
\begin{align*}
\frac{1}{n}\sum_{1\le |l|\le |k|<2|a|}e^{-\Theta((k-a)^{2})} e^{-\Theta((l/n-b)^{2})} \le \frac{1}{n}\sum_{1\le |l|\le |k|<2|a|}e^{-\Theta(b^{2})} \le \frac{16a^2e^{-\Theta(b^2)}}{n}= O\Big(\frac{1}{n}\Big)
\end{align*}
for $|a|\le \widetilde{B}|b|$.
\end{proof}

\section{Some generalizations of our results}\label{sect:general}
In this section we discuss a few more generalization of our results. First of all, Condition \eqref{eqn:separation} can be replaced by the following weaker assumption.

\begin{condition}[Non-degeneracy II]\label{cond:separation'} 
Let $\eps>0$, and suppose $A$ is sufficiently large depending on $\eps$.  
A sequence $w_{1},\dots, w_{n}$ is said to be {not too degenerate} if
\begin{equation}\label{eqn:separation'}
\sum_{\substack{i<j \\ |w_i - w_j|/\sigma(\BMw) \le 1/A \sqrt{\log n}}} (w_{i} - w_{j})^{2} \;>\; \eps n \sigma^2(\BMw).
\end{equation}
\end{condition}
In other words, we allow distances of order larger than $\sigma(\BMw)/A\sqrt{\log n}$, but require that the contribution from pairs at smaller distances is not too small relative to the main term $n \sigma^2(\BMw)$.

\begin{theorem}\label{thm:general'} 
All of our results, including Theorem \ref{thm:cont:1}, Theorem \ref{thm:cont:3/2}, Theorem \ref{thm:cont:5/2}, Theorem \ref{thm:cont:poly:1}, Theorem~\ref{thm:comparison:2vs1} and Theorem~\ref{thm:comparison:2vs1'}  extend to $(w_{i})$ satisfying \eqref{eqn:separation'} (with the normalization $\sigma(\BMw)=1$).
\end{theorem}
\begin{proof} Since in the proofs of all these theorems we focused only on pairs $w_i, w_j$ with 
$|w_i - w_j| \le \tfrac{1}{A \sqrt{\log n}}$, 
Condition~\eqref{eqn:separation'} ensures that the contribution from such pairs is significant. 
For example, in the proof of Theorem~\ref{thm:cont:1}, this condition was invoked in Case~1 of the analysis for ``large $|t|$'', in the treatment of ``small $|t|$'', and again in the treatment of ``very large $|t|$''.
\end{proof}

In the remainder of this section, we present several preparatory observations that will serve as useful ingredients for the proof of Theorem~\ref{thm:poly} in Section~\ref{sect:application}.

\begin{lemma}\label{lemma:mass}
Let $w_1,\dots,w_n \in \R$ satisfy $\overline{\BMw}=0$ and $\sigma(\BMw)=1$.  
Fix $\eps \in (0,1]$ and $K \ge \sqrt{2/\eps}$. 
Assume that
\[
\sum_{i: \, |w_i| \le K/\sqrt{n}} w_i^2 \;\ge\; \eps.
\]
Then there exist disjoint subsets $I, J \subset [n]$ with $|I| = |J| \ge \big\lfloor (\eps/32K^2)n \big\rfloor$,
such that 
\[
\sqrt{\frac{\eps}{2n}}\le |w_i - w_j| \le \frac{2K}{\sqrt{n}} 
\quad \text{for all } i \in I, \ j \in J.
\]
\end{lemma}

\begin{proof}(of Lemma \ref{lemma:mass})
Let $S$ denote the set of indices $i\in [n]$ with $|w_i|\le K/\sqrt{n}$. We begin by proving 
\[
\sum_{\substack{i<j \\ i,j\in S}} (w_i-w_j)^2 \;\ge\; \eps n/2.
\]
Indeed,
\[
\sum_{\substack{i<j \\ i,j\in S}} (w_i-w_j)^2
= |S|\sum_{i\in S} w_i^2 - \Big(\sum_{i\in S} w_i\Big)^2
= |S|\sum_{i\in S} w_i^2 - \Big(\sum_{i\in S^{c}} w_i\Big)^2.
\]
We bound the two terms separately. Since $|w_i| > K/\sqrt{n}$ for $i\in S^{c}$,  
\[
|S^{c}| \;\le\; n/K^2.
\]
As $K \ge \sqrt{2/\eps} \ge \sqrt{2}$, we obtain
\[
|S| \ge n-n/K^2 \ge n/2.
\]
Moreover, by Cauchy--Schwarz,
\[
\Big(\sum_{i\in S} w_i\Big)^2
= \Big(\sum_{i\in S^{c}} w_i\Big)^2
\le |S^{c}| \sum_{i\in S^{c}} w_i^2
\le |S^{c}|(1-\eps) 
\;\le\; (1-\eps)n/K^2.
\]
Combining these estimates and using $\sum_{i\in S} w_i^2 \ge \eps$, we obtain
\[
\sum_{\substack{i<j \\ i,j\in S}} (w_i-w_j)^2
\;\ge\; |S|\eps - (1-\eps)n/K^2
\;\ge\; (n-n/K^2)\eps - (1-\eps)n/K^2
\;\ge\; \eps n/2.
\]

\smallskip
We now proceed to construct $I$ and $J$. Without loss of generality, assume $S=[m]$ and $w_1 \ge w_2 \ge \cdots \ge w_m$.  
Let $p := \lfloor (\eps/32K^2)n \rfloor$, and define $I := \{1,2,\dots,p\}$, $J := \{m-p+1,\dots,m\}$.  
It suffices to show that
\[
w_p-w_{m-p+1}\;\ge\; \sqrt{\eps/2n}.
\]
Suppose instead that $w_p - w_{m-p+1} < \sqrt{\eps/2n}$.  
Then $|w_i-w_j| < \sqrt{\eps/n}$ whenever $p+1 \le i < j \le m-p$,  
and $|w_i-w_j| \le 2K/\sqrt{n}$ whenever $i \in I \cup J$, $j \in S$.  
Consequently, 
\[
\sum_{\substack{i<j \\ i,j\in S}} (w_i-w_j)^2
< |I\cup J||S| \, \Big(\frac{2K}{\sqrt{n}}\Big)^2
  + \binom{|S|}{2} \Big(\sqrt{\frac{\eps}{2n}}\Big)^2.
\]
Since $|I\cup J| \le (\delta/16K^2)n$ and $|S|\le n$, this gives
\[
\sum_{\substack{i<j \\ i,j\in S}} (w_i-w_j)^2
< (\eps/16K^2)n^2 \, \Big(\frac{2K}{\sqrt{n}}\Big)^2
   + \binom{n}{2} \Big(\sqrt{\frac{\eps}{n}}\Big)^2
< \eps n/2,
\]
contradicting the earlier bound. The lemma follows.
\end{proof}

\begin{lemma}\label{lemma:m/n} 
Let $w_1,\dots,w_n\in\R$ satisfy $\overline{\BMw}=0$ and $\sigma(\BMw)=1$. 
Suppose there exists $K\ge 1$ such that
\[
\frac{1}{n}\sum_{i=1}^n \big(n w_{i}^{2}\big)^{2} \le K .
\]
For $m\in[n]$, sample an ordered tuple $(w'_{1},\ldots,w'_{m})$ uniformly without replacement from $\{w_{1},\ldots,w_{n}\}$, and set
\[
S_m := \sum_{k=1}^m (w'_k)^2.
\]
Then for any $C>0$,
\[
\P\!\left( \,\Big|S_m-\frac{m}{n}\Big| \;\ge\; C\,\frac{m}{n} \right) \;\le\; \frac{K}{C^{2}m}.
\]
\end{lemma}

\begin{proof}(of Lemma \ref{lemma:m/n})
Let $Y_i := w_i^2$ and note that $\mu := \frac{1}{n}\sum_{i=1}^n Y_i = \frac{1}{n}$, so 
$\mathbb{E}S_m = m\mu = m/n$. The variance of a without-replacement sum is
\[
\mathrm{Var}(S_m) 
= \frac{m(n-m)}{n-1}\,\sigma_Y^2,
\qquad
\sigma_Y^2 := \frac{1}{n}\sum_{i=1}^n (Y_i-\mu)^2.
\]
We bound $\sigma_Y^2$ using the moment condition: 
$$\frac{1}{n}\sum_{i} Y_{i}^{2}  = \frac{1}{n^{2}}\frac{1}{n}\sum_{i=1}^n  (n w_{i}^{2})^{\,2} \le \frac{K}{n^{2}}.$$
Hence
\(
\sigma_Y^2 
= \frac{1}{n}\sum Y_i^2 - \mu^2 
\le Kn^{-2}.
\)
Therefore
\[
\mathrm{Var}(S_m) \le \frac{Km}{n^2}.
\]
Applying Chebyshev inequality gives
\[
\P\Big(\Big|S_m - \frac{m}{n}\Big| \ge C\frac{m}{n}\Big)
\;\le\;
\frac{\mathrm{Var}(S_m)}{(Cm/n)^2} \le \frac{K}{C^{2}m}. \qedhere
\]
\end{proof}

\begin{corollary}\label{cor:sampling:non-degenerate}
Let $w_1,\dots,w_n\in\mathbb R$ satisfy Condition~\ref{cond:shift} for some $K>1$. That is the rescaled squares
$X_i := n\,(w_i - \overline{\BMw})^2/\sigma^{2}(\BMw)$ satisfy the $\ell_2$ moment bound
\[
\frac{1}{n}\sum_{i=1}^n X_i^{\,2} \;\le\; K.
\]
For $m\in[n]$, sample an ordered tuple $(w'_{1},\ldots,w'_{m})$ uniformly without replacement from $\{w_{1},\ldots,w_{n}\}$. Then, with probability at least $1-O_{K}(1/m)$,
\vskip .1in
\begin{itemize}
\item[\rm (i)] $(w'_{i})$ is not too degenerate; 
\vskip .1in
\item[\rm (ii)] $\sup\limits_{x}\#\{i : w'_i = x\} \le cm$, for some constant $c\in (0,1)$ depending only on $K$.
\end{itemize}
\end{corollary}

\begin{proof}(of Corollary \ref{cor:sampling:non-degenerate}) We first deduce from Condition~\ref{cond:shift} that
$$\sum_{|w_{i} - \overline{\BMw}|/\sigma(\BMw) \le  K /\sqrt{n}}  (w_i-\overline{\BMw})^2 \ge (1-1/K) \sigma^{2}(\BMw).$$
Indeed, by assumption
$$\sum_{|w_{i} - \overline{\BMw}|/\sigma(\BMw) > K /\sqrt{n}}  (K \sigma^{}(\BMw)/\sqrt{n})^{2} n(w_i-\overline{\BMw})^2/\sigma^{4}(\BMw)  \le \sum_{|w_{i} - \overline{\BMw}|/\sigma(\BMw) > K /\sqrt{n}}  n(w_i-\overline{\BMw})^4/\sigma^{4}(\BMw)  \le K.$$
So 
$$\sum_{|w_{i} - \overline{\BMw}|/\sigma(\BMw) > K /\sqrt{n}}  (w_i-\overline{\BMw})^2 \le \sigma^{2}(\BMw)/K,$$ 
and hence 
$$\sum_{|w_{i} - \overline{\BMw}|/\sigma(\BMw) \le K /\sqrt{n}}  (w_i-\overline{\BMw})^2 \ge (1-1/K)\sigma^{2}(\BMw).$$

Next, without loss of generality, assume $\overline{\BMw}=0$ and $\sigma(\BMw)=1$. By the above estimate, Lemma \ref{lemma:mass} applies to the sequence $(w_i)$ (with $\eps =1-1/K$) and yields two disjoint subsets $I,J \subset [n]$ with $|I|=|J|=\Omega(n)$ such that $\tfrac{2K}{\sqrt{n}}\ge |w_i - w_j| \ge \sqrt{\tfrac{\eps}{2n}}$
for all $i \in I, \ j \in J$. Then, by Hoeffding's inequality for sampling without replacement (see \cite{Hoeffding1963}), with probability at least $1-\exp(-\Theta(m))$, the sample contains two disjoint subsets $I', J'\subset [m]$, each of size $\Omega(m)$, satisfying 
\[
\sqrt{\frac{\eps}{2n}} \;\le\; |w'_i - w'_j| \;\le\; \frac{2K}{\sqrt{n}}
\quad \text{for all } i\in I',\ j\in J'.
\] 
On the intersection of this event with the concentration event from Lemma~\ref{lemma:m/n}, we see that
\[
\frac{m}{2n}\;\le\;\sigma^{2}(\BMw')\;\le\;\frac{3m}{2n}.
\]
For any $i\in I'$ and $j\in J'$, we have 
$\sqrt{\tfrac{\eps}{2n}} \le |w'_i-w'_j| \le \tfrac{2K}{\sqrt{n}} \le \tfrac{3K\sigma(\BMw')}{\sqrt{m}}$, and hence
\[
\sum_{\substack{i<j \\ |w'_i - w'_j|/\sigma(\BMw') \le 3K/\sqrt{m}}} (w'_{i} - w'_{j})^{2} 
\;\ge\; |I'|\,|J'| \,\Big(\sqrt{\frac{\eps}{2n}}\Big)^{2}
= \Omega \Big(\frac{m^2}{n}\Big)
= \Omega \big(m \,\sigma^{2}(\BMw')\big).
\]
Thus $(w'_i)$ is not too degenerate. Finally, 
\[
\sup_{x}\#\{i : w'_i = x\} \le \min(m-|I'|,m-|J'|) = (1-\Omega(1))m. \qedhere
\]
\end{proof}

\section{Application to random polynomials: proof of Theorem \ref{thm:poly}}\label{sect:application} 

We use Descartes’ rule of signs to relate the number of nonzero critical points to events involving the coefficients. For convenience, we prove Theorem~\ref{thm:poly} under the assumption $d\ge 1$, which we maintain throughout this section. A minor modification of the argument also handles the case $d=0$.

For a real polynomial $Q$ and an interval $I \subset \mathbb{R}$, let $N_I(Q)$ be the number of roots of $Q$ in $I$, counted with multiplicity. For $x \in \mathbb{R}$ and $d \in \mathbb{N}$, the notation $(x)_d$ denotes the falling factorial $x(x - 1)\cdots(x - d + 1)$.

\begin{lemma}\label{lemma:Descartes}
Let $t \ge 2$ be an integer, and let $\pi$ be a uniform permutation 
of $\{1,\ldots,n\}$. The expected number of real roots of $P_{\pi}^{(d)}$ in $\R \setminus \{-1,0,1\}$, counted with multiplicity, is bounded by 
\[
4(t-1) + \sum_{m=2}^{n-d+1}  \P\Big(\Big| \sum_{i=1}^{m} \binom{m-i+t-1}{t-1} a_{i} \Big|<\Big| \sum_{i=1}^{m} \binom{m-i+t-2}{t-2} a_{i} \Big| \Big),
\]
where the sum runs over $4(n-d)$ events with $2 \le m \le n-d+1$, and $ \bm{a}=(a_1,\ldots,a_{n-d+1})$ is one of the following four random vectors
\[
\begin{array}{cc}
((d)_{d}w_{\pi(d)}, \dots, (n)_{d} w_{\pi(n)}), 
& ((d)_{d}w_{\pi(d)}, -(d+1)_{d} w_{\pi(d+1)}, \ldots, (-1)^{n-d} (n)_{d} w_{\pi(n)}), \\[3pt]
((n)_{d} w_{\pi(n)}, \ldots, (d)_{d} w_{\pi(d)}), 
& ((-1)^{n-d}(n)_{d} w_{\pi(n)}, \ldots, -(d+1)_{d} w_{\pi(d+1)}, (d)_{d} w_{\pi(d)}).
\end{array}
\]
\end{lemma}

\begin{proof}(of Lemma \ref{lemma:Descartes}) For notational convenience, define 
\[
Q_{1}(x)=P_{\pi}^{(d)}(x), \quad Q_2(x)=P_{\pi}^{(d)}(-x), \quad Q_3(x)=x^{n-d}P_{\pi}^{(d)}(1/x), \quad Q_4(x)=x^{n-d}P_{\pi}^{(d)}(-1/x).
\]
Each $Q_{i}$ is a real polynomial of degree at most $n-d$, and  
$$
N_{(-1,0)}(Q_{1})=N_{(0,1)}(Q_{2}), \quad N_{(1,\infty)}(Q_{1})=N_{(0,1)}(Q_{3}), \quad N_{(-\infty,-1)}(Q_{1})=N_{(0,1)}(Q_{4}).
$$
Hence, the expected number of roots of $P_{\pi}^{(d)}(x)$ in $\R\setminus\{-1,0,1\}$ equals $\sum_{i=1}^{4}\E N_{(0,1)} (Q_{i})$.

Given $Q\in \{Q_{1},Q_{2},Q_{3},Q_{4}\}$, write $Q(x)=a_1+a_2x+\cdots+a_{n-d+1}x^{n-d}$ and set $F(x) = Q(x)/(1-x)^t$. Then $(a_1,\ldots,a_{n-d+1})$ is one of the four vectors described in Lemma \ref{lemma:Descartes}. 
Clearly, $F$ and $Q$ have the same number of roots in $(0,1)$. In this interval, $F$ admits an absolutely convergent power series expansion:
$$
F(x)=Q(x)\cdot\sum_{k=0}^{\infty} \binom{k+t-1}{t-1} x^{k}= \sum_{m=1}^{\infty} \Big(\sum_{i=1}^{m} \binom{m-i+t-1}{t-1} a_{i} \Big) x^{m-1},
$$
where $a_{i} = 0$ for $i > n-d+1$ by convention. 
By Descartes' rule of signs, the number of roots of $Q$ in $(0,1)$ is at most the number of sign changes in the sequence $c_{m}:=\sum_{i=1}^{m} \binom{m-i+t-1}{t-1} a_{i}$, $m\ge 1$.

If $m \ge n-d+1$, then $c_{m}=\sum_{i=1}^{n-d+1} \binom{m-i+t-1}{t-1} a_{i}$, which is a polynomial in $m$ of degree at most $t-1$. Thus, there are at most $t-1$ sign changes beyond this point. 

Now consider $2\le m \le n-d+1$. If $c_{m-1}$ and $c_{m}$ have different signs, then 
$|c_{m}|<|c_{m}-c_{m-1}|$. Hence, the sign change here requires
\[
\Big| \sum_{i=1}^{m} \binom{m-i+t-1}{t-1} a_{i} \Big|<\Big| \sum_{i=1}^{m} \binom{m-i+t-2}{t-2} a_{i} \Big|.
\]
Therefore, $\E N_{(0,1)}(Q)$ is bounded by
$(t-1)+\sum_{m=2}^{n-d+1} \P(| \sum_{i=1}^{m} \binom{m-i+t-1}{t-1} a_{i}|<| \sum_{i=1}^{m} \binom{m-i+t-2}{t-2} a_{i}|)$.
\end{proof}

\begin{proof}(of Theorem \ref{thm:poly}) It is clear that under the hypotheses of the theorem, the expected number of zero roots (critical points) is of order $O(1)$, $\E(\#\{\text{zeros (critical points) at }0\}\big)=O(1)$. We will treat roots (critical points) at $\pm 1$ and non-zero roots different from $\pm 1$ separately.
\vskip .05in
\textbf{Counting $\pm 1$ roots.} 
We aim to show $\E N_{\{\pm 1\}}(P_{\pi}^{(d)}) = O(1)$. Suppose that $1$ is a root of $P_{\pi}^{(d)}$.
Then
\[
P_{\pi}^{(d)}(1)=\sum_{k=d}^{n} (k)_{d}\, w_{\pi(k)}=0.
\]
To analyze this event, sample an ordered $(n-d+1)$-tuple \((w'_{1},\ldots,w'_{n-d+1})\) uniformly without replacement from the multiset $\{w_1,\ldots,w_n\}$, and let $\sigma$ be a uniform random permutation of $\{1,\ldots,n-d+1\}$. Then $(w'_{\sigma(1)},\ldots,w'_{\sigma(n-d+1)})$ is distributed as $(w_{\pi(d)},\ldots,w_{\pi(n)})$. Consequently,   
\[
\P\Big(P_{\pi}^{(d)}(1)=0\Big)
= \P\Big(\sum_{i=1}^{n-d+1}(i+d-1)_{d} w'_{\sigma(i)} = 0\Big).
\]  
By Corollary~\ref{cor:sampling:non-degenerate}, there exists a constant $c\in (0,1)$ such that the event $\sup_x \bigl|\{\,i : w'_i = x\,\}\bigr| \le c(n-d+1)$ holds with probability at least $1-O(1/n)$. Conditioning on this event, and noting that the coefficients $(d)_{d}, \dots, (n)_{d}$ are distinct for $n\ge d\ge 1$, Theorem~\ref{thm:discrete:3/2} implies  
\[
\P\Big(\sum_{i=1}^{n-d+1}(i+d-1)_{d} w'_{\sigma(i)} = 0\Big) = O\Big(\frac{1}{n^{3/2}}\Big).
\]  
Therefore, we can bound the probability that $1$ is a root of $P_{\pi}^{(d)}$ from above by $O(\frac{1}{n}+\frac{1}{n^{3/2}})=O(\frac{1}{n})$.  
Since the root at $1$ has multiplicity at most $n$, it follows that $\E N_{\{1\}}(P_{\pi}^{(d)}) = O(1)$. By an identical argument, we also have $\E N_{\{-1\}}(P_{\pi}^{(d)}) = O(1)$.

\medskip

\textbf{Counting roots in $\R\setminus \{-1,0,1\}$.} Let $t\ge d+2$ be a fixed integer.\footnote{In the cases where $\bm{a} = \big((d)_d w_{\pi(d)}, \dots, (n)_d w_{\pi(n)}\big)$ or $\bm{a} = \big((n)_d w_{\pi(n)}, \ldots, (d)_d w_{\pi(d)}\big)$, one can take $t = 2$. For the other two cases, our proof does require $t \geq d + 2$.
} For each $2\le m\le n-d+1$, and for $\bm{a}=(a_{1},\dots,a_{n-d+1})$ being one of the four sequences described in Lemma \ref{lemma:Descartes}, let $\CE_{m}$ be the event that
\[
\Big| \sum_{i=1}^{m} \binom{m-i+t-1}{t-1} a_{i} \Big|<\Big| \sum_{i=1}^{m} \binom{m-i+t-2}{t-2} a_{i} \Big|.
\]
We will show that $\P(\CE_{m}) = O(1/m)$ for $m \ge \log n$, which would then lead to 
\begin{equation}\label{eqn:sumE}
\sum_{m=2}^{n-d+1} \P(\CE_{m}) = O(\sum_{m=1}^{\log n} 1 + \sum_{m=\log n}^{n-d+1} \frac{1}{m}) = O(\log n).
\end{equation}
Now we focus on the regime $\log n\le m \le n-d+1$.

We first consider the (easier) case. 

\textit{Case 1:} $\bm{a}=((d)_{d}w_{\pi(d)}, \dots, (n)_{d} w_{\pi(n)})$.

We can rephrase the event $\CE_{m}$ as follows. For $1\le i \le m$, define 
$$v_{i} = \frac{\binom{m-i+t-1}{t-1} (i+d-1)_{d}}{m^{t+d-1}}, \qquad v'_{i} = \frac{\binom{m-i+t-2}{t-2} (i+d-1)_{d}}{m^{t+d-2}}.$$ 
We sample an ordered $(m+1)$-tuple $(w_{1}',\dots, w'_{m})$ uniformly at random from $\{w_{1},\dots, w_{n}\}$, and let $\sigma$ be a random permutation of $\{0,\dots,m\}$. Then
$$\P(\CE_{m})=\P\Big(\Big|\sum_{i=1}^{m} v_{i}w'_{\sigma(i)})\Big|< \Big|\sum_{i=1}^{m} v'_{i}w'_{\sigma(i)}\Big|\Big).$$
We observe that $\binom{m-i+t-1}{t-1} (i+d-1)_{d}$ and $\binom{m-i+t-2}{t-2} (i+d-1)_{d}$ are polynomials in $i$ of degree $t+d-1$ and $t+d-2$, respectively, with leading coefficients $\frac{(-1)^{t-1}}{(t-1)!}$ and $\frac{(-1)^{t-2}}{(t-2)!}$. It is straightforward to see that $|v_{i}|, |v'_{i}|=O_{t,d}(1)$, and that both $|\sum_{i=1}^{m}v_{i}|$ and $|\sum_{i=1}^{m}v'_{i}|$ are of order $\Theta_{t,d}(m)$.

Thus, the sequences $(v_i)$ and $(v'_i)$ satisfy the conditions of Theorem~\ref{thm:comparison:2vs1'}. Moreover, by Corollary~\ref{cor:sampling:non-degenerate}, with probability at least $1 - \Theta(1/m)$, the sequence $(w'_{i})$ is not too degenerate in the sense of Condition \ref{cond:separation'}. Therefore, by applying Theorem~\ref{thm:general'}, we can invoke Theorem~\ref{thm:comparison:2vs1'} to $I = \{1, \dots, m\}$ and the sequence $(w'_i)$, yielding the desired probability bound $O(1/m)$ for $\CE_m$.

We next deal with the (harder) case.

\textit{Case 2:} $\bm{a} = \left( (d)_d w_{\pi(d)}, -(d+1)_d w_{\pi(d+1)}, \ldots, (-1)^{n-d}(n)_d w_{\pi(n)} \right)$.

Define
\[
v_i = \frac{(-1)^{i-1} \binom{m-i+t-1}{t-1} (i+d-1)_d}{m^{t+d-1}}, \qquad v'_i = \frac{(-1)^{i-1} \binom{m-i+t-2}{t-2} (i+d-1)_d}{m^{t+d-2}} \quad \text{for all} \quad 1\le i\le m.
\]
The treatment of this case closely follows that of the previous one, with the key difference being the verification of the condition 
$$|\sum_{i=1}^{m} v'_i| \asymp |\sum_{i=1}^{m} v_i|.$$ 
While this step was straightforward in the previous case, the proof of the present estimate is more delicate, and we postpone its justification to Lemma~\ref{lem:alternating-sum} below.

Finally, the cases $\bm{a} = \left( (n)_d w_{\pi(n)}, \ldots, (d)_d w_{\pi(d)} \right)$ and $\bm{a} = \left( (-1)^{n-d}(n)_d w_{\pi(n)}, \ldots, (d)_d w_{\pi(d)} \right)$ can be handled as in Case 1 and Case 2, respectively. The details are left to the reader. This completes the proof of \eqref{eqn:sumE} (up to Lemma \ref{lem:alternating-sum}), and thereby establishes the theorem.
\end{proof}

We conclude this section with the statement and proof of a technical lemma used in the above proof.

\begin{lemma}\label{lem:alternating-sum}
Let $S_{t,d}(m) = \sum_{i=0}^{m} (-1)^i \binom{m-i+t}{t} (i+d)_d$. For fixed non-negative integers $t$ and $d$,
\[
|S_{t,d}(m)| = \Theta_{t,d}(m^{\max\{t,d\}}).
\]
\end{lemma}
\begin{proof}
We begin by noting that there are two real polynomials, $P^{0}_{t,d}(m)$ and $P^{1}_{t,d}(m)$, each of degree at most $t + d$, with coefficients depending solely on $t$ and $d$, such that $S_{t,d}(m) = P^{i}_{t,d}(m)$ when $m \equiv i \pmod{2}$. 
To prove the lemma, it suffices to show that the polynomials $P^{i}_{t,d}(m)$ have degree exactly $\max\{t,d\}$. We prove this by induction on $\min\{t,d\}$, analyzing the discrete derivative $S_{t,d}(m) - S_{t,d}(m-2)$.

For $t = 0$, we have
\[
S_{0,d}(m) - S_{0,d}(m-2) = \sum_{i=m-1}^{m} (-1)^{i} (i+d)_{d} = (-1)^{m} d (m+1) \cdots (m+d-1).
\]
Thus, for each $i \in \{0,1\}$, we have $P^{i}_{0,d}(m) - P^{i}_{0,d}(m-2) = (-1)^{i} d (m+1) \cdots (m+d-1)$, which is a polynomial of degree $d-1$. It follows that $P^{i}_{0,d}(m)$ has degree $d$.

For $d = 0$, we have
\[
S_{t,0}(m) - S_{t,0}(m-2) = \sum_{i=0}^{1} (-1)^{i} \binom{m-i+t}{t} = \binom{m+t-1}{t-1},
\]
a polynomial of degree $t-1$. Therefore, each $P^{i}_{t,0}(m)$ has degree $t$.

Now we prove the claim for a pair $t,d\ge 1$, assuming that the hypothesis holds for all pairs $t',d'$ with $\min\{t',d'\}<\min\{t,d\}$. We start by expressing
\[
S_{t,d}(m) - S_{t,d}(m-1) = \sum_{i=0}^{m} (-1)^{i} \left[ \binom{m-i+t}{t} - \binom{m-1-i+t}{t} \right] (i+d)_d.
\]
Applying Pascal's identity $\binom{m-i+t}{t} - \binom{m-1-i+t}{t}=\binom{m-i+t-1}{t-1}$, we simplify this to
\[
S_{t,d}(m) - S_{t,d}(m-1)= \sum_{i=0}^{m} (-1)^{i} \binom{m-i+t-1}{t-1} (i+d)_d = S_{t-1,d}(m).
\]
It follows that
\begin{align*}
S_{t,d}(m) - S_{t,d}(m-2) &= S_{t-1,d}(m) + S_{t-1,d}(m-1)\\
&= \sum_{i=0}^{m} (-1)^i \binom{m-i+t}{t} \left[ (i+d)_d - (i-1+d)_d \right].
\end{align*}
Using the identity $(i+d)_d - (i-1+d)_d=d(i+d-1)_{d-1}$ , we obtain
\[
S_{t,d}(m) - S_{t,d}(m-2) = d \sum_{i=0}^{m} (-1)^i \binom{m-i+t}{t} (i+d-1)_{d-1} = d \cdot S_{t-1,d-1}(m).
\]
Thus, for each $i \in \{0,1\}$, we have $P^{i}_{t,d}(m) - P^{i}_{t,d}(m-2) = d \cdot P^{i}_{t-1,d-1}(m)$. By the induction hypothesis, $P^{i}_{t-1,d-1}(m)$ has degree $\max\{t-1,d-1\}$, so $P^{i}_{t,d}(m)$ has degree $\max\{t-1,d-1\} + 1 = \max\{t,d\}$.
\end{proof}

\section{Application to random row-permutation matrices: proof of Theorem \ref{thm:singularity}}\label{sect:singularity}

Without loss of generality, we may assume $0<\eps \le 1/4$. Let $\BMr_{1},\dots,\BMr_{n}$ denote the rows of $Q_{n\times n}$. For each
$1\le m\le n-1$, let $\CE_m$ be the event that $\BMr_{1},\dots,\BMr_{m}$ span
an $m$-dimensional subspace $H_m$, and that $\BMr_{m+1}$ belongs to $H_m$.
We write $Q_{m\times n}$ for the $m\times n$ matrix whose rows are
$\BMr_{1},\dots,\BMr_{m}$.

\begin{lemma}\label{lemma:small:m} For $m \le n-n^{\eps}$ we have 
$$\P(\CE_{m}) \le n^{-\omega(1)}.$$
\end{lemma}
\begin{proof}(of Lemma \ref{lemma:small:m}) Without loss of generality, assume that the first $m$ columns of $Q_{m\times n}$
span its column space. Once the first $m$ entries of $\BMr_{m+1}$ are fixed,
there exist $\alpha_1,\dots,\alpha_m$ such that $\BMr_{m+1}'=\al_{1}\BMr_{1}'+\dots + \al_{m}\BMr_{m}'$, where $\BMr_{i}'=(q_{i1},\dots, q_{im})$. Consequently, 
\[
\BMr_{m+1}''=\al_{1}\BMr_{1}''+\dots + \al_{m}\BMr_{m}'',
\]
where $\BMr_{i}''=(q_{i(m+1)},\dots, q_{in})$.

\begin{claim} Let $m' = n-m \ge n^{\eps}$. With probability at least $1-n^{-\omega(1)}$, the maximum multiplicity
of $\BMr_{m+1}''$ is at most $(1-\eps/2)m'$.
\end{claim}
\begin{proof}
Fix a value $x$, and let $a_x$ be the multiplicity of $x$ in $\BMv$. By assumption
$a_x\le (1-\eps)n$. The probability that $x$ appears more than $(1-\eps/2)m'$
times in $\BMr_{m+1}''$ is
\[
\sum_{t>(1-\eps/2)m'}
\frac{\binom{a_x}{t}\binom{n-a_x}{m'-t}}{\binom{n}{m'}}
\le
\sum_{t>(1-\eps/2)m'} \binom{m'}{t}(1-\eps)^t
\le e^{-\Omega_\eps(m')}.
\]
Taking a union bound over all $x$, the probability that some value appears more
than $(1-\eps/2)m'$ times in $\BMr_{m+1}''$ is at most
$n e^{-\Omega_\eps(m')}=e^{-\Omega_\eps(m')}$.
\end{proof}

\begin{claim}
Suppose $\BMu\in\R^{m'}$ has maximum multiplicity
$m_{\BMu}\le (1-\eps/2)m'$. Then for every $\BMx\in\R^{m'}$,
\[
\P(\pi(\BMu)=\BMx)\le e^{-\Omega_{\eps}(m')}.
\]
\end{claim}

\begin{proof}
If $\pi(\BMu)=\BMx$, then $\BMx$ must be a permutation of $\BMu$, so
$\P(\pi(\BMu)=\BMx)=\frac{\prod_i \mu_i!}{m'!}$, where $\mu_i$ are the multiplicities
of the distinct values of $\BMu$. Since $\sum_i \mu_i=m'$ and
$\mu_i\le (1-\eps/2)m'$ for all $i$, we have
\[
\frac{\prod_i \mu_i!}{m'!}\le
\frac{((1-\eps/2)m')!\,((\eps/2)m')!}{m'!}
= e^{-\Omega_\eps(m')}. \qedhere
\]
\end{proof}

Combining the above we obtain the lemma, noting that $m'\ge n^{\eps}$.
\end{proof}

From now on we assume $m \ge n - n^{\eps}$. Let $\BMw=(w_{1},\dots, w_{n})$ be a unit normal vector of $H_{m}$. We first need to rule out some degenerating properties of this vector. We follow the proof of \cite[Lemma 4.3]{Nguyen}. 

\begin{lemma}[Non-degeneracy of the normal vector]\label{lemma:multiplicity}
There exist positive constants $\delta,c$ depending only on $\eps$ such that,
with probability at least $1-\exp(-cn)$ over the randomness of
$\BMr_{1},\dots,\BMr_{m}$, we have $m_{\BMw}\le (1-\delta)n$.
\end{lemma}

\begin{proof}
For each $0\le n'\le \delta n$, consider the event that $m_{\BMw}=n-n'$. Losing a factor of $\binom{n}{n'}$ in probability, we may assume that
$w_1=\cdots=w_{n-n'}$. 
Consider the matrix $Q'_{m\times (n'+1)}$ whose first column is the sum of the first $n-n'$ columns of $Q_{m\times n}$, and whose remaining columns are the $(n-n'+1),\dots,n$ columns of $Q_{m\times n}$. The vector $\BMh=(w_1,w_{n-n'+1},\dots,w_n)$ is non-zero and orthogonal to all rows of this
matrix, and hence the matrix has rank at most $n'$. Losing another factor of $\binom{m}{n'}$ in probability, we may assume that the
first $n'$ rows span the row space of $Q'_{m\times (n'+1)}$. Then $\BMh$ is determined by the $\sigma$-algebra generated by these rows, and hence $\BMw$ is determined by the $\sigma$-algebra generated by $\BMr_1,\ldots,\BMr_{n'}$.

It remains to bound the probability that $\BMr_i\cdot\BMw=0$ for $n'+1\le i\le m$. 
We can write this as
\[
v_{\pi_i(1)}w_{1}+\cdots+ v_{\pi_i(n-n')}w_{1}
+ v_{\pi_i(n-n'+1)}w_{n-n'+1}+\cdots+ v_{\pi_i(n)}w_{n}=0.
\] 
Note that $n'\ge 1$ since $\sum_{i=1}^n v_i\neq 0$. 
Let $a=w_1$ and $b=w_{n-n'+1}$, so $a\neq b$. Observe that if $u\neq v$, then
$au+bv\neq av+bu$.
For any permutation $\pi$, since $m_{\BMv}\le (1-\eps)n$, we have
$v_{\pi(1)}\neq v_{\pi(n-n'+1)}$ with probability at least $\eps$. 
Conditioning on $\pi(i)$ for $i\neq 1,n-n'+1$, within this event there are two possible assignments of these two coordinates, and at most one makes the sum vanish. Hence $\P(\pi(\BMv) \cdot \BMw =0) \le 1/2$. Thus altogether $\P(\pi(\BMv) \cdot \BMw =0) \le 1-\eps/2$. 

Putting everything together,
\[
\P(m_{\BMw}\ge (1-\delta)n)\le \sum_{n'=1}^{\delta n}\binom{n}{n'} \binom{m}{n'} (1-\eps/2)^{m-n'} \le \exp(-cn),
\]
 provided that $\delta$ is sufficiently small given $\eps$, and $c>0$ is allowed to depend on $\delta$.
\end{proof}

Our next step concerns the rational commensurability of the normal vector; see
\cite[Section~3.4]{Nguyen} and \cite[Theorem~5.2]{TVJAMS}.
Roughly speaking, the commensurability lemma states that if $\BMw$ is a normal vector to a hyperplane spanned by vectors with bounded integer coordinates, and if all but a small fraction of the coordinates of $\BMw$ lie in a low-rank GAP, then in fact \emph{all} coordinates of $\BMw$ lie in a rank-$1$ GAP (that is, they are rationally commensurable), with controlled numerators and denominators.
For our application, we require a generalization in which the assumption of bounded integer coordinates is replaced by the condition that the coordinates lie in a given set of bounded size.

\begin{lemma}[Rational commensurability]\label{lemma:commensurability}
Let $0<\eps<1$ and $A>0$. Suppose all but at most $n^{\eps}$ coordinates of $\BMa=(a_1,\dots,a_n)\in\R^n$ lie in a proper GAP
$P=\{m_1g_1+\cdots+m_rg_r: m_i\in\Z,\ N_i'\le m_i\le N_i\}$ of rank $r=O_{A,\eps}(1)$ and size $n^{O_{A,\eps}(1)}$. Assume further that $\BMa$ is a normal vector to a hyperplane $H\subset\R^n$ spanned by vectors whose coordinates lie in a set $F\subset\R$ with $|F|=n^{O_{A,\eps}(n^{\eps})}$. Then there exists $i_0\in[n]$ such that
\[
\{a_1,\dots,a_n,g_1,\dots,g_r\}\subset \{(p/q)a_{i_0}: p,q\in G\}.
\]
Here
$G=\{\sum_{S\subset F\cup\{1\},\,|S|=r+n^{\eps}} c_S\prod_{x\in S}x:
\ c_S\in\Z,\ |c_S|\le n^{O_{A,\eps}(n^{\eps})}\}$
has size $n^{O_{A,\eps}(n^{2\eps})}$.
\end{lemma}

\begin{proof}
Assume that
$a_{n-n^{\eps}+1},\dots,a_n$ are the exceptional elements that may not belong to $P$.
Consider the map $\Phi:P\to\R^r$ defined by
$\Phi(m_1 g_1 + \cdots + m_r g_r)=(m_1,\ldots,m_r)$.
Since $P$ is proper, \cite[Lemma~3.2]{Nguyen} implies that there exists a proper GAP $Q$
containing $\{a_1,\ldots,a_{n-n^{\eps}}\}$ with $\rank(Q)\le r$ and $|Q|=O_r (n)$,
such that $\Phi(\{a_1,\ldots,a_{n-n^{\eps}}\})$ has full rank in $\R^{\rank(Q)}$.
Replacing $P$ by $Q$ if necessary, we may therefore assume that
$\Phi(\{a_1,\ldots,a_{n-n^{\eps}}\})$ has full rank in $\R^r$.

For each $i\le n-n^{\eps}$ there exist integers $a_{ij}$,
bounded by $n^{O_{A,\eps}(1)}$, such that
\[
a_i = a_{i1}g_1 + \cdots + a_{ir}g_r.
\]
Consider the $n\times (r+n^{\eps})$ matrix $M_{\BMa}$ whose $i$-th column is the vector
$(a_{i1},\dots,a_{ir},0,\dots,0)$ if $i\le n-n^{\eps}$, and
$(0,\dots,0,1,0,\dots,0)$ if $i> n-n^{\eps}$.
Then $M_{\BMa}$ has rank $r+n^{\eps}$.
Moreover,
\[
\BMa^{T}=M_{\BMa}\,\BMb^{T},
\]
where
$\BMb=(g_1,\dots,g_r,a_{n-n^{\eps}+1},\dots,a_n)$.

Next, let $\BMu_1,\dots,\BMu_{n-1}$ be vectors with entries in $F$ orthogonal to $\BMa$.
Form an $n\times n$ matrix $M_{\BMu}$ whose $i$-th row is $\BMu_i$ for $i\le n-1$,
and whose $n$-th row is $e_{i_0}$, the unit vector among the standard basis
$\{\BMe_1,\dots,\BMe_n\}$ that is linearly independent of $\BMu_1,\dots,\BMu_{n-1}$.

By definition,
\[
M_{\BMu}\BMa^{T}=(0,\dots,0,a_{i_0},0,\dots,0)^{T},
\]
and hence
\[
(M_{\BMu}M_{\BMa})\BMb^{T}=(0,\dots,0,a_{i_0},0,\dots,0)^{T}.
\]
This implies
\begin{equation}\label{eq:comm}
(M_{\BMu}M_{\BMa})(\tfrac{1}{a_{i_0}}\BMb)^{T}
=(0,\dots,0,1,0,\dots,0)^{T}.
\end{equation}

Let $M$ be a full-rank submatrix of size $(r+n^{\eps})\times (r+n^{\eps})$
of $M_{\BMu}M_{\BMa}$. Then
\begin{equation}\label{eq:comm2}
M(\tfrac{1}{a_{i_0}}\BMb)^{T}=\BMx,
\end{equation}
where $\BMx$ is the corresponding subvector of
$(0,\dots,0,1,0,\dots,0)$ in \eqref{eq:comm}.

Observe that the entries of $M$ and $\BMx$ are of the form
$\sum_{i=1}^r c_i f_i$, where the $c_i$ are integers bounded by
$n^{O_{C,\varepsilon}(1)}$ and $f_i\in F\cup\{1\}$.
Applying Cramer's rule to \eqref{eq:comm2} and expanding the determinants (which shows that they lie in $G$),
we conclude that each of $g_i/a_{i_0}$ and $a_j/a_{i_0}$
can be written as $p/q$ with $p,q\in G$.
\end{proof}

We now proceed to the final step in the proof of Theorem~\ref{thm:singularity}. 

\begin{lemma}\label{lemma:large:m}
Let $C>0$ be given. For $n-n^{\eps} \le m \le n-1$ we have
\[
\P(\CE_{m}) \le n^{-C}.
\]
\end{lemma}

\begin{proof}(of Lemma \ref{lemma:large:m}) Since $\BMr_1,\ldots,\BMr_m$ span an $m$-dimensional space $H_m$, there exists a nonzero vector
$\BMw=(w_1,\ldots,w_{m+1},0,\ldots,0)$ orthogonal to $H_m$ with $w_i=0$ for $i\ge m+2$.
Conditioning on $\BMr_1,\ldots,\BMr_m$, the probability that the next row
$\BMr_{m+1}=\pi_{m+1}(\BMv)$ lies in $H_m$ is therefore bounded by
\[
\rho = \P(\pi_{m+1}(\BMv) \cdot \BMw = 0).
\]
There is nothing to prove if $\rho\le n^{-C}$, so assume that
\[
\rho \ge n^{-C}.
\]

Since $m_{\BMv} \le (1-\eps)n$ by the assumption and $m_{\BMw}\le (1-\delta)n$ by conditioning on the event of Lemma~\ref{lemma:multiplicity}, we have 
\[
\rho = O(n^{-1/2}).
\]
Hence $2^{-s-1} <\rho \le 2^{-s}$ for some integer $1\le s\le O(\log n)$. Let $\CN_s$ denote the set of such $\BMw$, considered up to nonzero scalar multiples. We will estimate $|\CN_s|$ using Lemma~\ref{lemma:commensurability}.

Let $n'=(\eps/4) n^{\eps}$. 
By Theorem~\ref{thm:ILO:prod}, $(w_i-w_j)(v_k-v_l)$ belongs to a proper symmetric GAP $Q$ of rank $r=O_{C,\eps}(1)$ and size $O_{\eps}(2^{s}n^{-\eps r/2})$ for all but at most
$n'n^{3}$ quadruples $i,j,k,l$. Since $m_{\BMv}\le (1-\eps)n$, we have $v_k-v_l\neq 0$ for at least $\eps n^{2}/2$ pairs $k,l$. 
Hence among these pairs there exists $k_0,l_0$ such that $(w_i-w_j)(v_{k_0}-v_{l_0})\in Q$ for all but at most $2n'n/\eps$ pairs $i,j$. 
It follows that there exists $j_0$ such that $(w_i-w_{j_0})(v_{k_0}-v_{l_0})\in Q$ for all but at most $2n'/\eps \le n^{\eps}$ indices $i$. 
Letting $P=w_{j_0}+\frac{1}{v_{k_0}-v_{l_0}}\cdot Q$, we obtain

\begin{claim}\label{claim:AP}
There exists a proper GAP $P$ of rank $r=O_{C,\eps}(1)$ and size $O_{C,\eps}(2^{s}n^{-\eps r/2})$ such that $w_i\in P$ for all but 
$n^{\eps}$ indices $i\in[n]$.
\end{claim}

We apply Lemma~\ref{lemma:commensurability} to bound $|\CN_s|$. 
There are at most
$\binom{n}{n^{\eps}}(n^{O_{C,\eps}(n^{2\eps})})^{n^{\eps}}
= n^{O_{C,\eps}(n^{3\eps})}$
ways to choose the exceptional coordinates of $\BMw$.
There are only 
$(n^{O_{C,\eps}(n^{2\eps})})^{O_{C,\eps}(1)}
= n^{O_{C,\eps}(n^{2\eps})}$
ways to fix the generalized arithmetic progression $P$. 
Once $P$ is fixed, the remaining coordinates of $\BMw$ can be chosen in at most 
$|P|^{n}\le (2^{s}n^{-\eps/3})^{n}$ ways. 
Hence
\[
|\CN_s|
\le n^{O_{C,\eps}(n^{3\eps})}n^{O_{C,\eps}(n^{2\eps})}(2^{s}n^{-\eps/3})^{n}
= n^{O_{C,\eps}(n^{3\eps})}(2^{s}n^{-\eps/3})^{n}.
\]
Now fix $\BMw\in\CN_s$. The probability that $\BMr_1,\ldots,\BMr_m$
are all orthogonal to $\BMw$ is at most $(2^{-s})^m$. By the union bound,
the probability that $\BMr_1,\ldots,\BMr_m$ are orthogonal to some
$\BMw\in\CN_s$ is at most
\[
n^{O_{C,\eps}(n^{3\eps})}(2^{s}n^{-\eps/3})^{n}(2^{-s})^{m}
= n^{O_{C,\eps}(n^{3\eps})}(2^{s})^{n-m} n^{-\eps n/3}
\le n^{-\eps n/4},
\]
since $2^s\le n^{C}$, $n-m\le n^{\eps}$, and $0<\eps\le 1/4$.
Summing over the $O(\log n)$ possible values of $s$ completes the proof of the lemma.
\end{proof}

\section{Further comments}\label{section:comments} 

Beyond the applications presented in this note, Theorem~\ref{thm:ILO:prod1} and
Theorem~\ref{thm:LCD} also appear to be useful in a variety of counting problems.
For instance, Theorem~\ref{thm:LCD} is a key tool for establishing strong quantitative
invertibility estimates for matrices with fixed row sums and for adjacency matrices of
\(d\)-regular digraphs.

Indeed, in connection with Subsection~\ref{sub:sing}, consider the random \(0/1\) matrix
\(Q_{n,d}\) whose rows are independent vectors containing exactly \(d\) ones, where
\(\min(d,n-d)=\Omega(n)\). As in the proof of Theorem~\ref{thm:singularity}, we
construct \(Q_{n,d}\) row by row. Suppose the first \(n-1\) rows are independent and span
a hyperplane with normal vector \(\BMw=(w_1,\ldots,w_n)\). Conditioned on these rows,
the probability that \(Q_{n,d}\) is singular is
\[
\P(\BMw\cdot \BMv = 0) \;=\; \P(S_{\pi}=0),
\]
where \(\BMv=(v_1,\ldots,v_n)\) is the last row. If the LCD of the pair \((\BMw,\BMv)\)
is large, then Theorem~\ref{thm:LCD} implies that this probability is small. Tran~\cite{Tran}
used this approach to obtain the optimal bound \(\exp(-cn)\) for some constant \(c>0\).
We also note the related work of Jain, Sah, and Sawhney~\cite{JSS22}, where the authors
employed Theorem~\ref{thm:LCD} to give a nearly optimal bound for the same problem.

In general, it is expected that if either \(\BMw\) or \(\BMv\) arises from a random source,
then the LCD is large. We hope to return to this phenomenon for the matrices introduced
in Theorem~\ref{thm:singularity} in a different venue.

\subsection{Further problems} Directly related to our paper, we record below a few further interesting directions.

\begin{itemize}
\item Theorems \ref{thm:discrete:3/2} and \ref{thm:discrete:5/2} address the problem of determining when $\sup_{x}\P(S_{\pi}=x) \ge n^{-3/2}$ or $n^{-5/2}$. Our approach does not, however, yield a seemingly near-optimal inverse result--namely, that if $\sup_{x}\P(S_{\pi}=x)$ has order $n^{-1}$, then most of the $w_{i}$ must be zero.
\vskip 0.1in
\item It would be interesting to remove the $\log n$ factor from Theorem \ref{thm:discrete:5/2}.
\vskip 0.1in
\item While Theorem~\ref{thm:ILO:prod} is almost optimal in terms of the size of $Q$, it is interesting to deduce more structure on the sequences $(w_i)$ and $(v_i)$ separately. Similarly for Theorems \ref{thm:ILO:prod1}
and \ref{thm:LCD}.
\vskip 0.1in
\item While Remark \ref{rmk:optimal} shows that Condition \ref{eqn:separation} is nearly optimal if we rely on Theorem \ref{thm:Roo}, our Theorem \ref{thm:cont:1} may remain valid without this condition (whereas Theorems \ref{thm:cont:3/2} and \ref{thm:cont:5/2} would require additional assumptions on the $w_{i}$, as stated). It is therefore of interest to remove this condition from Theorem \ref{thm:cont:1}.
\vskip 0.1in
\item Similarly, we suspect that Theorem \ref{thm:poly} remains valid without the condition \eqref{eqn:separation}, though this appears to be a difficult problem.
\vskip 0.1in
\item It would be interesting to extend Theorems \ref{thm:cont:1}, \ref{thm:cont:3/2}, \ref{thm:cont:5/2}, and \ref{thm:cont:poly:1} beyond polynomial sequences.
\vskip 0.1in
\item While Theorem \ref{thm:poly} provides an optimal upper bound, it remains unclear under what conditions on $w_{1}, \dots, w_{n}$ the expected number of real roots $\mathbb{E}N_{\R}$ is truly of order $\log n$. 
Relatedly, we still need effective techniques to compute the asymptotics of the number of real roots and critical points of $P_{\pi}$ for natural choices of $w_1,\dots,w_n$. 
For instance, even in the simple case $w_i=i$, the asymptotic behavior of $\mathbb{E}N_{\R}(P_{\pi})$ is still unknown.
\end{itemize}

\appendix

\section{Proof of Theorem \ref{thm:cont:5/2}}\label{sec:proof-5/2}

Let $\CR\subset\{-n, \dots, n\}$ be a set such that $|\CR| = \Theta_{\delta}(n)$, as defined in Fact~\ref{fact:R}. Let $\Delta$ and $A$ be positive constants, chosen sufficiently large with respect to $\delta$ and $\eps$ (for instance, one may take $\Delta=A$).

To prove Theorem \ref{thm:cont:5/2}, it suffices to show $\sup_{x} \P\big(\big|\sum_{i} n^{3/2-\eps}w_{i} \pi(i) - x\big| \le \Delta\big)=O_{\Delta,A}\big( \frac{1}{n^{5/2 - \eps}} \big)$. Using Esseen's estimate together with Corollary \ref{cor:Roos}, we can write
\[
\sup_{x} \P\Big(\Big|\sum_{i} n^{3/2-\eps}w_{i} \pi(i) - x\Big| \le \Delta\Big)  
=O\Big(\int_{|t| \le 1} \exp \Big\{ - \frac{1}{2n^{3}} \sum_{i,j,k,l} \Big\| \frac{t n^{3/2-\eps}}{\Delta}(w_{i}-w_{j})(k-l)\Big\|_{\R/\Z}^{2}\Big\}\, dt\Big).
\]
The right-hand side can be reduced to 
\begin{equation}\label{eqn:int:5/2}
 \int_{|t| \le 1} \exp\Big\{- \frac{c_{\delta}}{n^{2}} \sum_{\substack{1\le i,j \le n \\r \in \CR}} \Big\| \frac{t n^{3/2-\eps}}{\Delta} (w_{i}-w_{j}) r\Big\|_{\R/\Z}^{2}\Big\}.
\end{equation}

Let $\CG$ denote the set of pairs $(i,j)\in [n]^2$ satisfying $|w_{i}-w_{j}| \ge \eps/2\sqrt{n}$.
Then we have 
\begin{equation}\label{eqn:5/2:L:starting}
|\CG|\ge \eps n^2, \qquad \sum_{(i,j)\in \CG} (w_{i}-w_{j})^{2} =\Theta_{\eps}(n).
\end{equation}

We divide our analysis into four cases.

\underline{\bf Intermediate \(|t|\), range 1.}  
Consider
\[
\frac{(\sqrt{A\log n})\Delta}{n^{5/2-\eps}} \le |t| \le \frac{\Delta}{n^{3/2-\eps}}.
\]
We can argue as in the proof of \eqref{eqn:1:u} for large $|t|$ (using \eqref{eqn:5/2:L:starting}), and conclude that
\[
\frac{c_{\delta}}{n^{2}} \sum_{\substack{1\le i,j \le n \\r \in \CR}} \Big\| \frac{t n^{3/2-\eps}}{\Delta} (w_{i}-w_{j}) r\Big\|_{\R/\Z}^{2} \ge 2\sqrt{A}\log n.
\]

\underline{\bf Intermediate \(|t|\), range 2.}  
Now take
\[
\frac{\Delta}{n^{3/2-\eps}} \le |t| \le \frac{1}{n^{1-\eps}}.
\]
Let $\CG_0 \subset \CG$ denote the set of pairs $(i,j)$ such that $|w_i - w_j| \le \frac{2}{\sqrt{\eps n}}$. Then $\CG_0$ contains at least $(\eps/2) n^2$ pairs $(i,j)$ satisfying $\frac{\eps}{2\sqrt{n}} \le |w_i - w_j| \le \frac{2}{\sqrt{\eps n}}$.
For $(i,j)\in \CG_{0}$, we find
\[
\frac{\eps}{2\sqrt{n}}\le \Big|\frac{t n^{3/2-\eps}}{\Delta} (w_{i}-w_{j})\Big|\le \frac{2}{\Delta\sqrt{\eps}}.
\]
Corollary \ref{cor:wraparound} then gives
\[
\frac{c_{\delta}}{n^{2}} \sum_{\substack{(i,j) \in \CG_{0}\\ r \in \CR}} 
\Big\| \frac{t n^{3/2-\eps}}{\Delta} (w_{i}-w_{j}) r \Big\|_{\R/\Z}^{2} 
=\Theta_{\delta}\Big(\frac{|\CG_{0}|}{n}\Big)= \Theta_{\delta,\eps}(n).
\]

\underline{\bf Large \(|t|\).} Consider $$\frac{1}{n^{1-\eps}} \le |t| \le 1.$$

By our assumption, there are \(A n \log n\) pairs \((i,j)\) with $\frac{1}{n^{3/2}} \le |w_{i} - w_{j}| \le \frac{1}{n^{3/2-\eps}}$. For each such pair,
\[
\frac{1}{\Delta n^{1-\eps}}\le \frac{t n^{3/2-\eps}}{\Delta} (w_{i}-w_{j})\le \frac{1}{\Delta}.
\]
Applying Corollary \ref{cor:wraparound} once again, we conclude
\[
\frac{c_{\delta}}{n^{2}} \sum_{\substack{1\le i,j\le n \\ r\in \CR}} 
\Big\| \frac{t n^{3/2-\eps}}{\Delta} (w_{i}-w_{j})r \Big\|_{\R/\Z}^{2} 
= \Omega_{\delta}(A\log n).
\]

\underline{\bf Small \(|t|\).}  
It remains to deal with
$$|t| \le \frac{(\sqrt{A\log n})\Delta}{n^{5/2-\eps}}.$$

Since $|w_{i}- w_{j}| \le 1/(A \sqrt{\log n})$,
\[
\Big| \frac{t n^{3/2-\eps}}{\Delta} (w_{i}-w_{j}) r\Big| \le \frac{1}{\sqrt{A}}<1.
\]
It follows that
\[
\sum_{\substack{1\le i,j \le n \\ r\in \CR}} 
\Big\| \frac{t n^{3/2-\eps}}{\Delta} (w_{i}-w_{j}) r\Big\|^2_{\R/\Z}= \sum_{\substack{1\le i,j \le n \\ r\in \CR}} 
\Big| \frac{t n^{3/2-\eps}}{\Delta} (w_{i}-w_{j}) r\Big|^2
=\Theta_{\delta,\Delta}(t^{2} n^{5-2\eps}).
\]
Therefore
\[
\int_{|t|\le \frac{(\sqrt{A\log n})\Delta}{n^{5/2-\eps}}} \exp\Big\{- \frac{c_{\delta}}{n^{2}} \sum_{\substack{1\le i,j \le n \\r \in \CR}} \Big\| \frac{t n^{3/2-\eps}}{\Delta} (w_{i}-w_{j}) r\Big\|_{\R/\Z}^{2}\Big\}\Big)\le \int_{\R} e^{-\Theta(t^{2} n^{5-2\eps})} \, dt 
= O\Big( \frac{1}{n^{5/2-\eps}} \Big).
\]

\section{Equidistribution of polynomial phases: proof of Lemma \ref{lemma:Weyl:inverse}}\label{sect:Dio}
Here we will follow \cite{Tao} closely with some modifications, that we are now dealing with a subset of positive density of the interval $\{-N,\dots,N\}$ and not with the entire interval.

\begin{lemma}\cite[Lemma 3]{Tao}\label{lemma:equi:1} 
Let $0<\delta<1$, $\eps\le 10^{-2}\delta$, and let $N$ be an integer with $N\ge 2/\delta$. 
Suppose that a real number $\theta$ satisfies $\|n\theta\|_{\R/\Z} \leq \eps$ for all $n\in I$, where $I$ is a subset of size at least $\delta N$ of $Q=\{-N,\dots, N\}$. 
Then there is a natural number $q \leq 2/\delta$ such that
\[
\| q \theta \|_{\R/\Z} \le \frac{3\eps}{\delta N}.
\]
\end{lemma}

\begin{proof} As $|I| \ge \delta N$ and $\|n\theta\|_{\R/\Z} \leq \eps$ for all $n\in I$, we can find $n_1 < n_2$ in $I$ with $\|n_1 \theta \|_{\R/\Z}, \|n_2 \theta \|_{\R/\Z} \leq \eps$ and $n_2-n_1 \leq \frac{2}{\delta}$. By the triangle inequality, we conclude that there exists at least one natural number $q \leq \frac{2}{\delta}$ for which
\[
\| q \theta \|_{\R/\Z} \leq 2\eps.
\]
We take $q$ to be minimal amongst all such natural numbers, then we see that there exists $a$ coprime to $q$ and $|\kappa| \leq 2\eps$ such that
\begin{equation}\label{eq:kappa}
\theta = \frac{a}{q} + \frac{\kappa}{q}.
\end{equation}

If $\kappa=0$ then we are done, so suppose that $\kappa \neq 0$. Suppose that $n < m$ are elements of $I$ such that $\|n\theta \|_{\R/\Z}, \|m\theta \|_{\R/\Z} \leq \eps$ and $m-n \leq \frac{1}{10 \kappa}$. Writing $m-n = qk + r$ for some $0 \leq r < q$, we have
\[
\| (m-n) \theta \|_{\R/\Z} = \Big\| \frac{ra}{q} + (m-n) \frac{\kappa}{q} \Big\|_{\R/\Z} \leq 2\eps.
\]
By hypothesis, $(m-n) \frac{\kappa}{q} \leq \frac{1}{10 q}$; note that as $q \leq 2/\delta$ and $\eps \leq 10^{-2} \delta$ we also have $\eps \leq \frac{1}{10q}$. This implies that $\| \frac{ra}{q} \|_{\R/\Z} < \frac{1}{q}$ and thus $r=0$. We then have
\[
|k \kappa| \leq 2 \eps.
\]

We conclude that for fixed $n \in I$ with $\|n\theta \|_{\R/\Z} \leq \eps$, there are at most $\frac{2\eps}{|\kappa|}$ elements $m$ of $[n, n + \frac{1}{10 |\kappa|}]$ such that $\|m\theta \|_{\R/\Z} \leq \eps$. Iterating this with a greedy algorithm, we see that the number of $n \in I$ with $\|n\theta \|_{\R/\Z} \leq \eps$ is at most 
\[
\Big( \frac{N}{1/10|\kappa|} + 1 \Big) \frac{2\eps}{|\kappa|};
\] 
since $\eps \le 10^{-2} \delta$, this implies that
\[
\delta N \le \frac{3 \eps}{\kappa}
\]
and the claim follows.
\end{proof}

Note that one can give an alternative proof with somewhat implicit constants. Indeed, it is known that if $I \subset Q$ and $|I| \ge \delta |Q|$, then for some sufficiently large $k$ depending on $\delta$, the sumset $J = kI-kI$ contains a symmetric arithmetic progression $Q'=\{-Nd,-(N-1)d,\dots, (N-1)d,Nd\}$ with step $d=O_{\delta}(1)$ and length $2N+1$.  This is an elementary version of the so-called S\'ark\H{o}zy-type theorem in progression, for which much more is known (see, for instance, \cite[Lemma 4.4, 5.5]{SzV} and \cite[Lemma B3]{Taosol}). Next, by the triangle inequality, for each $n\in J$ we have 
\[
\|n \theta\|_{\R/\Z} \le 2k \eps.
\]
It follows in particular that $\|(l d) \theta\|_{\R/\Z} \le 2k \eps$ for all $|l|\le N$, from which we easily deduce
\[
\|d\theta\|_{\R/\Z} \le \frac{2k \eps}{N}.
\]

We now turn to polynomials. The following result is a variant of \cite[Proposition 4]{Tao}, in which we do not assume $I$ to be an interval.

\begin{proposition}\label{prop:Weyl:inverse}
Let $\delta>0$ be a given positive number and $d\ge 1$ be a given natural number. The following holds for sufficiently large $N$. Let $I$ be a subset of size at least $\delta N$ of the interval $Q=\{-N,\dots, N\}$. Let $P(n) = \sum_{i \leq d} \alpha_i n^i$ be a polynomial from $\Z \to \R/\Z$ of degree at most $d$. If
\[
\frac{1}{N} \Big|\sum_{n \in I} e(P(n))\Big| \geq \delta
\]
then there exists a subprogression $Q'$ of $Q$ with $|Q'|= \Omega_d(\delta^{O_d(1)} N)$ such that $P$ varies by at most $\delta$ on $Q'$.
\end{proposition}
Before proving the result, let us deduce Lemma \ref{lemma:Weyl:inverse} (which, in turn, will be used in the induction scheme of the proof of Proposition \ref{prop:Weyl:inverse}).

\begin{proof}(of Lemma \ref{lemma:Weyl:inverse})
To simplify notation, we allow the implied constants to depend on $d$. We may assume that $\delta \leq c$ some sufficiently small constant $c>0$ depending only on $d$, and that $N$ is sufficiently large.

Applying Proposition \ref{prop:Weyl:inverse}, we can find a natural number $q = O(\delta^{-O(1)})$ and an arithmetic subprogression $Q'$ of $Q$ such that $|Q'| = \Omega(\delta^{O(1)}N)$ and $P$ varies by at most $\delta$ on $Q'$. Writing 
\[
Q' = \{ qn+r: n \in I'\}
\]
for some interval $I' \subset Q$ of length $|I|=\Omega(\delta^{O(1)} N)$ and some $0 \leq r < q$, we conclude that the polynomial $n \mapsto P(qn+r)$ varies by at most $\delta$ on $I'$. 

Taking $d^{\text{th}}$ order differences, we see that the $d^{\text{th}}$ coefficient of this polynomial is $O(\delta^{-O(1)} / N^d)$. By the binomial theorem, this implies that $n \mapsto P(qn+r)$ differs by at most $O(\delta)$ on $I'$ from a polynomial of degree at most $d-1$. Iterating this argument, we deduce that the $i^{\text{th}}$ coefficient of $n \mapsto P(qn+r)$ is $O(\delta N^{-i})$ for $i=0,\dots,d$. The claim then follows by inverting the change of variables $n \mapsto qn+r$ (and replacing $q$ by a larger quantity such as $q^d$ if necessary).
\end{proof}

\begin{proof}(of Proposition \ref{prop:Weyl:inverse}) We argue by induction on $d$. The case $d=1$ follows immediately from Lemma \ref{lemma:equi:1}. 
Now suppose that $d \ge 2$, and that the claim has already been proven for $d-1$. From our assumption,
\[
\frac{1}{N^{2}} \sum_{n_{1},n_{2} \in I} e(P(n_{1}) - P(n_{2})) \geq \delta^{2}.
\]
For each $h\in 2Q = \{-2N,\dots, 2N\}$, let $I_{h}\subset Q$ denote the collection of $n\in I$ such that $n+h \in I$. We can rewrite the above as 
$$\frac{1}{N^{2}} \sum_{h} \sum_{n \in I_{h}} e(P(n+h) - P(n)) \ge \delta^{2}.$$
Since $| \sum_{n \in I_{h}} e(P(n+h) - P(n)) |\le 2N$ for every $h$, there must be $\Omega(\delta^{2} N)$ values of $h \in 2Q$ such that
\[
\frac{1}{N} \Big|\sum_{n \in I_h} e(P(n+h) - P(n))\Big| = \Omega(\delta^{2}).
\]
Note that $P(n+h)-P(n)$ is a polynomial of degree at most $d-1$, with leading term $h \alpha_d n^{d-1}$. By the induction hypothesis, namely Lemma \ref{lemma:Weyl:inverse}, it follows that for each such $h$ there exists a natural number $q_h =O(\delta^{-O(1)})$ such that 
\[
\|q_h h \alpha_d \|_{\R/\Z} = O(\delta^{-O(1)} / N^{d-1}).
\] 
Since there are $\Omega(\delta^{2} N)$ choices of $h \in 2Q$ for which this holds, we obtain $\Omega(\delta^{O(1)} N)$ integers $n$ in the interval $[-\delta^{-O(1)} N, \delta^{-O(1)}N]$ such that 
\[
\|n \alpha_d \|_{\R/\Z} = O(\delta^{-O(1)} / N^{d-1}).
\] 
Applying Lemma \ref{lemma:equi:1}, we conclude
\begin{equation}\label{eq:alphad}
\| q \alpha_d \|_{\R/\Z} = O(\delta^{-O(1)}/ N^d).
\end{equation}
Next, we partition $Q$ into arithmetic progressions $Q'$ of spacing $q$ and length comparable to $\delta^{C}N$, for a sufficiently large constant $C$ to be chosen. By the hypothesis and an application of the pigeonhole principle, we have
\[
\frac{1}{|Q'|} \Big|\sum_{n \in I \cap Q'} e(P(n))\Big| \geq \delta
\]
for at least one such progression $Q'$. Assume that $Q'=\{i_{0}, i_{0}+q, \dots, i_{0}+ n'q\}$ where $n'= \Theta(\delta^{C}N)$. On this progression, for each $0\le k\le n'$ we write
\[
\al_{d}(i_{0}+kq)^{d} = \al_{d} (kq)^{d} +R(k),
\]
where $R(k)$ is a polynomial of degree $d-1$ in $k$. Since $k\le n' = \Theta(\delta^{C}N)$ and $kq \le N$, \eqref{eq:alphad} implies
\[
|\al_{d} (kq)^{d}| = |k (kq)^{d-1}(q\al_{d})|= O(\delta^{C -O(1)}).
\]
Thus for $n \in Q'$ we may write 
\[
P(n) = R(n) + O(\delta^{C-O(1)})
\] 
for some polynomial $R$ of degree at most $d-1$. Choosing $C$ sufficiently large and applying the triangle inequality, we obtain
\[
\frac{1}{|Q'|} \Big|\sum_{n \in I \cap Q'} e(R(n))\Big|=\Omega(\delta).
\]
By the induction hypothesis, there exists a subprogression $Q''$ of $Q'$ of size $|Q''|=\Omega (\delta^{O(1)} N)$ on which $R$ varies by at most $\delta/2$. Taking $C$ sufficiently large, it follows that $P$ varies by at most $\delta$ on $Q''$.
\end{proof}

\section{Proof of Theorem \ref{thm:LCD} }\label{sect:LCD}

Let $\varphi(t):=\E e^{itS_{\pi}}$ denote the characteristic function of $S_{\pi}$.
By Esseen’s inequality and Corollary~\ref{cor:Roos}, we have
\[
\sup_{x\in\R}\P\big(|S_{\pi}-x|\le \delta\big)
= O\Big(\int_{|t|\le 1}\big|\varphi(t/\delta)\big|\,dt\Big)
=
O\Big(\int_{|t|\le 1}
\exp\Big(
-\frac{1}{2n^{3}}
\sum_{i,j,k,l}
\Big\|
\frac{t}{\delta}(w_i-w_j)(v_k-v_l)
\Big\|_{\R/\Z}^{2}
\Big)\,dt\Big).
\]
Recall that $\BMu\in\R^{n^{4}}$ is the vector whose $(i,j,k,l)$-th coordinate is
\[
(w_i-w_j)(v_k-v_l), \qquad 1\le i,j,k,l\le n.
\]
With this notation, the exponent may be rewritten as
\[
-\frac{1}{2n^{3}}\,
\dist^{2}\Big(\frac{t}{\delta}\BMu,\Z^{n^{4}}\Big).
\]
Since $1/\delta \le \LCD_{\gamma,\kappa}(\BMw,\BMv)$, the definition of
$\LCD_{\gamma,\kappa}(\BMw,\BMv)$ implies that for any $t \in [-1,1]$,
\[
\dist\Big(\frac{t}{\delta}\BMu,\Z^{n^4}\Big)
\ge
\min\Big\{\gamma \Big\|\frac{t}{\delta}\BMu\Big\|_{2},\, \kappa\Big\}
\ge
\min\Big\{\gamma n^{3/2}\frac{|t|}{\delta},\, \kappa\Big\},
\]
provided that $\|\BMu\|_{2} \ge n^{3/2}$.
Therefore,
\[
\sup_{x\in\R}\P\big(|S_{\pi}-x|\le \delta\big)
=O\Big(
\int_{|t|\le 1}
\Big(
\exp\Big(-\frac12\Big(\frac{\gamma t}{\delta}\Big)^2\Big)
+ \exp\Big(-\frac{\kappa^{2}}{2n^{3}}\Big)
\Big)\,dt \Big)
=O\Big(
\frac{\delta}{\gamma}+e^{-\kappa^{2}/2n^{3}}\Big).
\]

\end{document}